\documentclass{imanum}
\usepackage{graphicx}
\usepackage{lipsum}
\usepackage{amsfonts}
\usepackage{graphicx}
\usepackage{epstopdf}
\usepackage{algorithmic}
\usepackage{algorithm}
 \usepackage{multirow}
 \usepackage{booktabs}
 \usepackage{amsopn}
 \usepackage{enumerate}
\usepackage{bbm}
\usepackage{color}

\newcommand{\lk}{\left(}
\newcommand{\rk}{\right)}
\newcommand{\R}{\mathbb{R}}
\newcommand{\N}{\mathbb{N}}

\newcommand{\E}{\mathbb{E}}

\newcommand{\del}{\delta_{ij}^{(m)}}
\newcommand{\m}{^{(m)}}

\jno{drnxxx}

\begin{document}

\title{Regularising linear inverse problems under unknown non-Gaussian white noise allowing repeated measurements}
\shorttitle{Linear inverse problems under unknown white noise}

\author{%
{\sc
Bastian Harrach\thanks{Email: harrach@math.uni-frankfurt.de}
and
Tim Jahn\thanks{Corresponding author. Email: jahn@math.uni-frankfurt.de}}\\[2pt]
Insitute of Mathematics, Goethe-University Frankfurt, Frankfurt am Main, Germany\\[6pt]
{\sc and}\\[6pt]
{\sc Roland Potthast}\thanks{Email: Roland.Potthast@dwd.de}\\[2pt]
German Weather Service, Offenbach am Main, Germany
}
\shortauthorlist{T. Jahn \emph{et al.}}

\maketitle

\begin{abstract}
{We deal with the solution of a generic linear inverse problem in the Hilbert space setting. The exact right hand side is unknown and only accessible through discretised measurements corrupted by white noise with unknown arbitrary distribution. The measuring process can be repeated, which allows to reduce and estimate the measurement error through averaging. We show  convergence against the true solution of the infinite-dimensional problem for a priori and a posteriori regularisation schemes as the number of measurements and the dimension of the discretisation tend to infinity under natural and easily verifiable conditions for the discretisation.}
{statistical inverse problems; discretisation; white noise; discrepancy principle.}
\end{abstract}

\section{Introduction and Prelimiaries}\label{intro}

We consider a compact linear operator $K:\mathcal{X}\to\mathcal{Y}$ between Hilbert spaces. The goal is to solve the ill-posed equation $K\hat{x}=\hat{y}$ for a given $\hat{y}\in \mathcal{D}(K^+)$, where $K^+$ is the generalised inverse and the right hand side $\hat{y}$ is ad hoc unknown and has to be reconstructed from measurements. Solving the problem then typically requires specific a priori information about the noise. Here, our key assumption will be that we are able to perform multiple measurements and we do not require any other specific assumption for the error distribution of one measurement. Measuring the same quantity repeatedly is a standard engineering practice to decrease the measurement error known as 'signal averaging' and was extensively studied in \cite{harrach2020beyond} and \cite{jahn2021modified} in the context of infinite-dimensional inverse problems with (strongly $L^2$-bounded) unknown noise. In this article we take discretisation into account and generalise the error distribution further to arbitrary unknown white noise.

 As an arbitrary element of an infinite-dimensional space $\hat{y}$ cannot be measured directly, but we may measure $l(\hat{y})$ for various linear functionals $l\in\mathcal{L}(\mathcal{Y},\R)$. If the unknown $\hat{y}$ is for example a continuous function, one may think of performing point evaluations or measuring the integrals of that function over small parts of the domain. We will refer to these linear functionals as measurement channels in the following. We assume that we have multiple and unbiased samples on each measurement channel corrupted randomly by additive noise. So, 

\begin{equation}\label{errmod}
Y_{ij}:=l_j(\hat{y})+\delta_{ij}
\end{equation}

is the $i$-th sample on the $j$-th measurement channel, with $\|l_1\|=\|l_2\|=...$ and 
 unbiased and independent measurement errors $\delta_{ij}$, $i,j\in\N$ with arbitrary unknown distribution. Thus $$\begin{pmatrix} Y_{i1} - l_1(\hat{y}) \\ ... \\ Y_{im} - l_m(\hat{y}) \end{pmatrix}_{i\in\N} \subset \R^m$$ are i.i.d white noise vectors with unknown distribution.  We assume that $(l_j)_{j\in \N}$ is complete and square-summable, i.e. for all $y\in\mathcal{Y} \setminus\{0\}$ there exists a $l_j$ with $l_j(y)\neq 0$ and  $\sum_{j=1}^\infty l_j(y)^2<\infty$. For a fixed number $m$ of measurement channels and a large number $n$ of repetitions we obtain an approximation

\begin{equation*}
\bar{Y}_n\m:=
\begin{pmatrix}
\frac{1}{n}\sum_{i=1}^n Y_{i1}\\...\\\frac{1}{n}\sum_{i=1}^n Y_{im}
\end{pmatrix}\approx \begin{pmatrix} l_1(\hat{y}) \\ ... \\ l_m(\hat{y})\end{pmatrix}.
\end{equation*}

  As a first approach we are using the method of Tikhonov and minimise the following functional with finite-dimensional residuum (fdr)

\begin{equation}\label{tikhapprox}
\arg\min_{x\in\mathcal{X}}\left\lVert \begin{pmatrix} l_1(Kx) \\ ... \\ l_m(Kx)\end{pmatrix}- \bar{Y}_n\m\right\rVert^2_{\mathbb{R}^m}+\alpha \lVert x \rVert_{\mathcal{X}}^2.
\end{equation}

The main question of this work is whether the unique minimiser of \eqref{tikhapprox}, denoted by $R\m_{\alpha}\bar{Y}_n\m$, converges to $\hat{x}$ for $m,n\to\infty$ for adequately chosen $\alpha=\alpha(m,n)$. 
Hereby, an important quantity is the measurement error $\| \bar{Y}_n\m - \begin{pmatrix} l_1(\hat{y}) & ... & l_m(\hat{y})\end{pmatrix}^T\|$, which by randomness is unknown and has to be guessed. The i.i.d assumption yields a natural estimator 

\begin{equation}\label{natest}
\delta_{m,n}^{est}:=\sqrt{\frac{m}{n} s^2_{m,n}},
\end{equation}

  where $s^2_{m,n}:=\frac{1}{m}\sum_{j=1}^m \frac{1}{n-1}\sum_{i=1}^n\left(Y_{ij}-\frac{1}{n}\sum_{l=1}^n Y_{lj}\right)^2$ is 
the mean of the sample variances. The estimator $s_{m,n}^2$ for the unknown variance $E(Y_{ij}-l_j(\hat{y}))^2$ is natural in our general setting. If one has more information about the structure of the discretisation, e.g. in regression problems where the unknown function is measured along a grid, more specified choices may also be reasonable. See \cite{rice1984bandwidth} and \cite{dette1998estimating}, where the variance is estimated in such settings with only one measurement on each channel (i.e. for $n=1$). In \cite{dai2015difference} different methods are compared to each other for repeated measurements on each channel $n>1$. In particular, it is shown that our choice is asymptotically optimal (for $n,m\to\infty$), but that there are better choices for finite sample sizes given that higher moments of the measurement error exist. From a deterministic view point, in order to guarantee convergence it would arguably be necessary to assure that the measurement error tends to $0$, i.e. that $\delta_{m,n}^{est}\to 0$ in probability (or a.s. or in root mean square), which holds if and only if $m/n\to 0$ (see Proposition \ref{sec:5prop1}). This will be a central assumption in most of this manuscript. In the lens of classical results from the statistical side this however seems to be an unnaturally strong condition, since in many special cases it is sufficient to have that the (overall) measurement error stays bounded, i.e. that $n=m\to\infty$ (\cite{vogel2002computational}) or  even that only the component-wise measurement error converges to $0$, i.e. that merely $m,n\to\infty$ without any specific relation between $m$ and $n$ (\cite{cavalier2011inverse}). We will show in Section \ref{sec:discussion} below that somehow surprisingly the condition $m/n\to0$ is in essence necessary to guarantee convergence in our general setting.
 
 One of the most natural and popular strategies to determine the regularisation parameter $\alpha$ in \eqref{tikhapprox} is the discrepancy principle \cite{morozov1968error}, which constitutes in solving

\begin{equation}\label{dp1}
 \left\|\begin{pmatrix} l_1(K R_\alpha\m \bar{Y}_n\m) \\ ... \\ l_m(K R_\alpha\m\bar{Y}_n\m)\end{pmatrix} - \bar{Y}_n\m\right\| \approx \delta_{m,n}^{est}
\end{equation} 
 
(see Algorithm 1 with $C_0=1$ for the numerical implementation). We obtain the following convergence result for the discrepancy principle.

\begin{corollary}\label{cor1}
Assume that $K$ is injective with dense range and that $(\delta_{ij})_{i,j\in\N}$ are independent and identically distributed with zero mean and bounded variance. Moreover assume that $(l_j)_{j\in\N}$ is complete and square-summable. Then for $\alpha_{m,n}$ determined by the discrepancy principle \eqref{dp1} there holds

$$\lim_{\substack{m\to\infty\\n\to\infty\\m/n\to 0}}\mathbb{P}\left(\left\|R_{\alpha_{m,n}}\m\bar{Y}_n\m-K^+\hat{y}\right\|\ge \varepsilon\right)=0$$

for all $\varepsilon>0$.
\end{corollary}

All the details to this result can be found in Section \ref{sec:fdr}, where we also more generally treat filter based regularisations as well as a priori parameter choice rules and discretisations $l_j\m$, $j=1,...,m$, $m\in\N$. Let us stress that Corollary \ref{cor1} guarantees convergence without any quantitative knowledge of the quality of the discretisation (error) and for arbitrary unknown error distributions. This might be surprising in view of the Bakushinskii veto (\cite{bakushinskii1984remarks}), which states that quantitative a priori knowledge about the noise is a crucial requirement for solving an inverse problem.  We stress that Corollary \ref{cor1} does not give a convergence rate. In order to obtain a rate additional smoothness assumptions (relative to the forward operator $K$) have to be imposed on the true solution $\hat{x}$ and the relation of $K$ and the discretisation will play a crucial role. This is a topic of actual research and postponed to a later work. Other than that we want to present an alternative approach, which allows to deduce rates in very general settings. However, note that in what follows the rates are deduced by a classical worst-case error analysis and are not optimal in the statistical setting. 
Whether the discrepancy principle can be modified to attain optimal rates (in the statistical setting) in our general frame work is beyond the scope of this work. We will discuss this in more detail in Section \ref{sec:discussion}.  

The main idea of the alternative approach is to first construct from the measured data in $\R^m$ continuous measurements in the Hilbert space $\mathcal{Y}$, see e.g. \cite{garde2021mimicking}.
 For that we  solve the following optimisation problem 

\begin{equation}\label{backproject}
\min_{y\in\mathcal{Y}}\left\|\begin{pmatrix} l_1(y)\\ ... \\ l_m(y)\end{pmatrix} - \bar{Y}_n\m\right\|,
\end{equation}

which has an unique solution with minimal norm (due to Moore-Penrose) denoted by $Z_n\m$ in the following. We restrict to discretisations for which \eqref{backproject} is well-conditioned, see Assumption \ref{disc:idr}. For general discretisations one would need to add an additional regularisation term.
Then, instead of \eqref{tikhapprox} we solve the following optimisation problem with infinite-dimensional residuum (idr)




\begin{equation}\label{tikhapprox2}
\arg\min_{x\in\mathcal{X}}\left\lVert K x- Z_n\m\right\rVert^2_{\mathcal{Y}}+\alpha \left\lVert x \right\rVert_{\mathcal{X}}^2
\end{equation}

 and the regularisation parameter $\alpha$ has to be chosen accordingly to $\| Z_n\m-\hat{y}\|$. With $y\m$ the (unique) minimum norm solution of

$$\min_{y\in\mathcal{Y}}\left\| \begin{pmatrix} l_1(y)\\ ... \\ l_m(y) \end{pmatrix} - \begin{pmatrix} l_1(\hat{y}) \\ ... \\ l_m(\hat{y})\end{pmatrix}\right\|$$

we may decompose this term into a measurement error and a discretisation error

$$\left\| Z_n\m-\hat{y}\right\| \le \left\| Z_n\m - y\m \right\|  + \left\|y\m-\hat{y}\right\|.$$

 Assume that we know an asymptotic bound $\delta_m^{disc}$ for the discretisation error $\|\hat{y}-y\m\|$ (which is natural in various settings, see Section \ref{sec:idr}). One may estimate $\|Z_n\m - y\m\|$ (see Algorithm 2) and should use that many repetitions $n(m,\delta_m^{disc})$, such that this estimator approximately equals $\delta_m^{disc}$. The regularisation parameter $\alpha$ is then again determined via the discrepancy principle

\begin{equation}\label{dp2}
\left\|K R_{\alpha}Z_{n(m,\delta_m^{disc})}\m - Z_{n(m,\delta_m^{disc})}\m\right\| \approx 2\delta_m^{disc},
\end{equation}

with $R_\alpha Z_{n(m,\delta_m^{disc})}\m$ the unique solution of \eqref{tikhapprox2} (see Algorithm 2 with $C_0=1$ for the numerical implementation). We obtain the following result on the convergence and the order optimality.
 
\begin{corollary}\label{cor2}
Assume that $K$ is injective with dense range and that $(\delta_{ij})_{i,j\in\N}$ are independent with zero mean and finite variance. Moreover, the discretisation is complete and well-conditioned (see Proposition \ref{idr:intr}). Let $(\delta_m^{disc})_{m\in\N}$ be an known upper bound for the discretisation error converging to 0 and determine $\alpha_m$ with the discrepancy principle \eqref{dp2}.  Then
$$\lim_{m\to\infty}\mathbb{P}\left(\left\|R_{\alpha_m} Z_{n(m,\delta_m^{disc})}\m-K^+\hat{y}\right\|\ge \varepsilon\right)=0$$

for all $\varepsilon>0$. If moreover there is a $0<\nu\le 1$ and a $\rho>0$ such that $K^+\hat{y}=(K^*K)^{\nu/2}w$ for some $w\in\mathcal{X}$ with $\| w \| \le \rho$, then
 
 \begin{equation*}
 \lim_{m\to\infty}\mathbb{P}\left(\left\lVert R_{\alpha_m}Z_{n(m,\delta_m^{disc})}\m - K^+\hat{y}\right\rVert\le L' \rho^\frac{1}{\nu+1} {\delta_m^{disc}}^\frac{\nu}{\nu+1}\right)= 1
 \end{equation*}
 
  for some constant $L'$.

\end{corollary}

 It is a long standing dilemma that solution strategies for inverse problems typically require a priori knowledge about the noise. For example, in the classical deterministic case an upper bound for the error is given or in the stochastic case one restricts to certain classes of distribution (often Gaussian). In \cite{harrach2020beyond} there was for the first time presented a rigorous convergence theory without any knowledge of the error distribution, if one has multiple measurements (strongly bounded in $L^2$) of the right hand side $\hat{y}$. Here we consider semi-discretised measurements under arbitrary unknown white noise.
It is widely known that discretisation has a regularising effect,
see for example \cite{mathe2001optimal} and \cite{hansen2010discrete} for the discretisation in the deterministic setting, \cite{o1986statistical}, \cite{mathe2001optimal}, 
\cite{mathe2003discretization} and \cite{mathe2006regularization} for the statistical frequentist setting and \cite{kaipio2007statistical} and \cite{ito2015inverse} for the Bayesian approach. In applications related to machine learning often one considers discretisation by random sampling, see e.g. \cite{de2006discretization} or \cite{bauer2007regularization}. In general, one can either first regularise the infinite-dimensional problem and then discretise or, as it is done in this article, one first discretises and then regularises the finite-dimensional problem. Fairly often, inverse problems under white noise (see e.g. \cite{donoho1995nonlinear} and \cite{cavalier2002sharp}) are treated the first way and the white noise is modelled as a Hilbert space process operating on $\mathcal{Y}$, see \cite{bissantz2007convergence} and \cite{cavalier2011inverse}. The major challenge of this modelling is that then the measurements are not elements of $\mathcal{Y}$. This implies that one has to restrict to sufficiently smoothing operators and to include correction terms in the convergence rates. Most importantly, the discrepancy principle, one of the most widely used parameter choice rules in practice, cannot be applied due to the unboundedness of the noise. Thus one rather relies on other parameter choice rules, e.g. cross validation \cite{wahba1977practical} or Lepski’s balancing principle \cite{mathe2003geometry}, even though a modified discrepancy principle could be applied  \cite{blanchard2012discrepancy}.
These technical difficulties are not present in the semi-discretised setting considered here. Among the first results on the discrepancy principle in such a setting we want to mention \cite{vogel2002computational}, where a convergence rate analysis is given under the assumption that the singular value decomposition is known. There the regularisation parameter is determined not for the random residual as in \eqref{dp1} but for its squared expectation. While this gives some important insight such a choice is clearly not implementable. Results for the truly data-driven implementable version \eqref{dp1} are presented in \cite{blanchard2018optimal} and \cite{blanchard2018early}, where optimal rates are deduced for polynomially ill-posed problems under Gaussian white noise. We will compare our results to these in more detail in Section \ref{sec:discussion}. In particular we show that the existence of the fourth moment of the error distribution is a crucial requirement for the latter references. A major difference to most of the aforementioned references is that there the variance of the measurement error (respectively the noise level of the white noise process) is assumed to be known. This is justified by the fact that usually little attention is put on the behaviour of the solution as the discretisation dimension grows and in fact the error distribution is assumed to be independent of the size of the discretisation. Here we explicitly allow the error distribution to vary with the size of the discretisation and thus we make the estimation part of the analysis.

To put it in a nut shell, the main result in this work guarantees convergence for arbitrary unknown distribution, as long as one is able to measure repeatedly, under quite general assumptions on the discretisation which are only of qualitative nature and  most importantly are independent of the unknown exact right hand side.
 In this paper we restrict to the discrepancy principle as an a posteriori rule that is known to be challenging in stochastic regularisation even for strongly $L^2$-bounded noise, see \cite{harrach2020beyond} and \cite{Jahn_2020}. Still, we expect that the results can be extended to other a posteriori parameter choice rules as well, since the central tools to handle the stochastic noise, namely Lemma \ref{estlem1} and Lemma \ref{estlem2}, do not depend on the chosen regularisation or parameter choice rule. Finally, if one neither has information about the noise level nor is one able to repeat a measurement solely so-called heuristic parameter choice rules could be used. The term heuristic is referring to the fact that convergence results only hold under a restricted noise analysis. Here we want to mention the quasi-optimality criterion as one the most popular heuristic rules, see e.g. \cite{tikhonov1965use} and \cite{kindermann2018quasi} for results under (almost surely) bounded noise. See also \cite{bauer2008regularization}, where an analysis of the quasi-optimality criterion under white noise in a Bayesian setting is presented.

The rest of the article is organised as follows. In Section \ref{sec:fdr} and Section \ref{sec:idr} we will show the $L^2$-convergence (a.k.a. convergence of the mean squared error) of a priori parameter choice rules and the convergence in probability of the discrepancy principle for the both approaches respectively.  In Section \ref{sec:discussion} we compare the results in detail to existing ones. The proofs are deferred to Section \ref{sec:proofs} and we conclude with a numerical study in Section \ref{sec:num} and some final remarks in Section \ref{sec:con}.

\section{Approach with finite-dimensional residuum}\label{sec:fdr} 
 
We start with a precise and more general definition of our discretisation scheme. Therefore we introduce as follows the discretisation (operators)

\begin{align}\label{discmod}
P_m:\mathcal{Y}\to \R^m,~  y\mapsto \begin{pmatrix} l_1\m(y)\\ ... \\ l_m\m(y)\end{pmatrix}, 
\end{align}

with the corresponding measurements

\begin{equation*}
Y_{ij}\m:=l_j\m(\hat{y})+\delta_{ij}\m.
\end{equation*}

and $\|l_1\m\| = ... = \|l_m\m\|$. That is the measurement channels and also the error distribution may depend on the number $m$ of measurement channels now. We will often use that by the Riesz representation theorem there are unique $(\eta_j\m)_{j\le m, m\in\N}$ such that $l_j\m(y)=(\eta_j\m,y)$ for all $y\in\mathcal{Y}$. For convenience we will assume that  $P_mP_m^*$ is bijective and thus $P_m$ has a singular value decomposition with exactly $m$ (non-zero) singular values. 

 From now on we consider general filter-based regularisations $R_{\alpha}\m:=F_{\alpha}\lk\lk P_mK\rk^*P_mK\rk \lk P_mK\rk^*$, where $(F_\alpha)_{\alpha}$ fulfills Assumption \ref{assfilt} below.
 
 \begin{assumption}[Filter]\label{assfilt}
$(F_\alpha)_{\alpha>0}$ are piecewise continuous real valued functions on $[0,\|K\|^2]$ with 

\begin{equation}\label{assfilt:1}
\lim_{\alpha\to0}\sup_{\varepsilon\le \lambda\le \|K\|^2}\left|F_\alpha(\lambda)-1/\lambda\right|=0
\end{equation}
 for all $\varepsilon>0$ and $\lambda |F_\alpha(\lambda)|\le C_R\in\R$ for all $\lambda\in(0,\|K\|^2]$ and $\alpha>0$.  Moreover it has qualification $\nu_0\ge0$, i.e. $\nu_0$ is maximal such that for all $0\le\nu\le \nu_0$ there exists a constant $C_\nu\in \R$ such that

$$
\sup_{\lambda\in(0,\|K\|^2]}\lambda^\frac\nu2\left| 1 - F_\alpha(\lambda)\lambda\right| \le C_\nu \alpha^\frac{\nu}{2}.
$$

Hereby, for $\nu=0$ the constant $C_0$ is assumed to be known. Finally, there exists a constant $C_F\in\R$ with $|F_\alpha(\lambda)| \le C_F/\alpha$ for all $\alpha>0$ and $\lambda\in(0,\|K\|^2]$.
\end{assumption}

\begin{remark}
Assumption \ref{assfilt} coincides with the classical ones in \cite{engl1996regularization} up to \eqref{assfilt:1}, which is usually replaced by the weaker condition $\lim_{\alpha\to 0} F_\alpha(\lambda) = 1/\lambda$ for all $\lambda\in(0,\|K\|^2]$. However, it is easy to verify that the generating filter of all popular methods, e.g. truncated singular value, (iterated) Tikhonov or Landweber regularisation fulfill Assumption \ref{assfilt}. In all these cases it holds that $C_0=1$.
\end{remark}

We impose the following more abstract condition on the discretisation, which generalises the ones from the introduction.
 
\begin{assumption}[Disretisation for finite-dimensional residuum]\label{disc:fdr}
There exists an injective operator $A\in\mathcal{L}(\mathcal{Y})$ such that $\lim_{m\to\infty}P_m^*P_m y= Ay$ for all  $y\in\mathcal{Y}$.
\end{assumption}

We list some popular discretisation schemes which fulfill Assumption \ref{disc:fdr}, starting with the one from the introduction.

\begin{proposition}\label{fdr:intr}
Assume that $l_j\m = l_j$ for all $j=1,...,m$ and $m\in\N$ with $(l_j)_{j\in\N} \subset \mathcal{L}(\mathcal{Y},\R)$, where $(l_j)_{j\in\N}$ is complete and square-summable, i.e. for all $y\in \mathcal{Y}\setminus \{0\}$ there is a $l_j$ such that $l_j(y)\neq 0$  and there holds $\sum_{j=1}^\infty l_j(y)^2 <\infty$. Then Assumption \ref{disc:fdr} is fulfilled.
\end{proposition}

Often the limit operator $A$ will be the identity $Id=Id_\mathcal{Y}$, e.g. in the case when we  discretise by box or hat functions.

\begin{proposition}\label{fdr:box}
Assume that $\mathcal{Y}=L^2([0,1])$ and we discretise by box functions, i.e. $l_j\m = (\eta_j\m,\cdot)$ with $\eta_j\m = \sqrt{m} \chi_{[\frac{j-1}{m},\frac{j}{m})}$ for $j=1,...,m$ and $m\ge 2$. Then Assumption \ref{disc:fdr} is fulfilled with $A=Id$.
\end{proposition}

\begin{proposition}\label{fdr:hat}
Assume that $\mathcal{Y}=L^2([0,1])$ and we discretise by hat functions, i.e. $l_j\m=(\eta_j\m,\cdot)$ with

\begin{enumerate}
\item $\frac{\eta_j\m}{\sqrt{m-1}} =\left(1-j+(m-1)x\right) \chi_{[\frac{j-1}{m-1},\frac{j}{m-1})}+ \left(j+1-(m-1)x\right)\chi_{[\frac{j}{m-1},\frac{j+1}{m-1})}$ for $j=2,...,m-1$,\\
\item $\eta_1\m = \sqrt{2(m-1)}(1+j-(m-1)x)\chi_{[\frac{j}{m-1},\frac{j+1}{m-1})}$,\\
\item $\eta_m\m = \sqrt{2(m-1)}((m-1)x-j+1)\chi_{[\frac{j-1}{m-1},\frac{j}{m-1}]}$.
\end{enumerate}

Then Assumption \ref{disc:fdr} is fulfilled with $A=Id$.
\end{proposition}

\subsection{A priori regularisation with finite-dimensional residuum}\label{sec:fdr:apriori}

We start with a priori regularisations and impose the following assumption on the error, which is weaker than the one in the introduction. Basically, solely independence on each measurement channel and a uniform boundedness of the variances are required.

\begin{assumption}[Error for a priori regularisation]\label{err:apriori}
For all $m,j\in\N$ the random variables $\left(\del\right)_{i\in\N}$ are independent with zero mean and there exists $C_d\in\R$ with

$$\sup_{\substack{m,i,j\in\N\\ j\le m}} \E[{\del}^2]\le C_d.$$
\end{assumption}

Since the sample variance depends on the data we set $s_{m,n}^2=1$ here, such that $\delta_{m,n}^{est} = \sqrt{m/n}$. This has the advantage that the regularisation parameter $\alpha$ is independent of the measurements $Y_{ij}\m$. We obtain convergence in $L^2$ for a priori regularisation. 
\begin{theorem}\label{th1}
Assume  that $K$ is injective, the discretisation fulfills Assumption \ref{disc:fdr}, the error is accordingly to Assumption \ref{err:apriori} and $(F_{\alpha})_{\alpha>0}$ fulfills Assumption \ref{assfilt}. Take an a priori parameter choice rule with $\alpha(\delta) \stackrel{\delta\to 0 }{\longrightarrow} 0$ and $\delta/\sqrt{\alpha(\delta)}\stackrel{\delta\to 0 }{\longrightarrow} 0$. Then there holds

$$\lim_{\substack{m,n\to\infty\\ m/n\to 0}}\E\left\| R_{\alpha(\delta_{m,n}^{est})}\m\bar{Y}_n\m-K^+\hat{y}\right\|^2=0.$$

\end{theorem}

\subsection{A posteriori regularisation with finite-dimensional residuum}

We turn our attention to the discrepancy principle. The regularisation parameter is determined through

\begin{equation}
\left\lVert P_mK R_{\alpha}\m \bar{Y}_n\m -\bar{Y}_n\m\right\rVert \approx \delta_{m,n}^{est}
\end{equation}

and in the definition of $\delta_{m,n}^{est}=\sqrt{s_{m,n}^2 m/n}$ we choose the mean of the sample variances $$s_{m,n}^2:=\frac{1}{m}\sum_{j=1}^m \frac{1}{n-1}\sum_{i=1}^n\lk Y_{ij}\m - \frac{1}{n}\sum_{l=1}^n Y_{lj}\m\rk^2,$$ since we will need a sharp estimation of the right hand side.
 We implement the discrepancy principle with Algorithm 1.

 \begin{algorithm}[H]\label{algorithm1}
 \caption{Discrepancy principle with fdr approach}
\begin{algorithmic}[1]
\STATE Choose $\tau>C_0$ (from Assumption \ref{assfilt}) and $q\in (0,1)$;
\STATE Input: Measurements $Y_{ij}\m=l_j\m(\hat{y})+\delta_{ij}\m$ with $i\le n$ and $j\le m$;
\STATE Set $\bar{Y}_n\m=\frac{1}{n}\sum_{i=1}^n\begin{pmatrix} Y_{i1}\m \\ ... \\Y_{im}\m\end{pmatrix}$;
\STATE Set $\delta_{m,n}^{est} = \sqrt{\frac{m}{n} \frac{1}{m}\sum_{j=1}^m \frac{1}{n-1}\sum_{i=1}^n\left(Y_{ij}\m-\frac{1}{n}\sum_{l=1}^n Y_{lj}\m\right)^2}$;
\STATE $k=0$;
  \WHILE{$\left\| \begin{pmatrix} l_1\m(KR_{q^k}\m \bar{Y}_n\m) \\ ... \\ l_m\m(K R_{q^k}\m\bar{Y}_n\m)\end{pmatrix}-\bar{Y}_n\m\right\| > \tau\delta_{m,n}^{est}$}
  \STATE $k=k+1$;
  \ENDWHILE
  \STATE $\alpha_{m,n}=q^k$;
\end{algorithmic}
\end{algorithm}

 Algorithm $1$  terminates (with a probability tending to $1$ for $m\to\infty$) if $K$ has dense range and (for $m$ large enough) $\E (Y_{11}\m-\E Y_{11}\m)^2>0$, for details see \cite{harrach2020beyond}. We now extend the assumptions of the error in the introduction.
\begin{assumption}[Error for a posteriori regularisation]\label{err:disc}
 It holds that either
  \begin{enumerate}
  \item the random variables $\left(\delta_{ij}^{(m)}\right)_{i,j,m\in\N}$ are i.i.d. with zero mean and bounded variance, or
  \item there are $C_d\in\R$ and $p>1$ such that for all $m\in\N$ the random variables $\left(\delta_{ij}\m\right)_{ i,j\in\N}$ are i.i.d with zero mean and $\frac{\E\left|\del\right|^{2p}}{\left(\E{\del}^2\right)^p}\le C_d$.
  \end{enumerate}
  
\end{assumption}

The main difference between Assumption \ref{err:disc}.1 and \ref{err:disc}.2 is that for the latter the error distribution may vary with $m$, to the cost of a uniform moment condition.

\begin{remark}
Assumption \ref{err:disc}.2 guarantees that the error distribution does not degenerate too much. It is trivially fulfilled if e.g. $\delta_{ij}\m \stackrel{d}{=} c_m X$, with $\E |X|^{2p} < \infty, (c_m)_{m\in\N}\subset\R \setminus \{0\}$. 
\end{remark}

Now we are ready to prove convergence of the discrepancy principle. In contrast to the previous section where we showed convergence in $L^2$ for a priori regularisation methods, the result will now be on convergence in probability (compare this to the counter example in 3.1 in \cite{harrach2020beyond}). 
\begin{theorem}\label{th2}
Assume that $K$ is injective with dense range and that the discretisation fulfills Assumption \ref{disc:fdr} and that the error is accordingly to Assumption \ref{err:disc} and $(F_\alpha)_{\alpha>0}$ fulfills Assumption \ref{assfilt} with a qualification $\nu_0>1$. Then, with $\alpha_{m,n}$ the output of Algorithm 1

$$\lim_{\substack{m,n\to\infty\\ m/n\to 0}}\mathbb{P}\left(\left\|R_{\alpha_{m,n}}\m\bar{Y}_n\m-K^+\hat{y}\right\|\ge \varepsilon\right)=0$$

for all $\varepsilon>0$.
\end{theorem}

Corollary \ref{cor1} is an easy consequence of Theorem \ref{th2} and Proposition \ref{fdr:intr}.
 We conclude the section with a remark regarding Assumption \ref{err:disc}.

\begin{remark}
As already mentioned Assumption \ref{err:disc} excludes distributions which are too degenerated and guarantees that $\E {\delta_{11}\m}^2$ is in some sense uniformly estimatable. We quickly sketch what can go wrong if the distributions degenerate too much.  Assume that $(\del)_{ij}$ are i.i.d. for all $m\in\N$, with $$\mathbb{P}(\del=x) = \begin{cases} \frac{1}{m^4}  & \mbox{ for }x=-\sqrt{m^4-1}\\
\frac{m^4-1}{m^4} & \mbox{ for } x=1/\sqrt{m^4-1}\end{cases}.$$
Thus $\E \delta_{11}\m = 0$ and $\E {\delta_{11}\m}^2 = 1$ but for any $p>1$
$$\frac{\E \left|\delta_{11}\m\right|^{2p}}{\left(\E{\delta_{11}\m}^2\right)^{p}}\ge \frac{1}{m^4}\left|\sqrt{m^4-1}\right|^{2p}=\left(1-\frac{1}{m^4}\right)\left|m^4-1\right|^p \to \infty$$ as $m\to\infty$. Thus Assumption \ref{err:disc} is violated and with the choice $n(m)=m^2$ it holds that $\lim_{m\to\infty} \frac{m}{n(m)}= 0$, but we have that

\begin{align*}
\mathbb{P}\left( \delta_{m,n(m)}^{est} = 0\right)& = \mathbb{P}\left( s_{m,n(m)}^2=0\right)\\
&= \mathbb{P}\left( \delta_{ij}\m = 1/\sqrt{m^4-1}~,~ i =1,...,m^2,j=1,...,m\right)\\
 &= \left(1-\frac{1}{m^4}\right)^{m^3} = \left( \left(1-\frac{1}{m^4}\right)^{m^4}\right)^\frac{1}{m}\to 1
\end{align*}
as $m\to\infty$. Thus with asymptotic probability $1$ the discrepancy principle cannot even be applied for this choice of $n$. The number of repetitions $n(m)=m^2$ is simply too small to estimate the variance of $\delta_{11}\m$ adequately.
\end{remark}

\section{Approach with infinite-dimensional residuum}\label{sec:idr}

We turn our attention to the second approach (\ref{tikhapprox2}). The strategy is to  use the measured data to construct virtual measurements in the infinite-dimensional Hilbert space $\mathcal{Y}$ and then to regularise the infinite-dimensional problem using classical methods. For the regularisation we will need in the following an upper bound for the discretisation error which we denote by  $\delta_m^{disc}\ge\|\hat{y}-P_m^+P_m \hat{y}\|$. Decomposing the true data error yields

\begin{equation*}
\left\lVert \hat{y}-P_m^+\bar{Y}_n\m\right\rVert\le \left\lVert \hat{y}-P_m^+P_m\hat{y}\right\rVert + \left\lVert P_m^+P_m\hat{y}-P_m^+\bar{Y}_n\m \right\rVert.
\end{equation*}

 As in the approach with a finite-dimensional residuum there is a generic way (given below) to estimate the (projected) measurement error $\|P_m^+\bar{Y}_{n}\m-P_m^+P_m\hat{y}\|$ . So that it is natural to choose the number of repetitions $n=n(m,\delta_m^{disc})$ in such a way  that this estimator  approximately equals the discretisation error  $\delta_m^{disc}$. After that one may use any deterministic regularisation together with total estimated noise level 
\begin{equation}\label{idr:est:disc} 
 2\delta_m^{disc} \approx \left\| \hat{y} - P_m^+P_m\hat{y}\right\| + \left\|P_m^+P_m \hat{y} - P_m^+\bar{Y}_{n(m,\delta_m^{disc})}\m\right\| \ge \left\| \hat{y} - P_m^+\bar{Y}_{n(m,\delta_m^{disc})}\m\right\|.
\end{equation} 
  We again consider regularisations $R_\alpha:=F_\alpha(K^*K)K^*$ induced by a regularising filter (see Assumption \ref{assfilt}) and make the following assumptions for the discretisation and our a priori knowledge of it.

\begin{assumption}[Discretisation for infinite-dimensional residuum]\label{disc:idr}
We assume that we know an asymptotic upper bound $(\delta_m^{disc})_{m\in\N}$ for the discretisation error and asymptotic upper and lower bounds $(c_m)_{m\in\N}, (C_m)_{m\in\N}$ for the singular values $(\sigma_j\m)_{j\le m, m\in\N}$ of  $(P_m)_{m\in\N}$. More precisely, these bounds have to fulfill $\| \hat{y}-P_m^+P_m\hat{y}\| \le \delta_m^{disc}, 0<c_m\le \sigma_j\m \le C_m$ for all $j=1,..,m$ and $m$ large enough, and $\delta_m^{disc}\to 0$ as $m\to\infty$ and
\begin{equation}\label{wellposednessdisc}
\limsup_{m\to\infty}\kappa(P_m):=\limsup_{m\to\infty}\|P_m\|\|P_m^+\| = \limsup_{m\to\infty} \frac{\max_{j=1,...,m} \sigma_j\m}{\min_{j=1,...,m} \sigma_j\m}\le \limsup_{m\to\infty} \frac{C_m}{c_m}<\infty.
\end{equation}
\end{assumption}

 Often the stability assumption \eqref{wellposednessdisc} can be guaranteed by an angle condition for the unique $\eta_j\m\in\mathcal{Y}$ that fulfill $l_j\m(y)=(\eta_j,y)$ for all $y\in\mathcal{Y}$.

 \begin{proposition}\label{idr:angle}
 Assume that 
 $$\sup_{m\in\N} \sup_{j \le m} \sum_{i\neq j} \frac{|(\eta_i\m,\eta_j\m)|}{\|\eta_1\m\|^2}\le c<1.$$
 Then $c_m:=\|\eta_1\m\|^2(1-c)\le \sigma_j\m\le\|\eta_1\m\|^2(1+c)=:C_m$ for $j=1,..,m$ and $m$ large enough and thus $\kappa(P_m) \le \frac{1+c}{1-c}$. 
 \end{proposition}

Clearly, the angle condition is always satisfied for orthogonal discretisations. It would be desirable to also have a simple criterion to guarantee that $\delta_m^{disc}$ tends to $0$. For $P_m = \begin{pmatrix} l_1(\cdot) & ... & l_m(\cdot)\end{pmatrix}^T$ (the $l_j$ do not depend on $m$) this could be guaranteed e.g. when $(l_j)_{j\in\N}$ is complete, because then $\mathcal{N}(P_m) \supset \mathcal{N}(P_{m+1}) \supset ...$ converges monotonically to $0$ and $P_m^+P_m = P_{\mathcal{N}(P_m)^\perp}$ is the orthogonal projection onto the orthogonal complement of $\mathcal{N}(P_m)$. A straight forward generalisation of completeness to discretisation schemes $P_m = \begin{pmatrix} l_1\m(\cdot) & ... & l_m\m(\cdot)\end{pmatrix}^T$ would be to presume, that for all $y\in\mathcal{Y}\setminus\{0\}$ there exists a $\varepsilon>0$ such that $\|P_m y\| \ge \varepsilon$ for $m$ large enough. The following counter example however shows that this is not sufficient to guarantee that the discretisation error tends to $0$.

\begin{remark}
Let $(v_j)_{j\in\N}$ be an orthonormal basis of $\mathcal{Y}$. Set $l_j\m(y)= (y,v_j)$ for $j=2,...,m$ and $l_1\m(y)= ( y, v_1/\sqrt{2} + v_{m+1}/\sqrt{2})$. For $y\neq 0$ we set $\varepsilon=|(y,v_j)|/2$ with $j=\min\{ j'~: (y,v_j') \neq 0\}$. Then clearly $\|P_m y \| \ge \varepsilon$ for $m$ large enough. But, it holds that $\mathcal{N}(P_m)=< v_1/\sqrt{2}-v_{m+1}/\sqrt{2},v_{m+2},v_{m+3},...,>$ and thus $$v_1-P_m^+P_mv_1=P_{\mathcal{N}(P_m)} v_1 = v_1/\sqrt{2} \not\to 0$$ for $m\to\infty$.
\end{remark}

 We now show that Assumption \ref{disc:idr} is fulfilled for various popular discretisation schemes. We start with the example from the introduction.

\begin{proposition}\label{idr:intr}
Assume that $l_j\m = l_j=(\eta_j,\cdot)$ for all $j=1,...,m$ and $m\in\N$ with $(l_j)_{j\in\N} \subset \mathcal{L}(\mathcal{Y},\R)$ and $(\eta_j)_{j\in\N} \subset \mathcal{Y}$ and that we know $c$ and $\delta_m^{disc}$ such that $\delta_m^{disc}\ge \|\hat{y}-P_m^+P_m\hat{y}\|$ and $(l_j)_{j\in\N}$ is complete, i.e. for all $y\in \mathcal{Y} \setminus \{0\}$ there exists a $l_j$ such that $l_j(y)\neq 0$, and well-conditioned that is  $$ \sup_{j\in\N}\sum_{\substack{i =1 \\ i\neq j}}^\infty |(\eta_i,\eta_j)|/\|\eta_1\|^2\le c < 1.$$ Then Assumption \ref{disc:idr} is fulfilled for $\delta_m^{disc}$ and $c_m = 1-c, C_m=1+c$.
\end{proposition}
Next we consider discretisation along the singular directions of $K$, see the beginning of Section \ref{sec:proofs} for the definition of the singular value decomposition. 
\begin{proposition}\label{idr:svd}
Assume that the singular value decomposition $(\sigma_l,v_l,u_l)_{l\in\N}$ of $K$ is known. Then for the discretisation $l_j\m =  (u_j, \cdot)$ Assumption \ref{disc:idr} is (asymptotically) fulfilled with the bounds $\delta_m^{disc} = f_m \sigma_{m+1}$ (where $f_m$ is any sequence with $f_m \to \infty$ as $m\to\infty$) and $c_m=C_m=1$.
\end{proposition}

In many important cases, for example if $K$ is a Fredholm integral equation with sufficient smoothing kernel, Assumption \ref{disc:idr} is also fulfilled for discretisation with box or hat functions. 

\begin{proposition}\label{idr:box}
Consider $\mathcal{X}=\mathcal{Y} = L^2(0,1)$ and $\eta_j\m$ the box functions from Proposition \ref{fdr:box}. If $\hat{y}$ is continuously differentiable, then Assumption \ref{disc:idr} is fulfilled with bounds $\delta_m = f_m/m$ and $c_m = C_m=1$ where $(f_m)_m$ is arbitrary with $\lim_m f_m = \infty$.
\end{proposition}

\begin{proposition}\label{idr:hat}
Consider $\mathcal{X}=\mathcal{Y} = L^2(0,1)$ and $\eta_j\m$ the hat functions from Proposition \ref{fdr:hat}. If $\hat{y}$ is continuously differentiable, then Assumption \ref{disc:idr} is fulfilled with bounds $\delta_m = f_m/m$ and $c_m = 1/6$ and $C_m=7/6$, where $\lim_m f_m = \infty$. If $\hat{y}$ is twice continuously differentiable, then Assumption \ref{disc:idr} is fulfilled with bounds $\delta_m = f_m/m^2$ and $c_m=1/6$ and $C_m=7/6$, with $\lim_m f_m = \infty$.
\end{proposition}

It remains to determine the number of repetitions $n(m,\delta_m^{disc})$ such that the (back projected) measurement error fulfills $\|P_m^+P_m\hat{y} - P_m^+\bar{Y}_{n(m,\delta_m^{disc})}\m\| \approx \delta_m^{disc}$. This number depends on the singular value decomposition of $P_m$ and the variance $\E {\delta_{11}\m}^2$. More precisely, with $(\sigma_j\m,v_j\m,u_j\m)_{j\le m}$ the singular value decomposition of $P_m$  and $e_1\m,...,e_m\m$ the standard basis of $\mathbb{R}^m$, it holds that

\begin{align*}
& \lVert P_m^+\bar{Y}_n\m - P_m^+P_m\hat{y}\rVert^2 = \sum_{j=1}^m \frac{1}{{\sigma_j\m}^2} \left(\sum_{l=1}^m \sum_{i=1}^n \frac{\delta_{ij}\m}{n} (u_j\m,e_l\m)\right)^2\\
\Longrightarrow& \E \|P_m^+\bar{Y}_n\m-P_m^+P_m\hat{y}\|^2  = \frac{\E{\delta_{11}\m}^2}{n}\sum_{j=1}^m\frac{1}{{\sigma_j\m}^2}.
\end{align*}

Thus with our lower bound $c_m\le \sigma_j\m$ we determine 

\begin{equation*}
n(m,\delta_m^{disc}) := \min\left\{ n \ge 2~:~ \frac{ms_{m,n}^2}{nc_m^2}\le {\delta_m^{disc}}^2 \right\},
\end{equation*}
with $s_{m,n}^2 = 1$ or $s_{m,n}^2 = \frac{1}{m}\sum_{j=1}^m \frac{1}{n-1}\sum_{i=1}^n\left(Y_{ij}\m-\frac{1}{n}\sum_{l=1}^n Y_{lj}\m\right)^2$.

\subsection{A priori regularisation with infinite-dimensional residuum}

For  a priori regularisations we set $s_{m,n}^2=1$ so that $n(m,\delta)$ and the measurements $Y_{ij}\m$ are independent. The convergence result holds true with the same assumption for the error as in Section \ref{sec:fdr:apriori}.

\begin{theorem}\label{th1b}
Assume that $K$ is injective, the discretisation fulfills Assumption \ref{disc:idr}, the error  is accordingly to Assumption \ref{err:apriori} and  $(F_\alpha)_{\alpha>0}$ fulfills Assumption \ref{assfilt}. Take an a priori parameter choice rule with $\alpha(\delta) \stackrel{\delta\to 0 }{\longrightarrow} 0$ and $\delta/\sqrt{\alpha(\delta)}\stackrel{\delta\to 0 }{\longrightarrow} 0$. Then there holds

$$\lim_{m\to\infty}\E\left\|R_{\alpha(\delta_m^{disc})}P_m^+\bar{Y}_{n(m,\delta_m^{disc})}\m-K^+\hat{y}\right\|^2 = 0$$

 for $n(m,\delta_m^{disc})=\lceil \frac{m}{c_m^2{\delta_m^{disc}}^2}\rceil$.
\end{theorem}
\begin{remark}
Note that for a priori regularisation one can relax the condition on $\delta_m^{disc}$ in Assumption \ref{disc:idr} to $\lim_{m\to\infty} \delta_m^{disc}=0$ and $\limsup_{m\to\infty} \frac{\delta_m^{disc}}{\|\hat{y}-P_m^+P_m\hat{y}\|} > 0$.
\end{remark}

\subsection{A posteriori regularisation with infinite-dimensional residuum}

 We determine the stopping index $n(m,\delta_m^{disc})$ more accurately with the sample variance and set $s_{m,n}^2:= \frac{1}{m}\sum_{j=1}^m \frac{1}{n-1}\sum_{i=1}^n\left(Y_{ij}\m-\frac{1}{n}\sum_{l=1}^n Y_{lj}\m\right)^2.$ We implement the discrepancy principle in Algorithm 2.

 \begin{algorithm}[H]\label{algorithm2}
 \caption{Discrepancy principle with idr approach}
\begin{algorithmic}[1]
\STATE Choose $\tau>C_0$ (from Assumption \ref{assfilt}) and $q\in (0,1)$;
\STATE Input: Number of measurement channels $m$, measurements $Y_{ij}\m$, $j\le m, i\in\N$, upper bound $\delta_m^{disc}$ for discretisation error, lower bound $c_m$ for singular values of $P_m$;
\STATE Determine $n(m,\delta_m^{disc}):=\min \left\{n'\ge 1~:~ \frac{ms_{m,{n'}}^2}{n'c_m^2}\le {\delta_m^{disc}}^2\right\}$ from measurements $Y_{ij}\m$.
\STATE Set $\bar{Y}_{n(m,\delta_m^{disc})}\m=\frac{1}{n(m,\delta_m^{disc})}\sum_{i=1}^{n(m,\delta_m^{disc})}\begin{pmatrix} Y_{i1}\m \\ ... \\ Y_{in}\m\end{pmatrix}$;
\STATE $k=0$;
  \WHILE{$\left\| (KR_{q^k}P_m^+\bar{Y}_{n(m,\delta_m^{disc})}\m-P_m^+\bar{Y}_{n(m,\delta_m^{disc})}\m\right\| > 2\tau\delta^{disc}_m$}
  \STATE $k=k+1$;
  \ENDWHILE
  \STATE $\alpha_m=q^k$;
\end{algorithmic}
\end{algorithm}

Algorithm 2 terminates under the same conditions as Algorithm 1. The back propagating of the measurements induces correlations, which forces us to impose slightly stricter conditions on the error distribution than in the setting before. On the other hand, the regularisation is now done in $\mathcal{Y}$ (no matter which $m$), which allows to use classical results to obtain a convergence rate.

\begin{theorem}\label{th4}
Assume that $K$ is injective with dense range and that the discretisation fulfills Assumption \ref{disc:idr} and that the error  is accordingly to Assumption \ref{err:disc}, with $p\ge 2$ in the case of \ref{err:disc}.2 and $(F_\alpha)_{\alpha>0}$ fulfills Assumption \ref{assfilt} with a qualification $\nu_0>1$. For $\tau>C_0$, let $\alpha_m$ and $\bar{Y}_{n(m,\delta_m^{disc})}\m$ be the output of the discrepancy principle as implemented in Algorithm 2. Then

$$\lim_{m\to\infty}\mathbb{P}\left(\left\|R_{\alpha_m}P_m^+\bar{Y}_{n(m,\delta_m^{disc})}\m-K^+\hat{y}\right\|\ge \varepsilon\right)=0.$$

 If moreover there exist a $0<\nu\le \nu_0-1$ and a $\rho>0$ such that $K^+\hat{y}=(K^*K)^{\nu/2}w$ for some $w\in\mathcal{X}$ with $\| w \| \le \rho$, then
 
 \begin{equation*}
 \lim_{m\to\infty}\mathbb{P}\left(\left\lVert R_{\alpha_m}P_m^+\bar{Y}_{n(m,\delta_m^{disc})}\m - K^+\hat{y}\right\rVert\le L' \rho^\frac{1}{\nu+1} \left(\delta_m^{disc}\right)^\frac{\nu}{\nu+1}\right)= 1
 \end{equation*}
 
  for some constant $L'$.

\end{theorem}

Now Corollary \ref{cor2} in the introduction is an easy consequence of Theorem \ref{th4} and Proposition \ref{idr:intr}.

\section{Discussion}\label{sec:discussion}

In this section we discuss the above results in more detail in the light of classical results for statistical inverse problems. Classical results are usually formulated for white noise with intensity $\sigma^2$. With averaging multiple measurements we control the size of the white noise, in fact it holds that $\sigma^2 \asymp 1/n$. Many convergence results require that the size of the discretised measurements is constant for $m$, i.e. that $\sigma^2 = 1/m$. Thus, the classical results hold for $n=m$ and the condition $m/n\to0$ for the convergence results in Section \ref{sec:fdr} seems very strong. 
 Our first result shows that in our general setting this condition is necessary to ensure convergence for a priori regularisation methods. Hereby, note that in the classical deterministic theory a general ill-posed linear problem $K\hat{x}=\hat{y}$ with $K$ some ill-posed linear operator and $(y^\delta)_{\delta>0}$ a sequence of (deterministic) measurements fulfilling $\|\hat{y}-y^\delta\|\le \delta$ can be solved using any filter based regularisation fulfilling Assumption \ref{assfilt} together with a proper a priori choice rule $\alpha=\alpha(\delta)$ (e.g. $\alpha(\delta)=\delta$). In particular, such a choice depends only on the noise level $\delta$ and as such is independent of $K$. Now we assume that in our statistical setting the number of measurement channels equals the number of repetitions, i.e. $m=n$. Further assume that $\alpha=\alpha_m$ is any possible a priori parameter choice rule that converges monotonically to $0$ as $m\to\infty$.
 
 \begin{proposition}\label{discussion:prop}
 There exist a compact operator $K:l^2(\N)\to l^2(\N)$, an element $\hat{y}\in l^2(\N)$, a discretisation scheme $P_m:l^2(\N)\to \R^m$ and an error model $(\delta_{ij})_{ij\in\N}$ such that

$$\E\|R_{\alpha_m}\m \bar{Y}_m\m - K^+\hat{y}\|^2 \to \infty$$

as $m\to\infty$, where $\bar{Y}_n\m = \sum_{i=1}^m\begin{pmatrix} Y_{i1} & ... & Y_{im}\end{pmatrix}^T/m$ with $Y_{ij} = (P_m\hat{y})_j + \delta_{ij}$ and $(R_\alpha\m)_{\alpha>0}$ is the Tikhonov regularisation.
 
 \end{proposition} 

Note that in the above proposition it was important that we fixed the a priori choice rule before the ill-posed problem given through $K$. If we restrict to certain classes, e.g. to mildly ill-posed problems (i.e. the singular values of $K$ fulfill $\sigma_j^2 \asymp j^{-q}$ for some $q>0$), one can give a priori parameter choice rules which converge for $m=n$. 




We now compare our results in detail to recent results for the discrepancy principle and ultimately show that here the condition $m/n\to 0$ is necessary even if one restricts to mildly ill-posed problems. In \cite{blanchard2018early} and \cite{blanchard2018optimal} order-optimal $L^2$-rates are given for the discrepancy principle under Gaussian noise and sufficiently unsmooth data. In these articles the implementation of the discrepancy principle differs from ours (see Algorithm 1) as there in essence the hyperparameter $\tau_m$ depends on the number of measurement channels $m$, whereas we choose a constant $\tau>1$ as in the classical deterministic theory. Precisely, there the regularisation parameter is essentially determined as

$$\inf_{\alpha>0} \|P_m K R_{\alpha}\m \bar{Y}_n\m-\bar{Y}_n\m\| \le \kappa_{m,n}$$

with $|\kappa_{m,n}^2 -\frac{m}{n}\| = \mathcal{O}\left(\frac{\sqrt{m}}{n}\right)$ (note that $\frac{m}{n} = \E[ {\delta_{m,n}^{est}}^2] = \E\|\bar{Y}_n\m-P_m\hat{y}\|^2$ is the (expected) squared noise if $\E[\delta_{ij}^2]=1$). Apart from the fact that we consider more general discretisation schemes, the main difference to the aforementioned results is that we allow for general unknown error distributions. If one instead of $L^2$-convergence asks only for convergence in probability it seems to us that the assumption of Gaussian noise in \cite{blanchard2018early} and \cite{blanchard2018optimal} could be relaxed to arbitrary distributions obeying a finite fourth moment. In fact, under that relaxed assumption one can show that the oscillation of the residual is of a comparably small order and then the choice $\kappa_{m,n}$, due to the correct leading order, can capture the smoothness of $\hat{x}$ more accurately  (and exactly up to saturation) than the plain discrepancy principle (i.e. the choice $\tau>1$) would do and thus gives better convergence rates. However, this procedure seems not to be stable if higher moments do not exist. From the following result one can directly deduce that then the analysis in \cite{blanchard2018early} and \cite{blanchard2018optimal} breaks down. In particular it  shows using a counter example that for both choices the discrepancy principle does not converge in any commonly used mode , if $n/m$ does not converge to $0$.

\begin{theorem}\label{discussion:prop1}
Let $K:l^2(\N)\to l^2(\N)$ be diagonal with singular values $\sigma_j^2 = j^{-q}$ with $q>1$ and singular basis $(v_j)_{j\in\N}$. Let $P_m$ be the discretisation along the singular values, i.e. $(P_m v_j)_l = (v_j,v_l)$ for $j\in\N$ and $l=1,...,m$. Let $Y_{ij} = (P_m \hat{y})_j + \delta_{ij}$ be i.i.d. measurements of $\hat{y}\in l^2(\N)$ (which will be specified below), where $\delta_{ij}$ has density $f_\varepsilon$ (with $\varepsilon<2/11$ as given in Proposition \ref{discproof:prop1} below) and consider $(R_\alpha)_{\alpha>0}$ the spectral cut-off regularisation. Let $c>0$. If $\hat{y} = \sum_{j=1}^\infty j^{-\frac{q}{2}-1} v_j$ and the regularisation parameter $\alpha_{m,n}$ is determined with Algorithm 1 and $\tau>1$, then there exists $C>0$ such that

$$\lim_{\substack{m,n\to\infty\\ m/n \ge c}}\mathbb{P}\left(\| R_{\alpha_{m,n}}\m \bar{Y}_n\m- K^+\hat{y}\| \ge C\right) = 1.$$

If $\hat{y}=0$ and the regularisation parameter $\alpha_{m,n}$ is determined with Algorithm 1, where in line 6 the right hand side $\tau \delta_{m,n}^{est}$ is replaced with $\sqrt{\frac{m+\sqrt{m}}{n}}$, then there exists $p_\varepsilon>0$ such that for any $L>0$ it holds that

$$\liminf_{\substack{m,n\to\infty\\m/n\ge c}}\mathbb{P}\left(\|R_{\alpha_{m,n}}\m\bar{Y}_n\m - K^+\hat{y}\|\ge L\right) >p_\varepsilon.$$

\end{theorem}

A modulation of the discrepancy principle which yields optimal rates (in probability) for linear problems requiring only a finite second moment is studied in \cite{jahn2021optimal}. The analysis there however is restricted to spectral cut-off regularisation. 

 We finish this section with a comment on the way we measure smoothness. As already mentioned in the introduction, existing results usually pay little attention on the behaviour of the solution as the discretisation dimension $m$ tends to $\infty$. Consequently, the source conditions allowing to perform a convergence rate analysis are formulated in the discretised setting. I.e. smoothness of $\hat{x}$ is not measured relative to the infinite-dimensional problem given by $K$ (as it is here), but to the discretised one given by $P_mK$. In the latter case a standard worst-case analysis would yield a convergence rate for the approach with finite-dimensional residuum (which again would not be optimal) and one could compare the rates of the both approaches with finite-dimensional an infinite-dimensional residuum respectively. In Section \ref{sec:num} this is done numerically with problems from the open source \texttt{MATLAB} package Regutools (\cite{hansen1994regularization}). There the approach with finite-dimensional residuum gives slightly better rates and is hence preferable, due to its better stability properties (e.g. convergence without knowledge of a discretisation error). The following example however shows that through discretisation the smoothness of $\hat{x}$ may be substantially deteriorated, in which case the approach with infinite-dimensional residuum would perform better.

 Let $K:l^2(\N)\to l^2(\N)$ be a diagonal operator with $Kv_j = j^{-1} v_j$. Let $\hat{x}=\sum_{j=1}^\infty j^{-2} v_j$ and consider the discretisation $P_m:l^2(\N)\to\R^m$ with

$$\left(P_m y\right)_j = l_j\m(y) := \begin{cases} (y,v_j) & \mbox{for } j=1,...,m-1,\\
                                      e^{-m} v_m + \sqrt{1-e^{-2m}}v_{\lceil e^{m}\rceil} & \mbox{for } j=m.
                                      \end{cases}$$

It holds that $\|l_j\m\| = 1$ for all $m\in\N, j=1,...,m,$ that $P_m^*P_my \to y$ as $m\to\infty$ for all $y\in l^2(\N)$ and that $\kappa(P_m)=1$, thus the discretisation $(P_m)_{m\in\N}$ fulfills Assumptions \ref{disc:fdr} and \ref{disc:idr}. Let $\hat{x}:=\sum_{j=1}^\infty j^{-2} v_j$. Then there exist $\nu, \rho>0$ and $w\in l^2(\N)$ with $\|w\|\le \rho$ and $\hat{x}=\left(K^*K\right)^\frac{\nu}{2} w$ (a possible choice would be $\nu=1$ and $\rho = \sqrt{\sum_{j=1}^\infty j^{-2}}= \frac{\pi}{\sqrt{6}}$). I.e. $\hat{x}$ obeys smoothness $\nu, \rho$ relative to $K$. However, the following proposition shows that $\hat{x}$ obeys asymptotically only a much worse smoothness relative to $P_mK$ (even though $(P_mK)^*P_mK \to K^* K$ uniformly as $m\to\infty$ by Lemma \ref{proof:lem1} below).

\begin{proposition}\label{discussion:prop2}
Let $\hat{x}$ and $K, P_m$ be given as above. Let $\nu_m, \rho_m >0$ and $w_m\in l^2(\N)$ be such that $ P_{\mathcal{N}(P_mK)^\perp}\hat{x} = \left(\left(P_m K\right)^*P_mK\right)^\frac{\nu_m}{2} w_m$ with $\|w_m\| =\rho_m$. Then there exist $c,\varepsilon>0$ such that either

\begin{align*}
&\lim_{m\to\infty} \nu_m = 0,~\lim_{m\to\infty} \rho_m = \infty,\\
\mbox{or}\quad &\liminf_{m\to\infty} \nu_m>0,~\liminf \rho_m \ge e^{\varepsilon m},\\
 \mbox{or}\quad &\limsup_{m\to\infty} \rho_m = \infty,~\limsup_{m\to\infty} \nu_m \le c\frac{\log(m)}{m}
 \end{align*}

holds.

\end{proposition}

\section{Proofs}\label{sec:proofs}
In this section we collect the proofs. We will need the singular value decomposition of an injective compact operator $A$ (see \cite{cavalier2011inverse}): there exists a monotone sequence
$\lVert A \rVert =\sigma_1\ge \sigma_2 \ge ...>0$. Moreover there are families of orthonormal vectors 
 $(u_l)_{l\le \dim(\mathcal{R}(A))}$ and $(v_l)_{l\le \dim(\mathcal{R}(A))}$ with $\overline{span( u_l:l\le \dim(\mathcal{R}(A))}=\overline{\mathcal{R}(A)}$, $\overline{span(v_l:l\le\dim(\mathcal{R}(A))}= \mathcal{N}(A)^\perp$ 
 such that $Av_l=\sigma_lv_l$ and  $A^*u_l=\sigma_lv_l$.

\subsection{Proofs for finite-dimensional residuum}
The assumptions for the discretisation when using the first approach (with finite-dimensional residuum) are such that the discretised operators $K^*P_m^*P_mK$ converge uniformly to a compact and injective operator $K^*AK$. The uniform convergence guarantees that the eigenvalues and spaces of the former converge pointwise to the ones of the latter and the injectivity of the limit operator assures that the unknown $\hat{x}$ is determined arbitrarily precisely by finitely many eigenvectors of the latter. We make this precise with the following lemma.

\begin{lemma}\label{proof:lem1}
Assume that $K$ is injective and that Assumption \ref{disc:fdr} holds true. Then 
$$\|K^*P_m^*P_mK - K^*AK\| \to 0$$
 for $m\to\infty$ and $K^*AK$ is injective, compact, self-adjoint and positive semidefinite. Denote by $(\lambda_j\m)_{j\le m}$ and $(\lambda_j^{(\infty)})_{j\in\N}$ the nonzero eigenvalues with corresponding orthonormal eigenvectors $(v_j\m)_{j\le m}$ of\newline $K^*P_m^*P_mK$ and $K^*AK$ respectively, ordered decreasingly. Then

\begin{enumerate}
\item $\lim_{m\to\infty}\lambda_j\m=\lambda_j^{(\infty)}$ for all $j\in \N$ and
\item for all $x\in \mathcal{X}$ and $\varepsilon>0$, there exists a $M=M(x,\varepsilon)\in\N$ such that $$\limsup_{m\to\infty} \sum_{j=M+1}^m(x,v_j\m)^2\le \varepsilon.$$
\end{enumerate}
\end{lemma}

\begin{proof}

Denote by $(\sigma_j,u_j,v_j)$ the singular value decomposition of $K$ and set $A_m=P_m^*P_m$ and\newline  $C:=\max\left\{ \|A\|,\sup_m \| A_m\|\right\}<\infty$ (uniform boundedness principle). For $\varepsilon>0$ arbitrary define $M\in\mathbb{N}$ implicitly through $2C\sigma_{M+1}\le \varepsilon/2$.
 Then

\begin{align*}
&\lVert A_mK - AK \rVert\\
= &\sup_{\substack{x\in\mathcal{X}\\\lVert x \rVert =1}}\lVert A_mKx-A Kx\rVert = \sup_{\substack{\sum \alpha_j^2=1\\x=\sum \alpha_ju_j}}\left\lVert \sum_{j=1}^\infty\alpha_j(A_mK- A K)u_j\right\rVert\\
 \le &\sup_{\substack{\sum \alpha_j^2=1\\x=\sum \alpha_ju_j}} \sum_{j=1}^M\sigma_j|\alpha_j|\left\lVert(A_m- A )v_j\right\rVert + \sup_{\substack{\sum \alpha_j^2=1\\x=\sum \alpha_ju_j}}\left\lVert (A_m-A)\sum_{j=M+1}^\infty\sigma_j\alpha_jv_j\right\rVert\\
 \le &\sigma_1\sum_{j=1}^M\lVert (A_m-A)v_j\rVert + \lVert A_m-A\rVert  \sup_{\substack{\sum \alpha_j^2=1\\x=\sum \alpha_ju_j}}\left\lVert \sum_{j=M+1}^\infty\sigma_j\alpha_jv_j\right\rVert\\
 &\le  \sigma_1\sum_{j=1}^M\lVert (A_m-A)v_j\rVert + 2C \sigma_{M+1}\le \sigma_1\sum_{j=1}^M\lVert (A_m-A)v_j\rVert + \varepsilon/2
\end{align*}

 Because $A_m\to A$ pointwise there exists an $m_0\in\mathbb{N}$ such that $\sigma_1\sum_{j=1}^M\lVert (A_m-A)v_j\rVert\le \varepsilon/2$ for all $m\ge m_0$, thus $A_mK\to AK$ and therefore $K^*A_mK\to K^*AK$ for $m\to\infty$ uniformly. Since $K^*P_m^*P_mK$ is compact, self-adjoint and positive semidefinite, so is $K^*AK$ as its uniform limit. Then (1.) holds by Section 6 of \cite{babuvska1991eigenvalue}.
We define iteratively $I_1:=\{ j~:~ \lambda_j^{(\infty)}=\lambda_1^{(\infty)}\}$, $I_i:=\{ j~:~\lambda_j^{(\infty)}=\lambda_{\max(I_{i-1})+1}\}$. So 
the cardinality of $I_i$ is the algebraic multiplicity of the $i$-th largest eigenvalue of $K^*AK$. We define the corresponding eigenspaces
$E_i:=span\left( v_j^{(\infty)}~,~ j\in I_i\right)$, $E_i^m:=span\left( v_j^{(m)}~,~j\in I_i\right)$. With $P_{E_i},P_{E_i^m}$ the orthogonal projections onto $E_i$ and $E_i^m$,
by Theorem 7.1 of \cite{babuvska1991eigenvalue}, there
exists a constant $C_i$ such that $\|P_{E_i^m}-P_{E_i}\|\le C_i \lVert K^*P_m^*P_mK-K^*AK\rVert$ (for $m$ sufficiently large). Thus there exists a $M\in\N$ with $M=\sum_{i=1}^{i^*}| I_i|$ for some $i^*\in\mathbb{N}$ such that
\begin{align*}
\left|\sum_{j=1}^M\left(\hat{x},v_j^{(m)}\right)^2-\sum_{j=1}^M\left(\hat{x},v_j^{(\infty)}\right)^2\right|&\le \sum_{i=1}^{i^*}\left|\lVert P_{E_i^m}\hat{x} \rVert^2 - \lVert P_{E_i}\hat{x}\rVert^2\right|\\
 &\le \sum_{i=1}^{i^*}\left(|\|P_{E_i^m}\hat{x}\| + \|P_{E_i}\hat{x}\|\right)\left| \|P_{E_i^m}\hat{x}\|-\|P_{E_i}\hat{x}\|\right|\\
&\le  2\|\hat{x}\|\sum_{i=1}^{i^*}\|P_{E_i}^m \hat{x}-P_{E_i}\hat{x}\|\\
&\le 2\|\hat{x}\|^2 \|K^*P_m^*P_mK-K^*AK\|\sum_{i=1}^{i^*} C_i \le  \varepsilon/2
\end{align*}

for $m$ sufficiently large and

\begin{equation*}
\left|\lVert \hat{x}\rVert^2- \sum_{j=1}^M\left(\hat{x},v_j^{(\infty)}\right)^2\right|= \sum_{j=M+1}^\infty (\hat{x},v_j^{(\infty)})^2 \le \varepsilon/2,
\end{equation*}
 where the second assertion followed from the injectivity of $K^*AK$. Thus

\begin{align*}
\sum_{j=M+1}^m \left(\hat{x},v_j^{(m)}\right)^2 &\le \left\lVert P_{\left(P_mK\right)^\perp}\hat{x}\right\rVert^2 - \sum_{j=1}^M\left(\hat{x},v_j^{(m)}\right)^2\\
   &\le  \lVert \hat{x}\rVert^2 - \sum_{j=1}^M\left(\hat{x},v_j^{(\infty)}\right)^2+\sum_{j=1}^M\left(\hat{x},v_j^{(m)}\right)^2-\sum_{j=1}^M\left(\hat{x},v_j^{(\infty)}\right)^2\le \varepsilon
\end{align*}

 for $m$ sufficiently large. 
\end{proof}

\subsubsection{Proof of Proposition \ref{fdr:intr}}
It holds that $\sup_{m\in\N} \|P_m y\| = \sup_{m\in\N} \sum_{j=1}^m l_j(y)^2 <\infty$, thus $\sup_m \|P_m\| < \infty$ and with the embedding $\R^m\subset l^2(\N)$ it follows that $\lim P_m y = P_\infty y$ with $P_\infty y =\begin{pmatrix} l_1(y) & l_2(y) & ...\end{pmatrix}$. Thus $P_m^*P_my \to A y$ with $A=P_\infty^*P_\infty$ and $A$ is injective because of the completeness condition.

\subsubsection{Proof of Proposition \ref{fdr:box}}

Since smooth functions are dense in $L^2$, it suffices to consider the case where $y$ is smooth.
 We have that $P_m^*P_m = P_m^+P_m=P_{\mathcal{N}(P_m)^\perp}$ and $\mathcal{N}(P_m)^\perp:=\{ \sum_{j=1}^m \alpha_j \Lambda_j\m\}$ is the set of all functions constant on a homogeneous grid with $m$ elements. Since the set of all functions constant on a homogeneous grid is dense in the set of smooth functions, the claim follows.

\subsubsection{Proof of Proposition \ref{fdr:hat}}

As above w.l.o.g. $y$ is assumed to be smooth. We denote by $A_m\in\R^{m\times m}$ the matrix representing $P_m: \mathcal{N}(P_m)^\perp\to \R^m$ with respect to the bases $(\eta_j\m)_{j=1,...,m}\subset {\mathcal{N}(P_m)^\perp}$ and $(e_j)_{j=1,...,m}\subset \R^m$, where the latter is the canonical basis of $\R^m$. So

\begin{equation*}
P_m^*P_m\eta_j\m = \sum_{i=1}^m \left( A_m^*A_m\right)_{ij} \eta_i\m,
\end{equation*}

and

\begin{equation*}
(A_m)_{ij} = \left(P_m \eta_i\m, e_j\right)_{\R^m} = l_j\m(\eta_i\m)=(\eta_j\m,\eta_i\m)_\mathcal{Y}
\end{equation*}

with
 
\begin{equation*}
(\eta_j\m,\eta_i\m) = \begin{cases} 2/3 & ,i=j\\
                                    1/3 & , |i-j|=1, \min(i,j)=1 \mbox{ or } \max(i,j)=m\\
                                    1/6 & , |i-j|=1, \min(i,j)>1 \mbox{ and } \max(i,j)<m\\
                                    0 &, else. \end{cases}
\end{equation*}

So it holds that 

\begin{equation*}
\|P_m\| \le \sqrt{\|P_m\|_1\|P_m\|_\infty} = \max_{j=1,...,m} \sum_{i=1}^m |(A_m)_{ij}| = \frac{7}{6},
\end{equation*}

where $\| . \|, \|.\|_1$ and $\|.\|_\infty$ are the spectral, the maximum absolute column and row norm respectively, and

\begin{equation*}
P_m^*P_m \eta_j\m = \frac{\eta_{j-2}\m}{36} + \frac{2\eta_{j-1}\m}{9}+\frac{\eta_j\m}{2} + \frac{2\eta_{j+1}\m}{9} + \frac{\eta_{j+2}\m}{36},
 \end{equation*}
 
 for $j=4,...,m-3$. Denote by $y_m = \sum_{j=1}^m y\left(\frac{j-1}{m-1}\right)\eta_j\m\sqrt{\frac{3}{2(m-1)}}$ the interpolating spline of $y$, then

\begin{align*}
& \left\|y-P_m^*P_m y \right\|\\
  \le& \left\| y_m - P_m^*P_my_m \right\| + \left\|(I - P_m^*P_m)(y-y_m)\right\|\\
                   \le& \left\| \sum_{j=1}^m y\left(\frac{j-1}{m-1}\right)\sqrt{\frac{3}{2(m-1)}}\left(I_m-P_m^*P_m\right) \eta_j\m \right\| + 2 \| y - y_m \|\\
                   \le& 2\|yg-y_m\|+6\sup_t |y(t)|\sqrt{\frac{3}{2(m-1)}}\left(1+\frac{7^2}{6^2}\right)+\\
                       \bigg\|&\left. \sum_{j=4}^{m-3}\left(\frac{y\left(\frac{j}{m-1}\right)}{2} - \frac{2y\left(\frac{j+1}{m-1}\right)}{9} - \frac{2y\left(\frac{j-1}{m-1}\right)}{9} - \frac{y\left(\frac{j+2}{m-1}\right)}{36} - \frac{y\left(\frac{j-2}{m-1}\right)}{36}\right)\frac{\sqrt{3}\eta_j\m}{\sqrt{2(m-1)}}\right\|\\
                     \le& 2\|y-y_m\|+30\sup_t |y(t)|\frac{1}{\sqrt{m-1}}\\
                     + & \sup_{j\le m}\left|\frac{y\left(\frac{j}{m-1}\right)}{2} - \frac{2y\left(\frac{j+1}{m-1}\right)}{9} - \frac{2y\left(\frac{j-1}{m-1}\right)}{9} - \frac{y\left(\frac{j+2}{m-1}\right)}{36} - \frac{y\left(\frac{j-2}{m-1}\right)}{36}\right|\\
                     *& \left\|\sum_{j=4}^{m-3}\frac{\sqrt{3}\eta_j\m}{\sqrt{2(m-1)}}\right\|\\
                     &\le 2\|y-y_m\| + 30\sup_{t\in(0,1)}|y(t)|\frac{1}{\sqrt{m-1}} + \sup_{t\in(0,1)}\left|y'(t)\right|\frac{3}{m}\to 0
\end{align*}
 as $m\to\infty$.

\subsubsection{Proof of Theorem \ref{th1}}

We will need the following proposition for the convergence proofs.

\begin{proposition}\label{proofs:prop1}
 Assume that Assumption \ref{disc:fdr} is fulfilled. Then, $P_{\mathcal{N}(P_mK)} x \to 0$ as $m\to\infty$ for all $x\in\mathcal{X}$.
\end{proposition}
\begin{proof}
 We assume w.l.o.g. that $x_m:=P_{\mathcal{N}(P_mK)}x \rightharpoonup z \in \mathcal{X}$ for $m\to\infty$ (weakly). Then $\lim_{m\to\infty}Kx_m = Kz$. Thus

\begin{align*}
\|AKz\|&=\limsup_{m\to\infty}\| P_m^*P_m K z\| =\limsup_{m\to\infty} \|P_m^*P_m(Kz-Kx_m)\|\\
 &\le\limsup_{m\to\infty} \|P_m\|^2\|Kz-Kx_m\| = 0,
 \end{align*}

so $AKz=0$ hence by injectivity $z=0$. In particular, $\lk P_{\mathcal{N}(P_mK)}v_i,v_i\rk  \to \lk 0,v_i\rk= 0$ for $m\to\infty$ and $i\in\N$ (set $x=v_i$ the $i$-th singular vector of $K$), so 

$$ 1\ge \| P_{\mathcal{N}(P_mK)}v_i - v_i \|^2 = \|P_{\mathcal{N}(P_mK)}v_i\|^2 - 2(P_{\mathcal{N}(P_mK)}v_i,v_i) + 1$$ 

and therefore

$$\limsup_{m\to\infty}\|P_{\mathcal{N}(P_mK)}v_i\| = 0.$$

Finally, by injectivity of $K$, for $\varepsilon>0$ there exists a $M\in\N$ with $\sum_{j=M+1}^\infty (x,v_j)^2 \le \varepsilon$, so

$$\limsup_{m\to\infty}\|P_{\mathcal{N}(P_mK)}x\|^2 \le \sum_{j=1}^M (x,v_j)^2 \limsup_{m\to\infty}\|P_{\mathcal{N}(P_mK)} v_j\|^2 + \varepsilon  = \varepsilon$$

and the claim follows with $\varepsilon\to 0$.
\end{proof}

We come to the main proof and split
\begin{align*}
&\E\left\lVert R_{\alpha(\delta_{m,n}^{est})}\m\bar{Y}_n\m-K^+\hat{y}\right\rVert^2\\
\le&\left\lVert K^+\hat{y}-R_{\alpha(\delta_{m,n}^{est})}\m P_m\hat{y}\right\rVert^2+\E\left\lVert R_{\alpha(\delta_{m,n}^{est})}\m P_m\hat{y}-R_{\alpha(\delta_{m,n}^{est})}\m\bar{Y}_n\m\right\rVert^2\\
\le& \left\lVert K^+\hat{y}-R_{\alpha(\delta_{m,n}^{est})}\m P_m\hat{y}\right\rVert^2+\left\lVert R_{\alpha(\delta_{m,n}^{est})}\m\right\rVert^2 \E \left\|\bar{Y}_n\m-P_m\hat{y}\right\|^2 
\end{align*}

 and because of independence

\begin{align*}
\E\left\lVert \bar{Y}_n\m-P_m\hat{y}\right\rVert^2 = \E \sum_{j=1}^m \left(\frac{1}{n}\sum_{i=1}^n\delta_{ij}\m\right)^2 = \frac{1}{n}\sum_{j=1}^m\E{\delta_{1j}\m}^2\le \frac{m}{n}C_d=\left(\delta_{m,n}^{est}\right)^2 C_d.
\end{align*}

Assumption \ref{assfilt} implies that

\begin{equation}\label{regunorm}
\|R_\alpha\| \le \sqrt{C_RC_F/\alpha},
\end{equation}

see e.g. \cite{engl1996regularization} or Proposition 1 of \cite{harrach2020beyond}. Therefore
 it follows that 

\begin{equation}\label{pr1eq1}
\left\lVert R_{\alpha(\delta_{m,n}^{est})}\m\right\rVert^2 \E {\delta_m^{meas}}^2 \le \left( \left\| R_{\alpha(\delta_{m,n}^{est})}\m \right\| \delta_{m,n}^{est}\right)^2C_d \le C_dC_RC_F \frac{{\delta_{m,n}^{est}}^2}{\alpha(\delta_{m,n}^{est})} \to0
\end{equation}

 for $m,n\to\infty, m/n \to 0$. Now

\begin{align*}
&\left\lVert K^+\hat{y}-R_{\alpha(\delta_{m,n}^{est})}\m P_m\hat{y}\right\rVert\\
\le& \left\lVert K^+\hat{y}-\left(P_mK\right)^+P_m\hat{y}\right\rVert + \left\lVert \left(P_mK\right)^+P_m\hat{y}-R_{\alpha(\delta_{m,n}^{est})}\m P_m\hat{y}\right\rVert\\
                                               = &\left\lVert K^+K\hat{x}-\left(P_mK\right)^+P_mK\hat{x}\right\rVert + \left\lVert \left(P_mK\right)^+P_m\hat{y}-R_{\alpha(\delta_{m,n}^{est})}\m P_m\hat{y}\right\rVert\\
                                               =& \left\lVert \hat{x}-P_{\mathcal{N}(P_mK)^\perp}\hat{x}\right\rVert + \left\lVert \left(P_mK\right)^+P_m\hat{y}-R_{\alpha(\delta_{m,n}^{est})}\m P_m\hat{y}\right\rVert
\end{align*}

and
 
 \begin{equation}\label{pr1eq2}
\lim_{m\to\infty} \|\hat{x} - P_{\mathcal{N}(P_mK)^\perp}\hat{x}\| = \lim_{m\to\infty}\|P_{\mathcal{N}(P_mK)}\hat{x}\|= 0
\end{equation} 
   by Proposition \ref{proofs:prop1}. Finally, for any $\varepsilon>0$ by Lemma \ref{proof:lem1}.2 there exists a $M\in\N$ such that $\sum_{j=M+1}^m\left(\hat{x},v_j\m\right)^2\le \varepsilon$ for $m$ large enough and therefore

\begin{align*}
&\left\lVert \left(P_mK\right)^+P_mK\hat{x} - R_{\alpha(\delta_{m,n}^{est})}\m P_mK\hat{x}\right\rVert^2\\
 = &\sum_{j=1}^m \left|1-F_{\alpha(\delta_{m,n}^{est})}({\sigma_j\m}^2){\sigma_j\m}^2\right|^2\left(\hat{x},v_j^{(m)}\right)^2\\
 \le &\sum_{j=1}^M \left|1-F_{\alpha(\delta_{m,n}^{est})}({\sigma_j\m}^2){\sigma_j\m}^2\right|^2\left(\hat{x},v_j^{(m)}\right)^2+\sum_{j=M+1}^m \left(\hat{x},v_j^{(m)}\right)^2\\
 \le &\|\hat{x}\|^2\sup_{j=1,...,M} \left|1-F_{\alpha(\delta_{m,n}^{est})}({\sigma_j\m}^2){\sigma_j\m}^2\right|^2 + \varepsilon.
\end{align*}

 By Lemma \ref{proof:lem1}.1, \eqref{assfilt:1} and since $\alpha(\delta_{m,n}^{est})\to 0$ for $m,n\to\infty, m/n\to 0$ 
\begin{equation*}
\sup_{j=1,...,M} \left|1-F_{\alpha(\delta_{m,n}^{est})}({\sigma_j\m}^2){\sigma_j\m}^2\right| \le \sup_{\frac{{\sigma_M^{(\infty)}}^2}{2}\le \lambda\le\|K\|^2}\left|1-F_{\alpha(\delta_{m,n}^{est})}(\lambda)\lambda)\right|\le \frac{\sqrt{\varepsilon}}{\|\hat{x}\|}
\end{equation*}

 for all  $m,n$ sufficiently large and $m/n$ sufficiently small. Thus with $\varepsilon\to 0$ it follows that 

\begin{equation*}
\lim_{\substack{m,n\to\infty\\ m/n\to 0}}\left\lVert \left(P_mK\right)^+P_mK\hat{x} - R_{\alpha(\delta_{m,n}^{est})}\m P_mK\hat{x}\right\rVert = 0,
\end{equation*}
 which concludes the proof together with \eqref{pr1eq1} and \eqref{pr1eq2}.

\subsubsection{Proof of Theorem \ref{th2}}

By the nature of white noise we cannot expect the error to concentrate along a certain direction, in contrast to \cite{harrach2020beyond}. However, the independence between the measurement channels implies that its amplitude is highly concentrated. First, the following Proposition affirms that we are estimating the variance correctly.

\begin{proposition}\label{sec:5prop1}
Assume that the error fulfills Assumption \ref{err:disc}. Then for the sample variance

 $$s_{m,n}^2 = \frac{1}{m} \sum_{j=1}^m \frac{1}{n-1} \sum_{i=1}^n\left( Y_{ij}\m - \frac{1}{n} \sum_{l=1}^m Y_{lj}\m\right)^2$$
 there holds

$$\lim_{m\to\infty}\mathbb{P}\left( \sup_{n\ge 2}\left|s_{m,n}^2 - \E {\delta_{11}\m}^2\right| \ge \varepsilon \E {\delta_{11}\m}^2\right) = 0$$

for all $\varepsilon>0$.

\end{proposition}

\begin{proof}
As a sum of $m$ reversed martingales $\left(s_{m,-n}^2-\E{\delta_{11}\m}^2\right)_{n\le -2}$ is a reversed martingale adapted to the filtration $$\mathcal{F}_{-n}=\sigma\left( \sum_{i=1}^n(\delta_{i1}\m - \overline{\delta_{i1}\m})^2, ..., \sum_{i=1}^n(\delta_{im}\m-\overline{\delta_{im}\m})^2\right), n\ge 2.$$ Under Assumption \ref{err:disc}.2, by the Kolmogorov-Doob-inequalities there holds

\begin{equation*}
\mathbb{P}\left( \sup_{n\ge 2} \left| s_{m,n}^2-\E{\delta_{11}\m}^2\right|\ge \varepsilon \E {\delta_{11}\m}^2\right) \le \frac{\E\left|s_{m,2}^2-\E {\delta_{11}\m}^2\right|^{p}}{\left(\varepsilon\E{\delta_{11}\m}^2\right)^p}.
\end{equation*}

By Marcinkiewicz-Zygmund inequality \cite{gut2013probability} there exists $C_p$ such that

\begin{align*}
\E\left|s_{m,2}^2-\E {\delta_{11}\m}^2\right|^{p} &=\E\left|\frac{1}{m}\sum_{j=1}^m\sum_{i=1}^2 (\del-\overline{\del})^2-\E {\delta_{11}\m}^2\right|^{p}\\
&\le \frac{C_p}{m^{p-1}}\E\left|\sum_{i=1}^2 (\delta_{i1}\m-\overline{\delta_{i1}}\m)^2-\E{\delta_{11}\m}^2\right|^{p}\\
&\le \frac{2^{p-1}(4^p+1)C_p}{m^{p-1}}\E|\delta_{11}\m|^{2p},
\end{align*}

so

\begin{equation*}
\mathbb{P}\left( \sup_{n\ge 2} \left| s_{m,n}^2-\E{\delta_{11}\m}^2\right|\ge \varepsilon \E {\delta_{11}\m}^2\right) \le \frac{\E\left|s_{m,2}^2-\E {\delta_{11}\m}^2\right|^{p}}{(\E{\delta_{11}\m}^2)^p} \le \frac{2^{p-1}(4^p+1)C_p C_d}{\varepsilon^pm^{p-1}}\to 0
\end{equation*}

as $m\to\infty$. Under Assumption \ref{err:disc}.1, by the Kolmogorov-Doob-inequality

\begin{equation*}
\mathbb{P}\left( \sup_{n\ge 2} \left| s_{m,n}^2-\E{\delta_{11}\m}^2\right|\ge \varepsilon \E {\delta_{11}\m}^2\right) \le \frac{\E\left|s_{m,2}^2-\E {\delta_{11}^{(1)}}^2\right|}{\varepsilon\E{\delta_{11}^{(1)}}^2}.
\end{equation*}

It holds that

$$s_{m,2}^2-\E{\delta_{11}^{(1)}}^2 = \frac{1}{m}\sum_{j=1}^m\sum_{i=1}^2 \left(\del-\overline{\del}\right)^2-\E {\delta_{11}^{(1)}}^2 =:\frac{1}{m}\sum_{j=1}^m X_{j}\m$$

with $X_{j}\m, j=1,...,m, m\in\N$ are i.i.d and $\E X_{j}\m=0, \E |X_{j}\m|<\infty$. To finish the proof we need to show that $\E |\sum_j X_m/m|\to 0$ as $m\to\infty$. Let $\varepsilon'>0$. By dominated convergence and integrability of $X_{j}\m$ there exists $M>0$ large enough such that for $Y_j\m:= X_j\m \chi_{\left\{|X_j\m| \le M\right\}}$ and $Z_j\m:=X_j\m \chi_{\left\{|X_j\m|> M \right\}}$ it holds that $ \E|Z_1^{(1)}| \le \varepsilon$.
 So, since $X_j\m$ are i.i.d.
\begin{align}\label{chapt2:prop:eq}
\E \left|\sum_{j=1}^m X_{j}\m \right| &\le \E \left|\sum_{j=1}^m Y_{j}\m-\E Y_{j}\m\right| + \E \left|\sum_{j=1}^m Z_{j}\m -\E Z_{j}\m\right|\\\notag
 &\le \sqrt{ \E\left|\sum_{j=1}^m Y_{j}\m-\E Y_{j}\m\right|^2} + \sum_{j=1}^m\E\left|Z_{j}\m-\E Z_{j}\m\right|\\\notag
 &\le \sqrt{ m\E\left| Y_{1}^{(1)}-\E Y_{1}^{(1)}\right|^2}+2m\E|Z_{1}^{(1)}|\le \sqrt{m2M\E|X_{1}^{(1)}|} + 2m\varepsilon,
\end{align}

thus $\E \left|\sum_{j=1}^m X_{j}\m/m\right| \le 3 \varepsilon$ for $m$ large enough.
\end{proof}

Now we need the following Lemma.

\begin{lemma}\label{estlem1}
Assume that the error model is accordingly to Assumption \ref{err:disc}. Then there holds

$$\lim_{m,n\to\infty}\mathbb{P}\left( \left|\frac{\left\| \bar{Y}_n\m-P_m\hat{y}\right\| - \delta_{m,n}^{est}}{\delta_{m,n}^{est}}\right| \ge \varepsilon\right) = 0.$$

\end{lemma}

\begin{proof}
It holds that

\begin{align*}
&\frac{\left\|\bar{Y}_n\m-P_m\hat{y}\right\| -\delta_{m,n}^{est}}{\delta_{m,n}^{est}}\\
  = &\sqrt{\frac{\E{\delta_{11}\m}^2}{s_{m,n}^2}}\left( \frac{\left\|\bar{Y}_n\m - P_m\hat{y}\right\| - \sqrt{m \E{\delta_{11}\m}^2/n}}{\sqrt{ m\E{\delta_{11}\m}^2/n}} + 1 - \sqrt{\frac{s_{m,n}^2}{\E{\delta_{11}\m}^2}}\right).
\end{align*}

Thus by Proposition \ref{sec:5prop1} it suffices to show that 

$$\lim_{m,n\to \infty}\mathbb{P}\left( \left| \frac{\left\| \bar{Y}_n\m - P_m\hat{y}\right\|^2 - \frac{m}{n}\E {\delta_{11}\m}^2 }{\frac{m}{n}\E{\delta_{11}\m}^2}\right| \ge \varepsilon\right) = 0.$$

Let us first assume that Assumption \ref{err:disc}.1 holds true. Then, by Markov's inequality

\begin{align*}
\mathbb{P}\left(\left|\frac{\left\|\bar{Y}_n\m - P_m\hat{y}\right\|^2 - \frac{m}{n}\E{\delta_{11}\m}^2}{\frac{m}{n}\E{\delta_{11}\m}^2}\right| \ge \varepsilon\right)&\le \frac{\E\left|\left\|\bar{Y}_n\m - P_m\hat{y}\right\|^2 - \frac{m}{n}\E{\delta_{11}\m}^2\right|}{\varepsilon\frac{m}{n}\E{\delta_{11}\m}^2}\\
 &= \frac{1}{m\varepsilon}\E\left|\sum_{j=1}^m\left(\left(\frac{\sum_{i=1}^n \del}{\sqrt{n \E{\delta_{11}\m}^2}}\right)^2 - 1\right)\right|.
\end{align*}

Now with

$$X_{jn}\m:=\left(\frac{\sum_{i=1}^n \del}{\sqrt{n \E{\delta_{11}\m}^2}}\right)^2 - 1$$

it holds that $(X_{jn}\m)_{j=1}, j=1,...,m, m\in\N$ are i.i.d and $\E X_{jn}\m = 0, \E | X_{jn}\m|=2<\infty$.  We proceed similarly as at the end of the proof of Proposition \ref{sec:5prop1} and show that

\begin{equation}\label{chap2:lem:eq0}
\E\left|\frac{1}{m}\sum_{j=1}^m X_{jn}\m\right| \to 0
\end{equation}

as $m\to\infty$ (uniformly in $n\in\N$), where we face an additional technical difficulty due to the dependence on $n$. Let $\varepsilon>0$ and $Z$ be a standard Gaussian (thus $\E Z^2=1$ in particular). Then for $M$ large enough it holds that

\begin{align}\label{chapt2:lem:eq1}
  \E\left[ \chi_{\left\{|Z^2-1|\ge M\right\}}\right] &\le \frac{\varepsilon}{4}\\
  \E\left[ Z^2  \chi_{\left\{ |Z^2-1|< M\right\}}\right] &\ge \E[Z^2] -   \frac{\varepsilon}{4} = 1-\frac{\varepsilon}{4}.
\end{align}

By the standard central limit theorem for real valued random variables it holds that

$$\frac{\sum_{i=1}^n \delta_{i1}^{(1)}}{\sqrt{n \E\left[ {\delta_{11}^{(1)}}^2\right]}} \to Z$$

weakly as $n\to\infty$. Since 

\begin{align*}
f_1:\R\to\R,~x&\mapsto \chi_{\left\{|x^2-1|\ge M\right\}},\\
f_2:\R\to\R,~x&\mapsto x^2\chi_{\left\{|x^2-1|<M\right\}}\\
\end{align*}

are bounded functions whose set of discontinuities has Lebesgue measure $0$ it holds that

$$\E\left[f_p\left(\frac{\sum_{i=1}^n \delta_{i1}^{(1)}}{\sqrt{n\E\left[{\delta_{11}^{(1)}}^2\right]}}\right)\right] \to \E\left[f_p(Z)\right]$$

as $n\to\infty$ for $p=1,2$ by Portmanteaus lemma (see e.g. \cite{klenke2013probability}). Thus by \eqref{chapt2:lem:eq1} there exists a $n^*$ such that

\begin{align*}
\E\left[f_1\left(\frac{\sum_{i=1}^n \delta_{i1}^{(1)}}{\sqrt{n\E\left[{\delta_{11}^{(1)}}^2\right]}}\right)\right] &\le  \E\left[f_1(Z)\right]+ \left|\E\left[f_1\left(\frac{\sum_{i=1}^n \delta_{i1}^{(1)}}{\sqrt{n\E\left[{\delta_{11}^{(1)}}^2\right]}}\right)-f_1(Z)\right]\right|\le \frac{\varepsilon}{2}\\
\E\left[f_2\left(\frac{\sum_{i=1}^n \delta_{i1}^{(1)}}{\sqrt{n\E\left[{\delta_{11}^{(1)}}^2\right]}}\right)\right] &\ge  \E\left[f_2(Z)\right]- \left|\E\left[f_2\left(\frac{\sum_{i=1}^n \delta_{i1}^{(1)}}{\sqrt{n\E\left[{\delta_{11}^{(1)}}^2\right]}}\right)-f_2(Z)\right]\right|\ge  1-\frac{\varepsilon}{2}
\end{align*}

for all $n\ge n^*$ and $p=1,2$. We again set $Y_{jn}\m:= X_{jn}\m \chi_{\left\{|X_{jn}\m| \le M\right\}}$ and $Z_{jn}\m:=X_{jn}\m \chi_{\left\{|X_{jn}\m|> M \right\}}$ and define

$$f_3:=\R\to\R,~x\mapsto x^2\chi_{\left\{|x^2-1|\ge M\right|}.$$

Then

\begin{align}\label{chapt2:lem:eq3}
\E| Z_{1n}^{(n)}| &\le \E\left[f_3\left( \frac{\sum_{i=1}^n \delta_{i1}^{(1)}}{\sqrt{n\E\left[{\delta_{11}^{(1)}}^2\right]}}\right)\right] + \E\left[f_1\left( \frac{\sum_{i=1}^n \delta_{i1}^{(1)}}{\sqrt{n\E\left[{\delta_{11}^{(1)}}^2\right]}}\right)\right]\\\notag
 &= \E\left[\left(\frac{\sum_{i=1}^n \delta_{i1}^{(1)}}{\sqrt{n\E\left[{\delta_{11}^{(1)}}^2\right]}}\right)^2\right] - \E\left[f_2\left( \frac{\sum_{i=1}^n \delta_{i1}^{(1)}}{\sqrt{n\E\left[{\delta_{11}^{(1)}}^2\right]}}\right)\right]\\\notag
 &\qquad + \E\left[f_1\left( \frac{\sum_{i=1}^n \delta_{i1}^{(1)}}{\sqrt{n\E\left[{\delta_{11}^{(1)}}^2\right]}}\right)\right]\\\notag
 &\le 1 -(1-\frac{\varepsilon}{2})+\frac{\varepsilon}{2} = \varepsilon,
\end{align}

for all $n\ge n^*$, where we used that $f_2(x)=f_3(x)=x^2$ in the second step. With the same argumentation as in \eqref{chapt2:prop:eq},

\begin{align*}
\E \left|\sum_{j=1}^m X_{jn}\m \right| &\le \E \left|\sum_{j=1}^m Y_{jn}\m-\E Y_{jn}\m\right| + \E \left|\sum_{j=1}^m Z_{jn}\m -\E Z_{jn}\m\right|\\
 &\le \sqrt{ \E\left|\sum_{j=1}^m Y_{jn}\m-\E Y_{jn}\m\right|^2} + \sum_{j=1}^m\E\left|Z_{jn}\m-\E Z_{jn}\m\right|\\
 &\le \sqrt{ m\E\left| Y_{1n}^{(1)}-\E Y_{1n}^{(1)}\right|^2}+2m\E|Z_{1n}^{(1)}|\le \sqrt{m2M\E|X_{1n}^{(1)}|} + 2m\varepsilon\\
 &\le \sqrt{4mM}+2m\varepsilon,
\end{align*}

for all $n\ge n^*$, where we used that $\E|X_{1n}^{(1)}|\le 2$ and \eqref{chapt2:lem:eq3} in the last step. Thus $\E|\sum_{j=1}^m X_{jm}\m/m| \le 3\varepsilon$ for $m,n$ large enough and sending $\varepsilon$ to $0$ proves the claim \eqref{chap2:lem:eq0}.  

Now assume that Assumption \ref{err:disc}.2 holds true. Then, by Markov's inequality

$$\mathbb{P}\left( \left| \frac{\left\| \bar{Y}_n\m - P_m\hat{y}\right\|^2 - \frac{m}{n}\E {\delta_{11}\m}^2 }{\frac{m}{n}\E{\delta_{11}\m}^2}\right| \ge \varepsilon\right) \le \frac{\E\left| \frac{n}{m\E{\delta_{11}\m}^2}\left\|\bar{Y}_n\m - P_m\hat{y}\right\|^2 -  1\right|^p}{\varepsilon^p}$$
and using further twice the Marcinkiewicz-Zygmund inequality one obtains

\begin{align*}
 &\E\left| \frac{n}{m\E{\delta_{11}\m}^2}\left\|\bar{Y}_n\m - P_m\hat{y}\right\|^2 -  1\right|^p\\
  = &\frac{1}{m^p}\E\left| \sum_{j=1}^m\left( \left( \sum_{i=1}^n \delta_{ij}\m/\sqrt{n\E{\delta_{11}\m}^2}\right)^2-1\right)\right|^p\\
  \le& \frac{B_p m^{\max(1,p/2)}}{m^{p}}\E\left|\left(\sum_{i=1}^n \delta_{i1}\m/\sqrt{n\E{\delta_{11}\m}^2}\right)^2-1\right|^p\\
\le& \frac{2^{p-1} B_p}{m^{\min(p-1,p/2)}}\left(\E\left|\sum_{i=1}^n \delta_{i1}\m/\sqrt{n\E{\delta_{11}\m}^2}\right|^{2p} + 1^p\right)\\
 \le& \frac{2^{p-1} B_p}{m^{\min(p-1,p/2)}}\left(B_{2p}\E\left|\delta_{11}\m\right|^{2p}/\left(\E{\delta_{11}\m}^2\right)^p + 1\right)\le\frac{C}{m^{\min(p-1,p/2)}}\to 0
\end{align*}

as $m\to\infty$, where we have used independence and $\E \left(\sum_{i=1}^n \delta_{ij}\m/\sqrt{n \E {\delta_{11}\m}^2}\right)^2 =1$ in the second step.

\end{proof}

Before we will start with the main proof we need one last proposition.

\begin{proposition}\label{proof:prop1}
For all $\varepsilon>0$, there exist $m_0\in\N$ and $\alpha_0>0$ such that

$$\lim_{m\to\infty} \left\|P_mK R_{\alpha}\m P_mK\hat{x} - P_mK\hat{x}\right\|/\sqrt{\alpha} \le \varepsilon$$

for all $m\ge m_0$ and $\alpha\le \alpha_0$.

\end{proposition}

\begin{proof}
Lemma \ref{proof:lem1}.2 guarantees the existence of $M\in\N$ such that $$C_1^2 \sum_{j=M+1}^m(\hat{x},v_j\m)^2\le \varepsilon/2$$ for $m$ sufficiently large. Then

\begin{align*}
& \left\lVert (P_mKR\m_{\alpha}-Id)P_mK\hat{x}\right\rVert^2/\alpha = \sum_{j=1}^m  \left(F_{\alpha}({\sigma_j\m}^2){\sigma_j\m}^2-1\right)^2 \frac{{\sigma_j\m}^2}{\alpha}(\hat{x},v_j\m)^2\\
  \le &\left(\sup_{\lambda>0} \lambda^\frac{\nu_0}{2}|F_{\alpha}(\lambda)\lambda-1|\right)^2\lVert \hat{x}\rVert^2\sum_{j=1}^M \frac{{\sigma_j\m}^{2(1-\nu_0)}}{\alpha}\\
  &\qquad + \left(\sup_{\lambda>0} \lambda^\frac{1}{2}|F_{\alpha}(\lambda)\lambda-1|\right)^2\frac{\sum_{j=M+1}^m (\hat{x},{v_j\m})^2}{\alpha}\\
  \le &C_{\nu_0}^2M {\sigma_M\m}^{2(1-\nu_0)}\lVert \hat{x}\rVert^2\alpha^{\nu_0-1}+C_1^2 \sum_{l=M+1}^m (\hat{x},v_j\m)^2\\
  \le &2C_{\nu_0}^2M {\sigma_M^{(\infty)}}^{2(1-\nu_0)}\lVert \hat{x}\rVert^2\alpha^{\nu_0-1}+\varepsilon/2 \le \varepsilon
\end{align*}

 for $m$ sufficiently large and $\alpha$ sufficiently small, where we have used that the qualification of $\left(F_{\alpha}\right)_{\alpha>0}$ is bigger than one in the third and Lemma \ref{proof:lem1}.1 in the fourth step.
\end{proof}

We start with the main proof. We define

$$\Omega_{m,n}:=\left\{ \left\| \bar{Y}_n\m-P_m\hat{y}\right\|\le \frac{\tau+C_0}{2C_0}\delta_{m,n}^{est}~,~\delta_{m,n}^{est} \le c \varepsilon\right\},$$

with $c \le \frac{1}{2}\max\left\{ \frac{C_0+3\tau}{{\sigma_M^{(\infty)}}^2},\frac{(\tau+C_0)) \sqrt{C_RC_F}}{\sqrt{\varepsilon'}}\right\}^{-1}$, where $\varepsilon'$ is given below.

By Proposition \ref{proofs:prop1},

$$\left\|(P_mK)^+P_m\hat{y}-K^+\hat{y}\right\| = \left\|P_{\mathcal{N}(P_mK)}\hat{x}\right\|\le \varepsilon$$

for $m$ large enough and by Lemma \ref{proof:lem1}.2

\begin{align*}
&\left\|R_{\alpha_{m,n}}\m P_m\hat{y}-K^+\hat{y}\right\|^2\\
\le &\sum_{j=1}^M\left|F_{\alpha_{m,n}}({\sigma_j\m}^2){\sigma_j\m}^2 - 1\right|^2(\hat{x},v_j\m)^2 + \sum_{j=M+1}^m (\hat{x},v_j\m)^2\\
 \le &\frac{1}{{\sigma_M\m}^2}\sum_{j=1}^M\left|F_{\alpha_{m,n}}({\sigma_j\m}^2){\sigma_j\m}^2 - 1 \right|^2 {\sigma_j\m}^2(\hat{x},v_j\m)^2 + \varepsilon/2\\
 =&\frac{1}{{\sigma_M\m}^2}\left\|(P_mKR_{\alpha_{m,n}}\m-Id)P_m\hat{y}\right\|+\varepsilon/2\\
 \le&\frac{1}{{\sigma_M\m}^2}\left(\left\|(P_mKR_{\alpha_{m,n}}\m-Id)\bar{Y}_n\m\right\|+ \left\|(P_mKR_{\alpha_{m,n}}\m-Id)(P_m\hat{y}-\bar{Y}_n\m)\right\|\right)+\varepsilon/2
\end{align*}

for $m$ sufficiently large. So Lemma \ref{proof:lem1}.1 and the defining relation of the discrepancy principle and of $\Omega_{m,n}$ ensure that

\begin{align*}
\left\|R_{\alpha_{m,n}}\m P_m\hat{y}-K^+\hat{y}\right\|\chi_{\Omega_{m,n}} &\le \frac{2}{{\sigma_M^{(\infty)}}^2}\left(\tau \delta_{m,n}^{est}+ C_0\frac{\tau+C_0}{2C_0}\delta_{m,n}^{est} \right) \chi_{\Omega_{m,n}}+\varepsilon/2 \le \varepsilon
\end{align*}

for $m$ sufficiently large. Moreover,

\begin{align*}
&\tau \delta_{m,n}^{est} \chi_{\Omega_{m,n}}\\
\le &\left\|(P_mKR_{\alpha_{m,n}/q}\m - Id)\bar{Y}_n\m\right\|\chi_{\Omega_{m,n}}\\
 \le &\left\|(P_mK R_{\alpha_{m,n}/q}\m-Id)P_m\hat{y}\right\| + \left\|(P_mKR_{\alpha_{m,n}/q}\m-Id)(\bar{Y}_n\m - P_m\hat{y})\right\|\chi_{\Omega_{m,n}}\\
 \le &\left\|(P_mK R_{\alpha_{m,n}/q}\m-Id)P_m\hat{y}\right\| +C_0\frac{\tau+C_0}{2C_0}\delta_{m,n}^{est}\chi_{\Omega_{m,n}},\\
 \Longrightarrow & \delta_{m,n}^{est} \chi_{\Omega_{m,n}} \le \frac{2}{\tau-C_0}\left\|(P_mKR_{\alpha_{m,n}/q}\m-Id)P_m\hat{y}\right\|
\end{align*}

 Proposition \ref{proof:prop1} guarantees the existence of $\varepsilon'$ such that for $m$ large enough
\begin{align*}
\left\|P_mKR_{\alpha}\m P_m\hat{y}-P_m\hat{y}\right\|/\sqrt{\alpha} &\le \frac{(\tau-C_0)qC_0}{(\tau+C_0)\sqrt{C_RC_F}}\frac{\varepsilon}{2}
\end{align*}

for all $\alpha\le  \varepsilon'/q$. So with \eqref{regunorm}

\begin{align*}
&\|R_{\alpha_{m,n}}\m(\bar{Y}_n\m-P_m\hat{y})\|\chi_{\Omega_{m,n}}\\
&\le \|R_{\alpha_{m,n}}\|\|\bar{Y}_n\m-P_m\hat{y}\| \chi_{\Omega_{m,n}} \le \sqrt{\frac{C_RC_R}{\alpha_{m,n}}} \frac{\tau+C_0}{2C_0} \delta_{m,n}^{est}\chi_{\Omega_{m,n}}\\
\le & \frac{(\tau+C_0)\sqrt{C_RC_F}}{2C_0}\left(\frac{\delta_{m,n}^{est}}{\sqrt{\alpha_{m,n}}}\chi_{\Omega_{m,n}\cap \{\alpha_{m,n}\le \varepsilon'\}} +  \frac{\delta_{m,n}^{est}}{\sqrt{\alpha_{m,n}}}\chi_{\Omega_{m,n}\cap\{\alpha_{m,n}\ge \varepsilon'\}}\right)\\
\le & \frac{(\tau+C_0)\sqrt{C_RC_F}}{2C_0}\left(\frac{2}{(\tau-C_0)q}\frac{\left\|(P_mKR_{\frac{\alpha_{m,n}}{q}}-Id)P_m\hat{y}\right\|}{\sqrt{\alpha_{m,n}/q}}\chi_{\{\alpha_{m,n} \le \varepsilon'\}} +\frac{\delta_{m,n}^{est}}{\sqrt{\varepsilon'}}\chi_{\Omega_{m,n}}\right)\\
\le & \frac{(\tau+C_0)\sqrt{C_RC_F}}{2C_0}\left(\frac{2}{(\tau-C_0)q} \frac{(\tau-C_0)qC_0}{(\tau+C_0)\sqrt{C_RC_F}} \frac{\varepsilon}{2}+ \frac{c\varepsilon}{\sqrt{\varepsilon'}}\right)\le  \varepsilon/2 + \varepsilon/2
\end{align*}

for $m$ large enough. Putting it all together yields

\begin{align*}
&\left\|R_{\alpha_{m,n}}^{(m)}\bar{Y}_n\m-K^+\hat{y}\right\|\chi_{\Omega_{m,n}}\\
 \le &\left\|R_{\alpha_{m,n}}^{(m)}(\bar{Y}_n\m-P_m\hat{y})\right\|\chi_{\Omega_{m,n}} + \left\|R_{\alpha_{m,n}}^{(m)}P_m\hat{y} - (P_mK)^+P_m\hat{y}\right\|\chi_{\Omega_{m,n}}\\
 &\qquad+\left\|(P_mK)^+P_m\hat{y} - K^+\hat{y}\right\|\chi_{\Omega_{m,n}}\\
 \le & 3\varepsilon
 \end{align*}

for $m$ sufficiently large, which together with $\lim_{\substack{m,n\to\infty\\m/n\to 0}}\mathbb{P}\left(\Omega_{m,n}\right)= 1$ finishes the proof.

\subsection{Proofs for infinite-dimensional residuum}

For the second approach (with infinite-dimensional residuum) we need to guarantee stable inversion of the discretisation operator $P_m$. Afterwards we will show strong concentration of the back projected measurements in $\mathcal{Y}$ in order to use classical results from deterministic regularisation theory.

\subsubsection{Proof of Proposition \ref{idr:angle}}
It holds that $\kappa(P_m) = \kappa(P_m|_{\mathcal{N}(P_m)^\perp})$. We again denote by $A_m\in\R^{m\times m}$ the matrix representing $P_m: \mathcal{N}(P_m)^\perp\to \R^m$ with respect to the bases $(\eta_j\m)_{j=1,...,m}\subset {\mathcal{N}(P_m)^\perp}$ and \newline $(e_j)_{j=1,...,m}\subset \R^m$ where the latter is the canonical basis of $\R^m$. Thus

\begin{equation*}
(A_m)_{ij} = \left(P_m \eta_i\m, e_j\right)_{\R^m} = l_j\m(\eta_i\m)=(\eta_j\m,\eta_i\m)_\mathcal{Y}.
\end{equation*}

By assumption, we have that

\begin{align*}
\left\| \frac{A_m}{\|\eta_1\m\|^2} - I_m \right\| &\le \sqrt{\left\| \frac{A_m}{\|\eta_1\m\|^2} - I_m \right\|_1 \left\| \frac{A_m}{\|\eta_1\m\|^2} - I_m \right\|_\infty}\\
& = \max_{j=1,...,m} \sum_{i\neq j} \frac{|(\eta_j\m,\eta_i\m)|}{\|\eta_1\m\|^2} =: c<1
\end{align*}

where $I_m\in\R^{m\times m}$ is the identity and $\| . \|, \| . \|_1, \| .\|_\infty$ are the spectral and the maximum absolute column or row norm. 
So by (2.3) in \cite{rump2011verified} it holds that

\begin{equation}
1-c \le \sigma_j\left(\frac{A_m}{\|\eta_1\m\|^2}\right) \le 1+ c,
\end{equation}

for $j=1,...,m$, where $\sigma_1(A),...,\sigma_m(A)$ denote the singular values of $A\in\R^{m\times m}$. This proves the proposition.

\subsubsection{Proof of Proposition \ref{idr:intr}}
The bounds $c_m,C_m$ follow directly from Proposition \ref{idr:angle}. It remains to show that $\|\hat{y}-P_m^+P_m\hat{y}\| \to 0$ as $m\to\infty$. It holds that $\mathcal{N}(P_1)\supseteq \mathcal{N}(P_2) \supseteq ...$. In particular, there exists an orthonormal basis $(w_{i})_{i\in\N}$ such that $\mathcal{N}(P_m) = span( w_{m+1},w_{m+2},...)$. Thus, $ \delta_m^{disc} = \| P_{\mathcal{N}(P_m)} y \|= \sqrt{ \sum_{j=m+1}^\infty (y,w_{j})^2} \to 0$ as $m\to\infty$. 

\subsubsection{Proof of Proposition \ref{idr:svd}}

The bound for the discretisation error follows from

\begin{align*}
\lVert \hat{y} - {P_m}^+P_m\hat{y}\rVert^2 &= \sum_{j>m} (\hat{y},u_j)^2 =  \sum_{j>m} \sigma_j^{2+2\nu} (w_,v_j)^2\le \sigma_{m+1}^{2(1+\nu)} \|w\|^2.
\end{align*}

Since $(v_j)_{j\in\N}$ is an orthonormal basis the claim follows with Proposition \ref{idr:angle}.

\subsubsection{Proof of Proposition \ref{idr:box}}
The choice $c_m=C_m=1$ follows from Proposition \ref{idr:angle} since\newline $(\eta_j\m)_{j=1,...,m}$ are orthonormal for all $m\in\N$. Denote by $y_m =\sum_{j=1}^m \hat{y}((j-1)/m)\chi_{(\frac{j-1}{m},\frac{j}{m})} \in \mathcal{R}(P_m^*)=\mathcal{N}(P_m)^\perp$ the piecewise constant interpolating spline of the continuously differentiable function $\hat{y}$. Then there holds

\begin{align*}
\|\hat{y}-P_m^+P_m \hat{y} \| &= \|\hat{y}-P_{\mathcal{N}(P_m)^\perp} \hat{y}\| \le \| \hat{y} - y_m\| \le \sqrt{ \int_0^1(\hat{y}(t)-y_m(t))^2dt}\\
& = \sqrt{\sum_{j=1}^m \int_{\frac{j-1}{m}}^\frac{j}{m}\left(\hat{y}(t)-\hat{y}\left((\frac{j-1}{m}\right)\right)^2 dt}\\
&= \sqrt{ \sum_{j=1}^m \int_{\frac{j-1}{m}}^\frac{j}{m} y'(\xi_t)\left(t-\frac{j-1}{m}\right)^2dt} \le     \frac{\sup_{t'\in(0,1)}|\hat{y}'(t')|}{m},
\end{align*}

with $\xi_t\in[\frac{j-1}{m},\frac{j}{m})$.

\subsubsection{Proof of Proposition \ref{idr:hat}}
It holds that

\begin{equation*}
(\eta_j\m,\eta_i\m) = \begin{cases} 2/3 & ,i=j\\
                                    1/3 & , |i-j|=1, \min(i,j)=1 \mbox{ or } \max(i,j)=m\\
                                    1/6 & , |i-j|=1, \min(i,j)>1 \mbox{ and } \max(i,j)<m\\
                                    0 & , else \end{cases}
\end{equation*}

Therefore

\begin{equation*}
\sup_{m\in\N} \max_{j \le m } \frac{\sum_{j\neq i}|(\eta_j\m,\eta_i\m)|}{\|\eta_1\m\|^2} = \frac{1/2}{2/3} = \frac{3}{4}
\end{equation*}

so that the bounds $c_m,C_m$ follow with Proposition \ref{idr:angle}. Let $y_m\in\mathcal{N}(P_m)^\perp$ be the interpolating spline of continuously differentiable $\hat{y}$. By the mean value theorem there exist $\xi_t,\zeta_t\in[\frac{j-1}{m-1},\frac{j}{m-1})$ such that

\begin{align*}
&\hat{y}(t) - y_m(t)\\
 = &\hat{y}\left(\frac{j-1}{m-1}\right) + \hat{y}'(\xi_t)\left(t-\frac{j-1}{m-1}\right)\\
 &\qquad  - \left(\hat{y}\left(\frac{j-1}{m-1}\right) + \left((\hat{y}\left(\frac{j}{m-1}\right)-\hat{y}\left(\frac{j-1}{m-1}\right)\right)\left((m-1)t-(j-1)\right)\right)\\
 =&(y'(\xi_t)-y'(\zeta_t)\left(t-\frac{j-1}{m-1}\right)
\end{align*}

for $t\in[\frac{j-1}{m-1},\frac{j}{m-1})$. Thus

\begin{align*}
\|\hat{y}-P_m^+P_m\hat{y}\| & \le \|\hat{y}-y_m\| \le \sqrt{\sum_{j=1}^m \int_{\frac{j-1}{m-1}}^\frac{j}{m-1} \left(\hat{y}'(\xi_t)-\hat{y}'(\zeta_t)\right)^2\left(t-\frac{j-1}{m-1}\right)^2dt}\\
&\le \frac{2\sqrt{m}\sup_{t\in(0,1)}|\hat{y}'(t)|}{(m-1)^{3/2}}\le \frac{2^{5/2}\sup_{t'\in(0,1)}|\hat{y}'(t')|}{m}
\end{align*}

If $\hat{y}$ is twice continuously differentiable, then there are $\xi_t',\zeta_t'\in(\frac{j-1}{m-1},\frac{j}{m-1}]$ such that

\begin{align*}
  |\hat{y}'(\xi_t)-\hat{y}'(\zeta_t)| &= \left|\hat{y}''(\xi'_t)\left(\xi_t-\frac{j-1}{m-1}\right) - \hat{y}''(\zeta_t')\left(\zeta_t-\frac{j-1}{m-1}\right)\right|\\
  &\le \frac{2\sup_{t'\in(0,1)}|\hat{y}''(t')|}{m-1}
\end{align*}

for $t\in[\frac{j-1}{m-1},\frac{j}{m-1})$ so that

\begin{align*}
\|\hat{y} - P_m^+P_m\hat{y}\| &\le \|\hat{y}-y_m\|\le \sqrt{\sum_{j=1}^m \int_{\frac{j-1}{m-1}}^\frac{j}{m-1}\left(\frac{2\sup_{t'\in(0,1)}|\hat{y}''(t')|}{m-1}\right)^2\left(t-\frac{j-1}{m-1}\right)^2dt}\\
 &\le \frac{2\sqrt{m}\sup_{t'\in(0,1)}|\hat{y}''(t')|}{(m-1)^{5/2}}\le \frac{2^{7/2}\sup_{t'\in(0,1)}|\hat{y}''(t')|}{m^2}.
\end{align*}

\subsubsection{Proof of Theorem \ref{th1b}}
We use the bias-variance decomposition

\begin{align*}
&\E\left\|R_{\alpha(\delta_m^{disc})}P_m^+\bar{Y}_{n(m,\delta_m^{disc})}\m-K^+\hat{y}\right\|^2\\
 =&\E\left\|R_{\alpha(\delta_m^{disc})}P_m^+(\bar{Y}_{n(m,\delta_m^{disc})}\m-P_m\hat{y})\right\|^2+\left\|R_{\alpha(\delta_m^{disc})}P_m^+P_m\hat{y}-K^+\hat{y}\right\|^2\\
 \le&\E\left\|R_{\alpha(\delta_m^{disc})}P_m^+(\bar{Y}_{n(m,\delta_m^{disc})}\m-P_m\hat{y})\right\|^2+2\left\|R_{\alpha(\delta_m^{disc})}P_m^+P_m\hat{y}-R_{\alpha(\delta_m^{disc})}\hat{y}\right\|^2\\
 &\qquad+2\left\|R_{\alpha(\delta_m^{disc})}\hat{y}-K^+\hat{y}\right\|^2\\
 \le &\left\|R_{\alpha(\delta_m^{disc})}\right\|^2\left(\|P_m^+\|^2\E \left\|P_m\hat{y}-\bar{Y}_{n(m,\delta_m^{disc})}\m\right\|^2 + 2\left\|P_m^+P_m\hat{y}-\hat{y}\right\|^2\right)\\
 &\qquad+2\left\|R_{\alpha(\delta_m^{disc})}\hat{y}-K^+\hat{y}\right\|^2\\
 \le & \frac{C_RC_F}{\alpha(\delta_m^{disc})}\left( \frac{\E{\delta_{11}\m}^2m}{c_m^2 n(m,\delta_m^{disc})}+2 {\delta_m^{disc}}^2\right) + 2\left\|R_{\alpha(\delta_m^{disc})}\hat{y}-K^+\hat{y}\right\|^2\\
 \le & \left(C_RC_F(C_d+2)\right)\frac{{\delta_m^{disc}}^2}{\alpha(\delta_m^{disc})}  + 2\left\|R_{\alpha(\delta_m^{disc})}\hat{y}-K^+\hat{y}\right\|^2\to 0
\end{align*}

as $m\to\infty$.

\subsubsection{Proof of Theorem \ref{th4}}

The proof of Theorem \ref{th4} is more technical than the one of Theorem \ref{th2} due to correlations coming from the back projecting of the measurements and the data-dependent determination of the stopping index $n(m,\delta_m^{disc})$. However, under slightly stronger conditions we obtain a similar concentration property of the measurement error.

\begin{lemma}\label{estlem2}
Assume that the discretisation fulfills Assumption \ref{disc:fdr}  and the error is accordingly to Assumption \ref{err:disc} with $p\ge 2$ in the case of Assumption \ref{err:disc}.2. For $m\in\N, \delta_0,\delta>0$ and the sample variance 

$$s_{m,n}^2:=\frac{1}{m}\sum_{j=1}^m \frac{1}{n-1}\sum_{i=1}^n\lk Y_{ij}\m - \frac{1}{n}\sum_{l=1}^n Y_{lj}\m\rk^2,$$

consider the (random) choice

$$n(m,\delta)= \min\left\{ n'\ge 1: \frac{ms_m^2(n')}{c_m^2n'}\le \delta^2\right\}$$
 with $\sigma_1\m,...,\sigma_m\m$ the singular values of $P_m$. Then for any $\varepsilon>0$ there holds

$$\lim_{m\to\infty}\sup_{0<\delta\le \delta_0}\mathbb{P}\left(\left|\frac{\left\| P_m^+\bar{Y}_{n(m,\delta)}\m-P_m^+P_m\hat{y}\right\| - \delta_m}{\delta_m}\right| \ge \varepsilon\right)=0$$

 with $\bar{Y}_{n(m,\delta)}\m=\frac{1}{n(m,\delta)}\sum_{i=1}^{n(m,\delta)}\begin{pmatrix} Y_{i1}\m & ... & Y_{im}\m\end{pmatrix}^T$ and $\delta_m:= \delta \sqrt{\sum_{j=1}^m\frac{c_m^2}{m{\sigma_j^{m}}^2}}$. 

\end{lemma}

\begin{proof}
The auxiliary parameter $\delta_m$ has to be introduced due to the fact that with the choice of $n(m,\delta)$ we are actually overestimating $\E\left\| P_m^+\bar{Y}_{n(m,\delta)}\m - P_m^+P_m\hat{y}\right\|^2$ since $c_m\le \sigma_j\m$. We define 

\begin{equation*}
 \mu_m^\delta:=\frac{m\E [{\delta_{11}\m}^2]}{c_m^2\delta^2}
\end{equation*}

\begin{equation*}
I_\varepsilon(m,\delta):=\left[(1-\varepsilon)\mu_m^\delta,(1+\varepsilon)\mu_m^\delta\right].
\end{equation*}

$$\delta_{m,n}^{meas}:=\|P_m^+\bar{Y}_n\m - P_m^+P_m\hat{y}\| =  \sqrt{\sum_{j=1}^m \lambda_j\m \left( \sum_{l=1}^m \sum_{i=1}^n \frac{\delta_{ij}\m}{n} (u_j\m,e_l\m)\right)^2}$$

where $\lambda_j\m = {\sigma_j\m}^{-2}$ and $(u_j\m)_{j \le m}, (e_j\m)_{j \le m} \subset \R^m$ are the singular basis of $P_m$ (fulfilling \newline $P_m{P_m}^* u_j\m = {\sigma_j\m}^2 u_j\m$) and the canonical basis of $\R^m$ respectively. So

$$\E {\delta_{m,n}^{meas}}^2 = \sum_{j=1}^m \lambda_j \E \left(\sum_{l=1}^m \sum_{i=1}^n \frac{\delta_{il}\m}{n} (u_j\m,e_l\m)\right)^2  = \frac{\E {\delta_{11}\m}^2}{n}\sum_{j=1}^m \lambda_j$$

and

\begin{align*}
\mathbb{P}\left( \left|\frac{{\delta_{m,n(m,\delta)}^{meas}}^2 - \delta_m^2}{\delta_m^2}\right| \le \varepsilon\right)
&\ge \mathbb{P}\left( \left|\frac{{\delta_{m,n(m,\delta)}^{meas}}^2-\delta_m^2}{\delta_m^2}\right|\le \varepsilon, n(m,\delta) \in I_{\varepsilon'}\right)\\
&\ge \mathbb{P}\left( \sup_{n\in I_{\varepsilon'}} \left| \frac{{\delta_{m,n}^{meas}}^2-\delta_m^2}{\delta_m^2}\right|\le \varepsilon, n(m,\delta)\in I_{\varepsilon'}\right)\\
&\ge 1 - \mathbb{P}\left(\sup_{n\in I_{\varepsilon'}}\left|\frac{{\delta_{m,n}^{meas}}^2-\delta_m^2}{\delta_m^2}\right| >\varepsilon\right) - \mathbb{P}\left( n(m,\delta)\notin I_{\varepsilon'}\right).
\end{align*}

 Since 

\begin{equation*}
\left|\frac{{\delta_{m,n}^{meas}}^2-\delta_m^2}{\delta_m^2}\right| \le \left(\left| \frac{{\delta_{m,n}^{meas}}^2-\E {\delta_{m,n}^{meas}}^2}{\E {\delta_{m,n}^{meas}}^2}\right| + \left| \frac{\E {\delta_{m,n}^{meas}}^2-\delta_m^2}{\E {\delta_{m,n}^{meas}}^2}\right|\right) \frac{\E {\delta_{m,n}^{meas}}^2}{\delta_m^2}
\end{equation*}

and 

\begin{equation*}
\sup_{n\in I_{\varepsilon'}}\left| \frac{\E {\delta_{m,n}^{meas}}^2-\delta_m^2}{\delta_m^2}\right| =\frac{\varepsilon'}{1-\varepsilon'},\quad \sup_{n\in I_{\varepsilon'}}\frac{\E {\delta_{m,n}^{meas}}^2}{\delta_m^2} = \frac{1}{1-\varepsilon'},
\end{equation*}

 we conclude that for $\varepsilon'=\frac{3}{16}\varepsilon\le 1/4$

\begin{align}\notag
&\mathbb{P}\left( \left| \frac{{\delta_{m,n(m,\delta)}^{meas}}^2-\delta_m^2}{\delta_m^2}\right|\le \varepsilon\right)\\\notag
  \ge& 1 - \mathbb{P}\left( \sup_{n\in I_{\varepsilon'}}\left|\frac{{\delta_{m,n}^{meas}}^2-\E {\delta_{m,n}^{meas}}^2}{\E {\delta_{m,n}^{meas}}^2}\right|>\varepsilon(1-\varepsilon') - \frac{\varepsilon'}{1-\varepsilon'}\right) - \mathbb{P}\left( n(m,\delta) \notin I_{\varepsilon'}\right)\\\label{eq:th4:proof}
 \ge& 1 - \mathbb{P}\left( \sup_{n\in I_{\frac{3}{16}\varepsilon}}\left| \frac{{\delta_{m,n}^{meas}}^2-\E {\delta_{m,n}^{meas}}^2}{\E {\delta_{m,n}^{meas}}^2}\right|> \varepsilon/2\right)-\mathbb{P}\left(n(m,\delta)\notin I_{\frac{3}{16}\varepsilon}\right).
\end{align}

Thus it remains to show that the both terms with negative sign tend to zero.
\begin{proposition}\label{sec:3:helpprop1}
For every $\varepsilon>0$ there holds

$$\sup_{\delta_0\ge\delta>0} \mathbb{P}\left(n(m,\delta) \in I_{\varepsilon}(m,\delta)\right) \to 1$$

for $m\to\infty$.
\end{proposition}

\begin{proof}
For $m$ large enough it holds that $\lfloor(1+\varepsilon) \mu_m^\delta\rfloor\ge(1+\varepsilon/2)\mu_m^\delta$ and

\begin{align*}
& \{n(m,\delta) \in I_\varepsilon(m,\delta)\} = \left\{  \left| n(m,\delta) - \mu_m^\delta\right|\le \varepsilon \mu_m^\delta \right\}\\
\supseteq &\left\{\frac{ms_{m,n}^2}{c_m^2n}> \delta^2~,~ \forall ~ n <(1-\varepsilon)\mu_m^\delta\right\}\\
&\qquad \cap \left\{ \frac{ms_{m,n}^2}{c_m^2n}\le \delta^2~,~ \mbox{for } n =\lfloor (1+\varepsilon)\mu_m^\delta\rfloor\right\}\\
= &\left\{ {ms_{n,m}^2}>\frac{n}{\mu_m^\delta}~,~\forall n<(1-\varepsilon)\mu_m^\delta\right\} \cap \left\{ s_{n,m}^2 \le \frac{n}{\mu_m^\delta}~,~\mbox{for } n= \lfloor(1+\varepsilon)\mu_m^\delta\rfloor\right\}\\
\supseteq &\left\{ |{s_{n,m}^2} - \E [{\delta_{11}\m}^2]| \le \varepsilon/2 \E[{ \delta_{11}\m}^2]~,~\forall n\ge 2\right\},
\end{align*}

and the claim follows by Proposition \ref{sec:5prop1}.

\end{proof}

For the first term in \eqref{eq:th4:proof} we will need the following proposition.
\begin{proposition}\label{sec:4:helpprop}
For $(X_{l})_{l\in\N}$ i.i.d. with $\E X_l =0$, $\E X_l^2 =1$ and $\E X_l^4<\infty$ and\newline $(u_j)_{j\le m}, (e_j)_{j\le m} \subset \mathbb{R}^m$ orthonormal bases and $(\lambda_j)_{j \le m} \in \R^+$, it holds that

$$\E\left| \sum_{j=1}^m \lambda_j \left(\left(\sum_{l=1}^m X_l (u_j,e_l)\right)^2 -1\right)\right|^2 \le  \max_{j \le m} \lambda_j^2 (\E X_1^4 + 5) m .$$

\end{proposition}
\begin{proof}
By Jensen's inequality

\begin{align*}
&\left(\E\left[\left| \sum_{j=1}^m \lambda_j \left(\left(\sum_{l=1}^m X_l (u_j,e_l)\right)^2 -1\right)\right|\right]\right)^2\\
 \le &\E\left[\left| \sum_{j=1}^m \lambda_j \left(\left(\sum_{l=1}^m X_l (u_j,e_l)\right)^2 -1\right)\right|^2\right]\\
= & \sum_{j,j'=1}^m\lambda_j \lambda_{j'} \left(\E\left[\left(\sum_{l=1}^m X_l(u_j,e_l)\right)^2\left(\sum_{l'=1}^m X_{l'} (u_{j'},e_{l'})\right)^2\right]\right.\\
&\quad \left. - 2\E\left[ \left(\sum_{l=1}^m X_l (u_j,e_l)\right)^2\right] +1 \right)\\
 = & \sum_{j,j'=1}^m \lambda_j\lambda_{j'}\left( \sum_{l,l',l'',l'''=1}^m \E \left[X_lX_{l'}X_{l''}X_{l'''}\right](u_j,e_l)(u_j,e_{l'})(u_{j'},e_{l''})(u_{j'},e_{l'''})\right.\\
 &\quad\left. +2 \left(\E[X_1]^2\right)^2 - 1\right)\\
 = & \sum_{j,j'=1}^m \lambda_j\lambda_{j'} \left( \E X_1^4 \sum_{l=1}^m (u_j,e_l)^2(u_{j'},e_l)^2 + \left(\E[X_1^2]\right)^2\sum_{\substack{l,l'=1\\ l\neq l'}}^m (u_j,e_l)^2(u_{j'},e_{l'})^2 \right.\\
 &\quad \left. + 2 \left(\E[X_1^2]\right)^2\sum_{\substack{l,l'=1\\l\neq l'}}^m (u_j,e_l)(u_j,e_{l'})(u_{j'},e_l)(u_{j'},e_{l'}) - 1\right).
\end{align*}

With

 $$\sum_{\substack{l'=1\\l'\neq l}}^m (u_{j'},e_{l'})^2 = 1 - (u_{j'},e_l)^2$$
  and 
$$\sum_{\substack{l'=1\\ l'\neq l}}^m (u_j,e_{l'})(u_{j'},e_{l'}) = (u_j,u_{j'})- (u_j,e_l)(u_{j'},e_{l})$$

we further deduce that

\begin{align*} 
&\left(\E\left[\left| \sum_{j=1}^m \lambda_j \left(\left(\sum_{l=1}^m X_l (u_j,e_l)\right)^2 -1\right)\right|\right]\right)^2\\
  = & \sum_{j,j'=1}^m\lambda_j\lambda_{j'} \left( \E X_1^4 \sum_{l=1}^m (u_j,e_l)^2(u_{j'},e_l)^2 + \sum_{l=1}^m (u_j,e_l)^2(1-(u_{j'},e_l)^2) \right.\\
  &\quad \left. + 2 \sum_{l=1}^m (u_j,e_l)(u_{j'},e_l)\left((u_j,u_{j'})-(u_j,e_l)(u_{j'},e_l)\right) - 1\right)\\
    = & \sum_{j,j'=1}^m\lambda_j\lambda_{j'} \left( \E X_1^4 \sum_{l=1}^m (u_j,e_l)^2(u_{j'},e_l)^2 + 1 -\sum_{l=1}^m (u_j,e_l)^2(u_{j'},e_l)^2) \right.\\
  &\quad \left. + 2 \left((u_j,u_{j'})^2-\sum_l(u_j,e_l)^2(u_{j'},e_l)^2\right) - 1\right)\\
  &\le \max_{j\le m} \lambda_j^2  \left( \sum_{l=1}^m \sum_{j,j'=1}^m|\E X_1^4 -3|(u_j,e_l)^2(u_{j'},e_l)^2 + 2 \sum_{j,j'=1}^m (u_j,u_{j'})^2\right)\\
  &\le \max_{j\le m} \lambda_j^2 (\E X_1^4+5) m,
\end{align*}

\end{proof}

Finally, it holds that

\begin{align*}
M_n\m&:=  n\frac{{\delta_{m,n}^{meas}}^2 - \E {\delta_{m,n}^{meas}}^2}{\E {\delta_{m,n}^{meas}}^2}\\
 &= n\frac{\sum_{j=1}^m \lambda_j \left( \sum_{l=1}^m \sum_{i=1}^n \frac{\delta_{il}\m}{\sqrt{n}} (u_j\m,e_l\m)\right)^2 - \E{\delta_{11}\m}^2 \sum_{j=1}^m \lambda_j}{\E{\delta_{11}\m}^2 \sum_{j=1}^m \lambda_j}\\
& = \frac{n}{\sum_{j'=1}^m \lambda_{j'}} \sum_{j=1}^m \lambda_j \left( \left( \sum_{l=1}^m \sum_{i=1}^n \frac{\delta_{il}\m}{\sqrt{n\E{\delta_{11}\m}^2}} (u_j\m,e_l\m)\right)^2 - 1 \right).
\end{align*}

It is easy to verify that $(M_n\m)_{n\in\N}$ is a martingale adapted to the filtration $(\mathcal{F}_n)_{n\in\mathbb{N}}$ generated by the measurement errors $\mathcal{F}_n:=\sigma\left( \delta_{ij}\m~,~i\le n,j\le m\right)$ for every fixed $m\in\mathbb{N}$. 
Now assume that Assumption \ref{err:disc}.2 with $p\ge 2$ holds true. With $n_-:=(1+\frac{3}{16}\varepsilon)\mu_m^{\delta}, n_+:=(1+\frac{3}{16}\varepsilon)\mu_m^\delta$ we obtain via the Kolmogorov-Doob-inequality

\begin{align*}
\mathbb{P}\left(\sup_{n\in I_\frac{3}{16}\varepsilon}\left| \frac{{\delta_{m,n}^{meas}}^2-\E {\delta_{m,n}^{meas}}^2}{\E {\delta_{m,n}^{meas}}^2}\right|\ge \frac{\varepsilon}{2}\right)& = \mathbb{P}\left(n_-\sup_{n\in I_\frac{3}{16}\varepsilon}\left| \frac{{\delta_{m,n}^{meas}}^2-\E {\delta_{m,n}^{meas}}^2}{\E {\delta_{m,n}^{meas}}^2}\right|\ge \frac{n_-\varepsilon}{2}\right)\\
 &\le\mathbb{P}\left( \sup_{n\in I_{\frac{3}{16}\varepsilon}} |M_n\m| \ge \frac{n_-\varepsilon}{2}\right)
\le \frac{4\E\left[ {M_{n_+}\m}^2\right]}{\varepsilon^2n_-^2}.
\end{align*}

With $X_l:=\sum_i \delta_{ij}\m/\sqrt{n \E {\delta_{ij}\m}^2}$ Proposition \ref{sec:4:helpprop} yields

\begin{align*}
\frac{4\E\left[ {M_{n_+}\m}^2\right]}{n_-^2\varepsilon^2} &= \frac{4n_+^2}{n_-^2\varepsilon^2} \frac{\max_{j\le m}\lambda_j^2(\E X_1^4 + 5)m}{(\sum_j \lambda_j)^2}\\
& = \frac{4n_+^2}{\varepsilon^2n_-^2} \frac{\max_{j\le m} \sigma_j^{-4}}{\min_{j \le m} \sigma_j^{-4}} \left(\frac{\E {\delta_{11}\m}^4}{n_+ (\E {\delta_{11}\m}^2)^2} + 3 \frac{n_+-1}{n_+} + 5\right)\frac{1}{m}\\
& = \frac{n_+^2}{\varepsilon^2n_-^2} \kappa(P_m)^4 \left(\frac{C_d}{n_+} + 3 \frac{n_+-1}{n_+} + 5\right)\frac{1}{m}\to 0
\end{align*}
 as $m\to\infty$. In the following we write $u_j$ and $e_j$ for $u_j\m$ and $e_j\m$. Under Assumption \ref{err:disc}.1, the Kolmogorov-Doob-inequality yields

$$\mathbb{P}\left( \sup_{n\in I_{\frac{3}{16}\varepsilon}}\left|\frac{{\delta_{m,n}^{meas}}^2-\E {\delta_{m,n}^{meas}}^2}{\E {\delta_{m,n}^{meas}}^2}\right| \ge \frac{\varepsilon}{2}\right) \le \frac{\E{\left|M_{n_+}\m\right|}}{\varepsilon n_-}.$$

We set $S_m:=\frac{M_{n_+}\m}{n_+}\sum_{j=1}^m \lambda_j$ and $Z_l\m:=\sum_{i=1}^n \delta_{il}\m/\sqrt{n_+\E {\delta_{11}\m}^2}$ ($Z_l\m, j=1,...,m, m\in\N$ are i.i.d.). For $K>0$ we truncate

\begin{align*}
V_l\m:&= Z_l\m \chi_{\{|Z_l\m|\le K\}} - \E\left[  Z_l\m \chi_{\{|Z_l\m|\le K\}}\right]\\
 W_l\m:&= Z_l\m \chi_{\{|Z_l\m|> K\}} - \E\left[  Z_l\m \chi_{\{|Z_l\m|> K\}}\right].
\end{align*} 
  Then $\E V_l\m = \E W_l\m = 0=V_l\m W_l\m$ and therefore

\begin{align*}
&\E |S_m| \\
 = &\E\left| \sum_{j=1}^m \lambda_j\left( \left(\sum_{l=1}^m Z_l\m(u_j,e_l)\right)^2 -1 \right)\right|\\
     \le &\E \left| \sum_{j=1}^m  \lambda_j \left( \left(\sum_{l=1}^m V_l\m (u_j,e_l)\right)^2 - \E \left[{V_1^{(1)}}^2\right]\right)\right| + \E\left| \sum_{j=1}^m \lambda_j\left(\sum_{l=1}^m W_l\m (u_j,e_l)\right)^2 \right|\\
     &\qquad + 2\E \left| \sum_{j=1}^m \lambda_j \sum_{\substack{l,l'=1\\ l \neq l'}}^m V_l\m W_{l'}\m (u_j,e_l)(u_j,e_{l'})\right|+ \left|1-\E \left[{V_1^{(1)}}\right]^2\right|\sum_{j=1}^m \lambda_j.
\end{align*}

Since $\E \left[{V_1^{(1)}}^4\right]<\infty$, by Proposition \ref{sec:4:helpprop} above and Jensen's inequality

\begin{align*}
&\E \left| \sum_{j=1}^m  \lambda_j \left( \left(\sum_{l=1}^m V_l\m (u_j,e_l)\right)^2 - \E \left[{V_1^{(1)}}^2\right]\right)\right|\\
 \le& \sqrt{\E \left| \sum_{j=1}^m  \lambda_j \left( \left(\sum_{l=1}^m V_l\m (u_j,e_l)\right)^2 - \E \left[{V_1^{(1)}}^2\right]\right)\right|^2}\\
 \le& \|P_m^+\|^2 \sqrt{\E\left[ {V_1^{(1)}}^4\right] + 5}\sqrt{m}.
\end{align*}

For the second term

\begin{align*}
\E\left| \sum_{j=1}^m \lambda_j\left(\sum_{l=1}^m W_l\m (u_j,e_l)\right)^2 \right| &\le\E\left| \|P_m^+\|^2  \sum_{l,l'=1}^m W_l\m W_{l'}\m \sum_{j=1}^m(u_j,e_l)(u_j,e_{l'})\right|\\
& = \|P_m^+\|^2 \E\left|\sum_{l,l'=1}^m W_l\m W_{l'}\m (e_l,e_{l'})\right|\\
& = m\|P_m^+\|^2\E\left[{ W_1^{(1)}}^2\right]. 
\end{align*}

For the third term we calculate the variance

\begin{align*}
&\left(\E \left| \sum_{j=1}^m \lambda_j \sum_{\substack{l,l'=1\\ l \neq l'}}^m V_l\m W_{l'}\m (u_j,e_l)(u_j,e_{l'})\right|\right)^2\\
\le&\E \left| \sum_{j=1}^m \lambda_j \sum_{\substack{l,l'=1\\ l \neq l'}}^m V_l\m W_{l'}\m (u_j,e_l)(u_j,e_{l'})\right|^2\\
  \le &\E \sum_{j,j'=1}^m\lambda_j\lambda_{j'} \sum_{\substack{l,l'=1 \\ l\neq l'}}^m \sum_{\substack{l'',l'''=1\\l'' \neq l'''}}^m V_l\m W_{l'}\m V_{l''}\m W_{l'''}\m (u_j,e_l)(u_j,e_{l'})(u_{j'},e_{l''})(u_{j'},e_{l'''})\\
  =  &\E\left[{V_1^{(1)}}^2\right]  \sum_{j,j'=1}^m\lambda_j\lambda_j' \sum_{\substack{l,l'=1\\l \neq l'}}^m (u_j,e_l)(u_j,e_{l'})(u_{j'},e_l)(u_{j'},e_{l'})\\
  = &\E\left[ {V_1^{(1)}}^2\right] \E\left[ {W_1^{(1)}}^2\right]\sum_{j,j'=1}^m\lambda_j\lambda_{j'} \sum_{l=1}^m (u_j,e_l)(u_{j'},e_l)\left((u_j,u_{j'})-(u_j,e_l)(u_{j'},e_{l'})\right)\\
  = &\E \left[{V_1^{(1)}}^2\right]\E \left[{W_1^{(1)}}^2\right] \sum_{j,j'=1}^m \lambda_j \lambda_{j'} \left((u_j,u_{j'})^2 - \sum_{l=1}^m (u_j,e_l)^2(u_{j'},e_l)^2\right)\\
  \le& \E \left[{V_1^{(1)}}^2\right] \E\left[{ W_1^{(1)}}^2\right] \sum_{j,j'=1}^m \lambda_j \lambda_{j'} (u_j,u_{j'})^2 = \E \left[{V_1^{(1)}}^2\right] \E\left[{ W_1^{(1)}}^2\right] \sum_{j=1}^m \lambda_j^2\\
   \le& \E\left[ {V_1^{(1)}}^2\right] \E\left[{ W_1^{(1)}}^2\right] \|P_m^+\|^4 m.
\end{align*}

Altogether,

\begin{align*}
 &\frac{2\E |M_{n_+}\m|}{ \varepsilon n_-}\\
  \le&  \frac{2n_+}{ \varepsilon n_-} \frac{1}{\sum_j \lambda_j}\E|S_m|\\
  \le& \frac{2n_+}{\varepsilon n_-} \frac{1}{\sum_j \lambda_j}\left(\|P_m^+\|^2\sqrt{\E\left[{ V_1^{(1)}}^4\right] + 5}\sqrt{m} + m\|P_m^+||^2 \E\left[ {W_1^{(1)}}^2\right]  \right.\\
  &\qquad \left.+\sqrt{\E\left[{ V_1^{(1)}}^2\right]\E\left[{ W_1^{(1)}}^2\right]} \|P_m^+\|^2 \sqrt{m}  |1-\E V^2| m\|P_m^+\|^2\right)\\
  \le& \frac{2n_+\kappa(P_m)^2}{\varepsilon n_-\sqrt{m}}\left(\sqrt{\E \left[{V_1^{(1)}}^4\right]+5}+\sqrt{\E\left[{ V_1^{(1)}}^2\right] \E\left[{ W_1^{(1)}}^2\right]}\right)\\
  &\qquad +\frac{2n_+\kappa(P_m)^2}{\varepsilon n_-}\left(\E\left[{ W_1^{(1)}}^2\right]+\left|1-\E\left[{ V_1^{(1)}}^2\right]\right|\right).
 \end{align*}
 
 The claim follows with $\lim_{K\to\infty}\E\left[{ V_1^{(1)}}^2\right] =1, \lim_{K\to\infty}\E\left[{ W_1^{(1)}}^2\right] = 0$ and\newline $\sup_m \kappa(P_m)^2 < \infty$.

\end{proof}

We come to the main proof

\begin{proof}
We set 

\begin{equation*}
\Omega_m:=\left\{ \left\|P_m^+\bar{Y}_{n(m,\delta_m^{disc})}\m - P_m^+P_m\hat{y}\right\| \le \frac{\tau+C_0}{2C_0} \delta_m^{disc}\right\}.
\end{equation*}

Then

\begin{align}\label{th4:proof:eq}\notag
\left\lVert P_m^+\bar{Y}_{n(m,\delta_m^{disc})}\m-\hat{y}\right\rVert \chi_{\Omega_m} &\le \left\lVert P_m^+\bar{Y}_{n(m,\delta_m^{disc})}\m-P_m^+P_m\hat{y}\right\rVert \chi_{\Omega_m}+ \left\lVert P_m^+P_m\hat{y}-\hat{y}\right\rVert\chi_{\Omega_m}\\
 &\le \frac{\tau+3C_0}{2C_0} \delta_m^{disc}.
\end{align}

By Algorithm 2 it holds that

\begin{align*}
&\alpha_m\\
:=&\left\{ q^k~,~ k \in \N_0~,~\left\|KR_{\alpha_m}P_m^+\bar{Y}_{n(m,\delta_m^{disc})}\m - P_m^+\bar{Y}_{n(m,\delta_m^{disc})}\m \right\| \le 2\tau \delta_m^{disc}\right\}\\
=&\left\{ q^k~,~ k \in \N_0~,~\left\|KR_{\alpha_m}P_m^+\bar{Y}_{n(m,\delta_m^{disc})}\m - P_m^+\bar{Y}_{n(m,\delta_m^{disc})}\m \right\| \le \frac{4\tau C_0}{\tau+3C_0}  \frac{\tau+3C_0}{2C_0}\delta_m^{disc}\right\}
\end{align*}

and because of $\frac{4\tau C_0}{\tau+3C_0}>C_0$,\eqref{th4:proof:eq} and $\lim_{m\to\infty} \delta_m^{disc} = 0$ it follows that 

\begin{equation*}
\lim_{m\to\infty} \left\|R_{\alpha_m}P_m^+\bar{Y}_{n(m,\delta_m^{disc})}\m - K^+\hat{y}\right\|\chi_{\Omega_m} = 0
\end{equation*}

by Theorem 4.17 and Remark 4.18 from \cite{engl1996regularization}. With the same reasoning it follows that there exists a $L'\in\R$ such that

\begin{equation*}
\left\|R_{\alpha_m}P_m^+\bar{Y}_{n(m,\delta_m^{disc})}\m - K^+\hat{y}\right\| \chi_{\Omega_m} \le L' \rho^\frac{1}{\nu+1} {\delta_m^{disc}}^\frac{\nu}{\nu+1}
\end{equation*}

if there are $0<\nu\le \nu_0-1$ and $w\in\mathcal{X}$ with $K^+\hat{y} =(K^*K)^{\nu/2} w$ and $\|w\|\le \rho$. Lemma \ref{estlem2} implies that $\lim_{m\to\infty} \mathbb{P}\left(\Omega_m\right)=1$, which concludes the proof.

\end{proof}

\subsection{Proofs for Section \ref{sec:discussion}}

We start with the proof of Proposition \ref{discussion:prop}.

\begin{proof}[Proof of Proposition \ref{discussion:prop}]
For simplicity we assume that $\hat{x}=0=\hat{y}$. Let $K:l^2(\N)\to l^2(\N)$ be the diagonal operator with $K v_j = \sigma_j v_j:= \alpha_j^\frac{1}{4} v_j$ (with $(v_j)_{j\in\N}$ the canonical basis) and let $P_m$ be the discretisation along the singular basis, i.e. $(P_m y)_j = l_j\m(y) = (y,v_j)$ for all $j=1,...,m$ with $m\in\N$ and $y\in l^2(\N)$.  Then

\begin{align*}
\E\|R_\alpha\m \bar{Y}_m\m - K^+\hat{y}\|^2 &= \sum_{j=1}^m \frac{\sigma_j^2}{\left(\alpha_m+\sigma_j^2\right)} \E\left(\bar{Y}_m\m - P_m\hat{y},v_j\right)^2 = \E[\delta_{11}^2]\sum_{j=1}^m \frac{(\alpha_j^\frac{1}{4})^2}{\left(\alpha_m + (\alpha_j^\frac{1}{4})^2\right)^2}\\
 &= \E[\delta_{11}^2]\sum_{j=1}^m \frac{\sqrt{\alpha_j}}{\left(\alpha_m+\sqrt{\alpha_j}\right)^2} \ge \E[\delta_{11}^2] \sum_{j=1}^m \frac{\sqrt{\alpha_m}}{\left(\alpha_m+\sqrt{\alpha_m}\right)^2}\\
  &\ge \frac{\E[\delta_{11}^2]}{\sqrt{\alpha_m}} \sum_{j=1}^m \frac{1}{\left(\sqrt{\alpha_m}+1\right)^2} \ge \frac{m\E[\delta_{11}^2]}{\alpha_m(1+\sqrt{\alpha_1})^2}\to\infty
\end{align*}

as $m\to\infty$.

\end{proof}





\begin{proof}[Proof of Theorem \ref{discussion:prop1}]
We need the following auxiliary result, which we will afterwards use to confirm that the oscillations  of the residual are too strong  when the error distribution lacks of higher moments.

\begin{lemma}\label{discproof:prop1}
For $0<\varepsilon<1$ let $(\delta_{ij})_{i,j\in\N}$ be i.i.d. with density

$$f_\varepsilon(x) =  \frac{c_\varepsilon}{|x|^{3+\varepsilon}} \chi_{\{(-\infty,-b_\varepsilon]\cup[b_\varepsilon,\infty)\}}(x),$$

where $c_\varepsilon= \frac{\varepsilon^{1+\frac{\varepsilon}{2}}}{2(2+\varepsilon)^\frac{\varepsilon}{2}}$ and $b_\varepsilon = \sqrt{\frac{\varepsilon}{2+\varepsilon}}$. Then there exist $p_\varepsilon, \kappa_\varepsilon>0$ such that 

$$\mathbb{P}\left( \sum_{j=1}^m \left( \left(\frac{\sum_{i=1}^n \delta_{ij}}{\sqrt{n}}\right)^2 - 1\right) \ge \kappa_\varepsilon m^\frac{2-\varepsilon}{2+\varepsilon}\right) \ge p_\varepsilon$$

for all $c n\le m$ with $m,n$ large enough.

\end{lemma}

\begin{proof}[Proof of Lemma \ref{discproof:prop1}]
Straight forward computations show that $f_\varepsilon$ is indeed the density of a probability distribution and that $\E[\delta_{ij}] = 0$ and $\E[ \delta_{ij}^2] = 1$ and $\mathbb{P}\left( \delta_{ij} \ge x\right) = \frac{c_\varepsilon}{2+\varepsilon}\frac{1}{x^{2+\varepsilon}} = \mathbb{P}\left(|\delta_{ij}|\ge x\right)/2$ hold for $x\ge b_\varepsilon$.

Thus we may apply Corollary 1.1.2. from \cite{vinogradov1994refined} and obtain that there exist constants $K_1,K_2,K_3$ such that for any $n\in\N$  and  all $y\ge \sqrt{K_1 n \log(n)}$ there holds

\begin{align*}
&\left|\mathbb{P}\left( \sum_{i=1}^n \delta_{i1} \ge y\right) - \frac{c_\varepsilon}{2+\varepsilon} \frac{n}{y^{2+\varepsilon}}\right|\\
 \le &K_2 \frac{n}{y^{2+\varepsilon}}\left( \left(\frac{\sqrt{n}}{y}\right)^\frac{2+\varepsilon}{3+\varepsilon}+ e^{-K_3\left(\frac{y^2}{n}\right)^\frac{1}{3+\varepsilon}}\right) + n\sup_{x\ge y/3}\left|\mathbb{P}\left(\delta_{11}>x\right) - \frac{c_\varepsilon}{2+\varepsilon} \frac{1}{x^{2+\varepsilon}}\right|\\
 = & K_2 \frac{n}{y^{2+\varepsilon}}\left( \left(\frac{\sqrt{n}}{y}\right)^\frac{2+\varepsilon}{3+\varepsilon}+ e^{-K_3\left(\frac{y^2}{n}\right)^\frac{1}{3+\varepsilon}}\right)
\end{align*}

Therefore, by symmetry of the distribution of $\delta_{ij}$ we have for all $x\ge K_2 \log(n)$

\begin{align}\notag
\mathbb{P}\left(\left(\frac{\sum_{i=1}^n \delta_{i1}}{\sqrt{n}}\right)^2 \ge x\right) &= \mathbb{P}\left(|\sum_{i=1}^n \delta_{i1}| \ge \sqrt{xn}\right) = 2 \mathbb{P}\left(\sum_{i=1}^n \delta_{i1} \ge \sqrt{xn}\right)\\\notag
&= \frac{2c_\varepsilon}{2+\varepsilon} \frac{n}{(\sqrt{nx})^{2+\varepsilon}} + \mathcal{O}\left(K_2 \frac{n}{(\sqrt{nx})^{2+\varepsilon}}\left( \left(\frac{\sqrt{n}}{\sqrt{nx}}\right)^\frac{2+\varepsilon}{3+\varepsilon}+ e^{-K_3\left(\frac{(\sqrt{nx})^2}{n}\right)^\frac{1}{3+\varepsilon}}\right)\right)\\\notag
&= \frac{2c_\varepsilon}{2+\varepsilon} \frac{1}{n^\frac{\varepsilon}{2}x^{1+\frac{\varepsilon}{2}}} + \frac{1}{n^\frac{\varepsilon}{2}x^{1+\frac{\varepsilon}{2}}}\mathcal{O}\left(x^{-\frac{2+\varepsilon}{3+\varepsilon}}+e^{- x^\frac{1}{3+\varepsilon}}\right)\\\label{discproof:eq1}
&= \frac{2c_\varepsilon}{2+\varepsilon} \frac{1}{n^\frac{\varepsilon}{2}x^{1+\frac{\varepsilon}{2}}}\left(1+\mathcal{O}(1)\right)\qquad (\mbox{for} ~~x\to\infty).
\end{align}

Let $x_{n}\ge K_2\log(n)+1$ (will be specified later)  and $X_{jn}:=\left(\frac{\sum_{i=1}^n \delta_{ij}}{\sqrt{n}}\right)^2$. We truncate and split the sum in two parts

$$\sum_{j=1}^m \left( \left(\frac{\sum_{i=1}^n \delta_{ij}}{\sqrt{n}}\right)^2 - 1\right) = \sum_{j=1}^m\left(X_{jn}-1\right)\chi_{\{ X_{jn}\le x_n\}} + \sum_{j=1}^m\left(X_{jn}-1\right)\chi_{\{X_{jn} > x_n\}}.$$

The second term contains only the extremes of the sum and we will show that here both parts will contribute to the overall sum (note that when sufficiently high moments exist (at least a fourth moment) one could show that the overall sum is dominated by the first part).
We first treat the second term. Since $(X_{jn}-1)\chi_{\{X_{jn}>x_n\}} \ge x_n-1>0$ (for $n$ sufficiently large) we have that for $t_{n,m}\ge x_n-1$

\begin{align}\notag
\mathbb{P}\left( \sum_{j=1}^m\left(X_{jn}-1\right)\chi_{\{X_{jn} > x_n\}} > t_{n,m}\right) &\ge \mathbb{P}\left(\max_{j=1,...,m}\left(X_{jn}-1\right)\chi_{\{X_{jn} > x_n\}} >t_{n,m}\right)\\\notag
  &= 1 - \mathbb{P}\left( \max_{j=1,...,m}\left(X_{jn}-1\right)\chi_{\{X_{jn} > x_n\}} \le t_{n,m}\right)\\\notag
   &= 1 - \mathbb{P}\left( \left(X_{1n}-1\right)\chi_{\{X_{1n} > x_n\}} \le t_{n,m}\right)^m= 1 - \left(1-\mathbb{P}\left( X_{1n}-1 \ge t_{n,m}\right)\right)^m\\\label{discproof:eq2}
  &= 1 - \left(1 - \frac{2c_\varepsilon}{2+\varepsilon}\frac{1}{n^\frac{\varepsilon}{2}t_{n,m}^{1+\frac{\varepsilon}{2}}}\left(1+\mathcal{O}(1)\right)\right)^m.
\end{align}

For the remaining term we need the first three moments

\begin{align*}
\mu_{n}&:=\E\left[\left(X_{jn}-1\right)\chi_{\{X_{jn}\le x_n\}}\right],\\
\sigma_{n}&:=\sqrt{\E\left[\left(\left(X_{jn}-1\right)\chi_{\{X_{jn}\le x_n\}} - \mu_{n}\right)^2\right]},\\
\rho_{n}&:=\E\left[\left|\left(X_{jn}-1\right)\chi_{\{X_{jn}\le x_n\}} - \mu_{n}\right|^3\right].\\
\end{align*}

We claim that if $x_n/\log(n)\to 0$

\begin{align}\label{discproof:eq2a}
\mu_n &= -\frac{4c_\varepsilon(1+\mathcal{O}(1))}{(2+\varepsilon)\varepsilon} \frac{1}{n^\frac{\varepsilon}{2}x_n^\frac{\varepsilon}{2}}\\\label{discproof:eq2b}
\sigma_n^2 &\ge \left(1+\mathcal{O}(1)\right)2c_\varepsilon\frac{2-\varepsilon}{2+\varepsilon} \frac{x_n^{1-\frac{\varepsilon}{2}}}{n^\frac{\varepsilon}{2}}\\\label{discproof:eq2c}
\rho_n &\le \left(1+\mathcal{O}(1)\right)3c_\varepsilon\frac{4-\varepsilon}{2+\varepsilon} \frac{x_n^{2-\frac{\varepsilon}{2}}}{n^\frac{\varepsilon}{2}},
\end{align}

for $n\to\infty$. We will prove the assertions \eqref{discproof:eq2a} - \eqref{discproof:eq2c} with \eqref{discproof:eq1} and Theorem 12.1 of \cite{gut2013probability}. Note that $\varepsilon<2$ and $X_{jn}$ is positive. For \eqref{discproof:eq2a} because of $\E[ X_{jn}-1] = 0$ there holds

\begin{align*}
\mu_n &= \E\left[\left(X_{jn}-1\right)\chi_{\{X_{jn} \le x_n\}}\right] = - \E\left[\left(X_{jn}-1\right)\chi_{\{X_{jn} > x_n\}}\right]= - \E\left[X_{jn}\chi_{\{X_{jn} > x_n\}}\right] + \E\left[\chi_{\{X_{jn} > x_n\}}\right]\\
 &= - \int_{x_n}^\infty \mathbb{P}\left(X_{jn} > t\right) dt + \mathbb{P}\left( X_{jn} > x_n\right)= \frac{2c_\varepsilon(1+\mathcal{O}(1))}{2+\varepsilon}\frac{1}{n^\frac{\varepsilon}{2}} \left(-\int_{x_n}^\infty \frac{1}{t^{1+\frac{\varepsilon}{2}}} dt + \frac{1}{x_n^{1+\frac{\varepsilon}{2}}}\right)\\
  &= -\frac{4c_\varepsilon(1+\mathcal{O}(1))}{(2+\varepsilon)\varepsilon} \frac{1}{n^\frac{\varepsilon}{2}x_n^\frac{\varepsilon}{2}}\qquad (\mbox{for}~~n\to\infty).
\end{align*}

 Further, we obtain

\begin{align*}
 \E\left[\left(X_{jn}-1\right)^2\chi_{\{X_{jn}\le x_n\}}\right] &\ge \E\left[ (X_{jn}-1)^2\chi_{\{K_1\log(n)\le X_{jn} \le x_n\}}\right]= 2 \int_{K_1\log(n)-1}^{x_n-1} t \mathbb{P}\left( X_{jn}-1 > t\right)dt\\
  &= \frac{4c_\varepsilon(1+\mathcal{O}(1))}{2+\varepsilon}\frac{1}{n^\frac{\varepsilon}{2}}\int_{K_1\log(n)}^{x_n} t^{-\frac{\varepsilon}{2}}dt = \left(1+\mathcal{O}(1)\right)2c_\varepsilon\frac{2-\varepsilon}{2+\varepsilon} \frac{x_n^{1-\frac{\varepsilon}{2}}}{n^\frac{\varepsilon}{2}}
\end{align*}

for $n\to\infty$, which together with \eqref{discproof:eq2a} yields \eqref{discproof:eq2b}. Finally, a similar reasoning proves

\begin{align*}
\E\left[\left|X_{jn}-1\right|^3\chi_{\{X_{jn}\le x_n\}}\right] &\ge \left(1+\mathcal{O}(1)\right)3c_\varepsilon\frac{4-\varepsilon}{2+\varepsilon} \frac{x_n^{2-\frac{\varepsilon}{2}}}{n^\frac{\varepsilon}{2}},
\end{align*}

which in turn implies \eqref{discproof:eq2c}.

Now

\begin{align*}
\mathbb{P}\left( \sum_{j=1}^m\left(X_{jn}-1\right)\chi_{\{X_{jn}\le x_n\}} \le t\right) &= \mathbb{P}\left( \frac{\sum_{j=1}^m\left(X_{jn}-1\right)\chi_{\{X_{jn}\le x_n\}} - \mu_n}{\sqrt{m\sigma_n^2}} \sqrt{m\sigma_n^2} + m \mu_n \le t\right)\\
&= \mathbb{P}\left(\frac{\sum_{j=1}^m\left(X_{jn}-1\right)\chi_{\{X_{jn}\le x_n\}} - \mu_n}{\sqrt{m\sigma_n}} \le \frac{t-m\mu_n}{\sqrt{m\sigma_n}}\right).
\end{align*}

By The Berry-Esseen Theorem (Theorem 6.1 in \cite{gut2013probability}), there exist $C>0$ such that

\begin{align*}
\mathbb{P}\left( \sum_{j=1}^m(X_{jn}-1)\chi_{\{X_{jn}\le x_n\}}\le t\right) &\le \Phi\left(\frac{t-m\mu_n}{\sqrt{m\sigma_n^2}}\right) + C \frac{\rho_n}{\sigma_n^3\sqrt{m}},
\end{align*}

where $\Phi$ is the cumulative distribution function of a standard Gaussian random variable. With \eqref{discproof:eq2a}-\eqref{discproof:eq2c} we see that

\begin{align*}
\frac{m\mu_n}{\sqrt{m\sigma_n^2}} &= -(1+\mathcal{O}(1))\frac{1}{\varepsilon}\sqrt{\frac{8 c_\varepsilon}{4-\varepsilon^2}} \frac{m^\frac{1}{2}}{n^\frac{\varepsilon}{4}x_n^{\frac{1}{2}+\frac{\varepsilon}{4}}},\\
\frac{\rho_n}{\sigma_n^3\sqrt{m}} &=(1+\mathcal{O}(1))3(4-\varepsilon)\sqrt{\frac{2+\varepsilon}{8c_\varepsilon(2-\varepsilon)^3}} \frac{n^\frac{\varepsilon}{4}x_n^{\frac{1}{2}+\frac{\varepsilon}{4}}}{m^\frac{1}{2}}.
\end{align*}

Therefore, for 

$$x_{n}:= \left(\frac{\eta \sqrt{m}}{n^\frac{\varepsilon}{4}}\right)^\frac{1}{\frac{1}{2}+\frac{\varepsilon}{4}} =  \eta^\frac{4}{2+\varepsilon} \frac{m^\frac{2}{2+\varepsilon}}{n^\frac{\varepsilon}{2+\varepsilon}},$$

 there exist constants $\beta_\varepsilon, \gamma_\varepsilon>0$ such that

$$\mathbb{P}\left( \sum_{j=1}^m(X_{jn}-1)\chi_{\{X_{jn}\le x_{n}\}}\le 2m\mu_n\right)\le \Phi\left(-\frac{\beta_\varepsilon}{\eta}\right) + \gamma_\varepsilon \eta$$

for all $m,n$ large enough, since with the above choice because of $n\le c m$ and $\varepsilon<1$ there holds $x_n/\log(n)\to 0$. Note that the right hand side does not depend on $m$ and $n$. There holds

$$-m\mu_n = \left(1+\mathcal{O}(1)\right) \frac{4c_\varepsilon}{\eta^\frac{2\varepsilon}{2+\varepsilon}(2+\varepsilon)\varepsilon} \frac{m^\frac{2}{2+\varepsilon}}{n^\frac{\varepsilon}{2+\varepsilon}},$$

thus for $\eta$ small we have $-m\mu_n > x_n$. Consequently, with $t_{m,n} = -3 m \mu_n$ \eqref{discproof:eq2} becomes

\begin{align*}
\mathbb{P}\left( \sum_{j=1}^m\left(X_{jn}-1\right)\chi_{\{X_{jn} > x_n\}} > t_{n,m}\right)&\ge 1 - \left(1 - \frac{2c_\varepsilon}{2+\varepsilon}\frac{1}{n^\frac{\varepsilon}{2}t_{n,m}^{1+\frac{\varepsilon}{2}}}\left(1+\mathcal{O}(1)\right)\right)^m\\
                   &=1-\left(1-\frac{2c_\varepsilon}{2+\varepsilon}\frac{1}{n^\frac{\varepsilon}{2}\left(\left(1+\mathcal{O}(1)\right)3 m^\frac{2}{2+\varepsilon}n^{-\frac{\varepsilon}{2+\varepsilon}}\eta^{-\frac{2\varepsilon}{2+\varepsilon}}\right)^{1+\frac{\varepsilon}{2}}}\right)^m\\
                   &=1-\left(1-\left(1+\mathcal{O}(1)\right)\frac{2c_\varepsilon}{2+\varepsilon}\frac{\eta^\varepsilon}{m}\right)^m\ge 1-e^{-\left(1+\mathcal{O}(1)\right) \frac{2c_\varepsilon}{2+\varepsilon} \eta^\varepsilon}\\
                   &\ge \zeta_\varepsilon \eta^\varepsilon,
\end{align*}

for some constant $\zeta_\varepsilon>0$, when $m,n$ and $\eta$ are  sufficiently large respectively small. Putting all together we obtain

\begin{align*}
&\mathbb{P}\left( \sum_{j=1}^m\left(\left(\frac{\sum_{i=1}^n \delta_{ij}}{\sqrt{n}}\right)^2-1\right)\ge -m\mu_n\right)\\
 \ge &\mathbb{P}\left( \sum_{j=1}^m \left(X_{jn}-1\right)\chi_{\{X_{jn}\le x_{n}\}} \ge 2m\mu_n,~\sum_{j=1}^m \left(X_{jn}-1\right)\chi_{\{X_{jn}>x_{n}\}}\ge -3m\mu_n\right)\\
 \ge &1-\mathbb{P}\left(\sum_{j=1}^m \left(X_{jn}-1\right)\chi_{\{X_{jn}\le x_{n}\}} < 2m\mu_n\right) - \mathbb{P}\left(\sum_{j=1}^m \left(X_{jn}-1\right)\chi_{\{X_{jn}>x_{n}\}}< -3m\mu_n\right)\\
 \ge &1-\left(\Phi\left(-\frac{\beta_\varepsilon}{\eta}\right) + \gamma_\varepsilon \eta\right) - \left(1-\zeta_\varepsilon \eta^\varepsilon\right) = \zeta_\varepsilon \eta^\varepsilon - \gamma_\varepsilon \eta - \Phi\left(-\frac{\beta_\varepsilon}{\eta}\right)\ge \zeta_\varepsilon \eta^\varepsilon - \gamma_\varepsilon \eta - e^{-\frac{\beta_\varepsilon^2}{\eta^2}} \ge p_\varepsilon
\end{align*}

for some $p_\varepsilon>0$ for $\eta=\eta_\varepsilon$ sufficiently small and fixed (since $\varepsilon<1$). Finally, the assertion follows with

$$-m\mu_n = \left(1+\mathcal{O}(1)\right) \frac{4c_\varepsilon c^\frac{\varepsilon}{2+\varepsilon}}{\eta^\frac{2\varepsilon}{2+\varepsilon}(2+\varepsilon)\varepsilon} \frac{m^\frac{2}{2+\varepsilon}}{n^\frac{\varepsilon}{2+\varepsilon}}\ge \kappa_\varepsilon m^\frac{2-\varepsilon}{2+\varepsilon},$$

which holds for some $\kappa_\varepsilon>0$ and all $n\le m/c$ with $m,n$ large enough.

\end{proof}


We come to the main proof of Theorem \ref{discussion:prop1} and first look at the case where $\tau>1$ and $\hat{x}= K^+\hat{y} = \sum_{j=1}^\infty j^{-1} v_j$. Set 

$$C:=\sqrt{\sum_{\substack{j=1\\ j^{-q} \le \frac{qc(\tau-1)^2}{\pi^2}}}^\infty j^{-2}}.$$

Now, for $\alpha>0$ we have
\begin{align*}
\left\|\left(P_mK R_{\alpha}\m-Id_{\R^m}\right)\bar{Y}_n\m\right\|&\le\left\|\left(P_mK R_{\alpha}\m-Id_{\R^m}\right)P_m\hat{y}\right\| +\left\|\left(P_mK R_{\alpha}\m-Id_{\R^m}\right)\left(\bar{Y}_n\m-P_m\hat{y}\right)\right\|\\
&\le \sqrt{\sum_{\substack{j=1\\ j^{-q}<\alpha}}^m(\hat{y},v_j)^2} + \|\bar{Y}_n\m-P_m\hat{y}\| = \sqrt{\sum_{\substack{j=1\\ j^{-q}<\alpha}}^mj^{-q-2}} + \|\bar{Y}_n\m-P_m\hat{y}\|\\
 &\le  \sqrt{\alpha \sum_{j=1}^\infty j^{-q-2}} + \|\bar{Y}_n\m-P_m\hat{y}\| \le \sqrt{\frac{\alpha\pi^2}{6}} + \|\bar{Y}_n\m-P_m\hat{y}\|.
\end{align*}

This, together with the defining relation of the discrepancy principle $$\tau \delta_{m,n} <\left\|\left(P_mK R_{\alpha_{m,n}/q}\m-Id_{\R^m}\right)\bar{Y}_n\m\right\|,$$ Lemma \ref{estlem1}, Proposition \ref{sec:5prop1} and $\E[\delta_{11}^2]=1$ ultimately yields

$$\mathbb{P}\left( \alpha_{m,n} \ge \frac{q(\tau-1)^2}{\pi^2}\frac{m}{n}\right)\to 1$$

as $m\to\infty$ (uniformly in $n\in\N$). Consequently, we have that

\begin{align*}
\mathbb{P}\left( \|R_{\alpha_{m,n}}\bar{Y}_n\m-\hat{x}\|\ge C\right)&= \mathbb{P}\left(\sqrt{\sum_{\substack{j=1\\ \alpha_{m,n}\le j^{-q}}}^{m} \frac{\left(\bar{Y}_n\m-P_m\hat{y}\right)_j^2}{j^q} + \sum_{\substack{j=1\\\alpha_{m,n}>j^{-q}}}^\infty (\hat{x},e_j)^2}\ge C\right)\\
        &\ge\mathbb{P}\left(\sqrt{\sum_{\substack{j=1\\ \alpha_{m,n}>j^{-q}}}^\infty j^{-2}} \ge C,~\alpha_{m,n}\ge \frac{q}{\pi^2}\frac{m}{n}\right)= \mathbb{P}\left(\alpha_{m,n}\ge \frac{q(\tau-1)^2}{\pi^2}\frac{m}{n}\right) \to 1
\end{align*}

as $m,n\to\infty$ with $m/n\ge c$, where we used the definition of $C$ in the third step. Thus Theorem \ref{discussion:prop} is proved in the case that $\tau>1$.

Now we discuss the case, where the right hand side $\tau \delta_{m,n}$ in line 6 of Algorithm 1 is replaced with $\sqrt{\frac{1+\sqrt{m}}{n}}$ and where $\hat{x}=K^+\hat{y}=0$. Let $\varepsilon<\frac{2}{11}$ in the definition of the density of the error distribution of the $\delta_{ij}$. Then there exists $\gamma$ with $\frac{5}{6} < \gamma < \frac{2-\varepsilon}{2+\varepsilon}$ and by Lemma \ref{discproof:prop1} there exist $p_\varepsilon>0$ such that

\begin{equation}\label{discproof:eq4}
\mathbb{P}\left(\sum_{j=m^\frac{1}{2}+1}^m\left(\frac{\sum_{i=1}^n\delta_{ij}}{n}\right)^2 \ge \frac{m-m^\frac{1}{2}}{n}+\frac{(m-m^\frac{1}{2})^\gamma}{n}\right) \ge p_\varepsilon
\end{equation}

for all $m,n$ large enough with $n\le m$. Since

$$m-m^\frac{1}{2} +(m-m^\frac{1}{2})^\gamma \ge m-m^\frac{1}{2} + \frac{1}{2^\gamma}m^\gamma = m + m^\frac{1}{2} + \frac{m^{\frac{1}{2}}}{2^\gamma}\left(m^{\frac{\gamma}{2}-1}-2^\gamma\right)>m+m^\frac{1}{2}$$

for $m$ sufficiently large there holds for $\alpha^*_{m,n}:=(m^\frac{1}{2})^{-q} = \sigma_{\sqrt{m}}^2$

$$\mathbb{P}\left(\left\lVert P_mK R_{\alpha*_{m,n}}\m \bar{Y}_n\m - \bar{Y}_n\m\right\rVert = \sum_{j=m^{\frac{1}{2}}+1}^m\left(\frac{\sum_{i=1}^n\delta_{ij}}{n}\right)^2 > \frac{m+m^\frac{1}{2}}{n} = \left(\tau_m\delta_{n,m}^{est}\right)^2\right)\ge p_\varepsilon$$

for all $m,n$ large enough with $n\le m$. By monotonicity of the spectral cut-off regularisation we deduce that

$$\mathbb{P}\left( \|P_mKR_\alpha\m\bar{Y}_n\m-\bar{Y}_n\m \|>\tau_m \delta_{m,n}^{est},~\forall \alpha_{m}^*\le \alpha\le \|K\|^2\right)\ge p_\varepsilon,$$

thus $\mathbb{P}\left( \alpha_{m,n}\le \alpha_{m}^*\right)\ge p_\varepsilon$ for all $m,n$ large enough with $n\le m$. Further, 

\begin{align*}
&\mathbb{P}\left( \|R_{\alpha_{m,n}}\m\bar{Y}_{n}\m-\hat{x}\|^2 \ge \frac{m^\frac{1+q}{2}}{n}\right) \ge \mathbb{P}\left(\sum_{j=1}^{m^\frac{1}{2}} j^q \left(\frac{\sum_{i=1}^n\delta_{ij}}{n}\right)^2 \ge \frac{m^\frac{1+q}{2}}{n},~\alpha_{m,n}\le \alpha_m^*\right)\\
\ge &\mathbb{P}\left(\sum_{j=\frac{m^\frac{1}{2}}{2}}^{m^\frac{1}{2}} j^q \left(\frac{\sum_{i=1}^n\delta_{ij}}{n}\right)^2 \ge \frac{m^\frac{1+q}{2}}{n},~\alpha_{m,n}\le \alpha_m^*\right)\\
\ge&\mathbb{P}\left(\frac{m^\frac{q}{2}}{2^q}\sum_{j=\frac{m^\frac{1}{2}}{2}}^{m^\frac{1}{2}} \left(\frac{\sum_{i=1}^n\delta_{ij}}{n}\right)^2 \ge \frac{m^\frac{1+q}{2}}{n},~\alpha_{m,n}\le \alpha_m^*\right)\\
\ge &1-\mathbb{P}\left(\alpha_{m,n}> \alpha_m^*\right)-\mathbb{P}\left(\sum_{j=\frac{m^{\frac{1}{2}}}{2}}^{m^\frac{1}{2}}\left(\frac{\sum_{i=1}^n\delta_{ij}}{n}\right)^2 < \frac{2^q m^\frac{1}{2}}{n}\right)\ge p_\varepsilon-\mathbb{P}\left(\sum_{j=\frac{m^{\frac{1}{2}}}{2}}^{m^\frac{1}{2}}\left(\frac{\sum_{i=1}^n\delta_{ij}}{n}\right)^2 < \frac{2^q m^\frac{1}{2}}{n}\right)
\end{align*}

for all $m,n$ large enough with $n\le m$. Finally, with \eqref{chap2:lem:eq0} it follows that

$$\mathbb{P}\left(\sum_{j=\frac{m^{\frac{1}{2}}}{2}}^{m^\frac{1}{2}}\left(\frac{\sum_{i=1}^n\delta_{ij}}{n}\right)^2 < \frac{2^q m^\frac{1}{2}}{n}\right)\to 0$$

as $m\to\infty$ (uniformly in $n\in\N$), which finishes the proof of Theorem \ref{discussion:prop1} since $q>1$.

\end{proof}

\begin{proof}[Proof of Proposition \ref{discussion:prop2}]
Let $\sigma_{1}\m,...,\sigma_m\m$ and $v_1\m,...,v_m\m$ denote the singular values respectively vectors of $KP_m$. Clearly, $v_j\m=v_j$ and $\sigma_j\m=j^{-1}$ for all $m\in\N$ and $j=1,...,m-1$. Moreover, the ansatz $v_m\m=a v_m + b v_{\lceil e^{m}\rceil}$  in 

$$(P_mK)^*P_mK v_m\m = {\sigma_m\m}^2 v_m\m$$

yields

$$\sigma_m\m =\sqrt{ m^{-2}e^{-2m} + \lceil e^{m}\rceil^{-2} \left(1-e^{-2m}\right)}\quad\mbox{and}\quad v_m\m= \frac{ m^{-1}e^{-m}v_m + \lceil e^{m}\rceil^{-1} \sqrt{1-e^{-2m}}v_{\lceil e^{m}\rceil}}{\sigma_m\m}.$$

Now 

\begin{align*}
\sum_{j=1}^{m-1} (\hat{x},v_j)v_j + \left(\hat{x},v_m\m\right)v_m\m &= \sum_{j=1}^m \left(\hat{x},v_j\m\right)v_j\m=P_{\mathcal{N}(P_mK)^\perp}\hat{x} = \left(\left(P_mK\right)^*P_mK\right)^{\frac{\nu_m}{2}}w_m\\
 &= \sum_{j=1}^m {\sigma_j\m}^{\nu_m} (w_m,v_j\m)v_j\m = \sum_{j=1}^{m-1} \sigma_j^{\nu_m} (w_m,v_j) v_j + {\sigma_m\m}^{\nu_m}\left(w_m,v_m\m\right)v_m\m.
\end{align*}

Consequently,

\begin{align*}
\rho_m\ge|(w_m,v_m\m)|&\stackrel{!}{=} \frac{\left(\hat{x},v_m\m\right)}{{\sigma_m\m}^{\nu_m}} = \frac{\left(\sum_{j=1}^\infty j^{-2} v_j, m^{-1}e^{-m} v_m + \lceil e^{m}\rceil^{-1} \sqrt{1-e^{-2m}} v_{\lceil e^{m}\rceil}\right) }{{\sigma_m\m}^{1+\nu}}\\
 &= \frac{ m^{-3}e^{-m} + \lceil e^{m}\rceil^{-3} \sqrt{1-e^{-2m}}}{\left( m^{-2}e^{-2m} + \lceil e^{m}\rceil^{-2}\left(1-e^{-2m}\right)\right)^\frac{1+\nu_m}{2}} \ge \frac{ m^{-3}e^{-m}}{\left(2 e^{-2m}\right)^\frac{1+\nu_m}{2}} = \frac{e^{\nu_m m}}{2^\frac{1+\nu_m}{2} m^3}
\end{align*}

and the proof of the proposition can be finished by a case-by-case analysis.

\end{proof}

\section{Numerical Demonstration}\label{sec:num}
We provide numerical experiments to complement the theoretical analysis. Three model examples, i.e. \texttt{phillips} (mildly
ill-posed, smooth), \texttt{gravity} (severely ill-posed, medium smooth)
and \texttt{shaw} (severely ill-posed, non smooth) are taken from the open source \texttt{MATLAB} package Regutools \cite{hansen1994regularization}.
The problems cover a variety of
settings, e.g. different solution smoothness and degree of ill-posedness. These examples are discretisations
of Fredholm/Volterra integral equations of the first kind by means of either the Galerkin approximation with piecewise constant basis functions or quadrature rules. We approximate our infinite-dimensional $K$ with one of the above examples with dimension $m_\infty\gg 1$. The number of measurements channels $m$ is then always chosen such that $m\ll m_\infty$. In most of the examples we use discretisation by box functions as follows, compare to Lemma \ref{fdr:box}.
With $k=m_\infty/m$ we set

\begin{align*}
 P_m &: \mathbb{R}^{m_\infty}\to\mathbb{R}^m\\
     &   \begin{pmatrix} y_{(i-1)k+1}\\ ... \\ y_{(i-1)k+k}\end{pmatrix} \mapsto \frac{1}{\sqrt{k}}\left(y_{(i-1)k+1}+...+y_{(i-1)k+k}\right)e_{i}
\end{align*}

 where $i=1,...,m$ and $e_1,...,e_m$ is the canonical basis of $\R^m$. In  Subsection \ref{sec:num:comp} we will also consider discretisation by hat functions to give an example with non-orthogonal discretisation.
We chose a shifted generalised Pareto distribution for the distribution of the measurement error, i.e. $\delta_{ij}\m= Z_{ij}\m-E Z_{ij}\m$, where $Z_{ij}\m$ are i.i.d and follow a generalised Pareto distribution (gprnd($l$,$\sigma$,$\theta$,$m$,$n$) in Matlab, with $l=1/3$, $\sigma=\sqrt{(1-l)^2(1-2l)\|\hat{y}\|}$ and $\theta=0$). This distribution is highly non-symmetric with a heavy tail. The above choices for the parameters imply that $\E {\delta_{ij}\m}^2 = \| \hat{y}\|$ and  $\E |\delta_{ij}\m|^3=\infty$. Thus the error fulfills Assumption \ref{err:disc}.1 in all the examples. The parameter $\tau$ in the definition of the discrepancy principle is set to $\tau=1.2$. All the statistical quantities are computed for $100$ independent runs, and the results are presented as box plots.

\subsection{Convergence of finite-dimensional residuum approach}
First we visualise the convergence of the discrepancy principle with the finite-dimensional\\ residuum approach, as stated in Corollary \ref{cor1}. We use discretisation by box functions as presented above and set $m_\infty=4000$ and $m=5,10,20$. For each $m$ we plot in Figure \ref{conv1} the resulting relative errors $\left\| R_{\alpha_{m,n}}\m \bar{Y}_n\m - \hat{x}\right\|/\|\hat{x}\|$
for $n=10,...,10^9$ repetitions. For $m$ fix the relative errors first decrease steadily and then saturate (at $\left\| \hat{x} -(P_mK)^+P_mK\hat{x}\right\|$) as the number of repetitions $n$ grows.
The saturation level decreases rapidly while $m$ grows, confirming the convergence of the approach. It is notable that for all examples a fairly small number of measurement channels is sufficient to yield good approximations.

\subsection{(Semi-)Convergence of infinite-dimensional residuum approach}
Now we come to the discrepancy principle with the infinite-dimensional residuum approach as stated in Corollary \ref{cor2}. Again we chose discretisation by box functions for the measurements with $m_\infty = 4000$ and this time we set $m=20,50,100$. For each $m$ we plot in the right column of Figure \ref{conv1} the resulting relative errors $\left\|R_{\alpha_m}P_m^+ \bar{Y}_{n(m,\delta_m^{disc})}\m-\hat{x}\right\|/\|\hat{x}\|$ for varying upper bound $\delta_m^{disc}$ from Assumption \ref{disc:idr}. More precisely we chose the latter in relation to the exact discretisation error $d_m:=\left\|\hat{y}-P_m^+P_m\hat{y}\right\|$. In particular we also consider $\delta_m^{disc}<d_m$ and we exhibit a semi-convergence.  Strictly speaking, the last two choices ($d_m/2$ and $d_m/4$) for $\delta_m^{disc}$ violate Assumption \ref{disc:idr} and we thus illustrate the sensitiveness to underestimation of the true discretisation error. It is notable that for the choice $\delta_m^{disc}=d_m/2$ (e.g. underestimation of the discretisation error by a factor $1/2$) the relative errors are still decreasing. This is explained by the fact that the estimation in \eqref{idr:est:disc} is quite coarse. Together with the choice $\tau=1.2$ this implies that it still holds that the true unknown error $\left\|P_m^+\bar{Y}_{n(m,\delta_m^{disc})}\m - \hat{y}\right\|$ fulfills $\left\|P_m^+\bar{Y}_{n(m,\delta_m^{disc})}\m - \hat{y}\right\|<2\tau \delta_m^{disc}$. For the choice $\delta_m^{disc}=d_m/4$ the errors then diverge.
The semi-convergence is in contrast to the saturation observed in the left column of Figure \ref{conv1} and illustrates the fundamental difference that for the finite-dimensional approach no quantitative knowledge of the discretisation error is required, while for the infinite-dimensional approach it is.

\begin{figure}
\label{conv1}
\centering
\setlength{\tabcolsep}{0pt}
\begin{tabular}{c c}
\includegraphics[width=.4\textwidth]{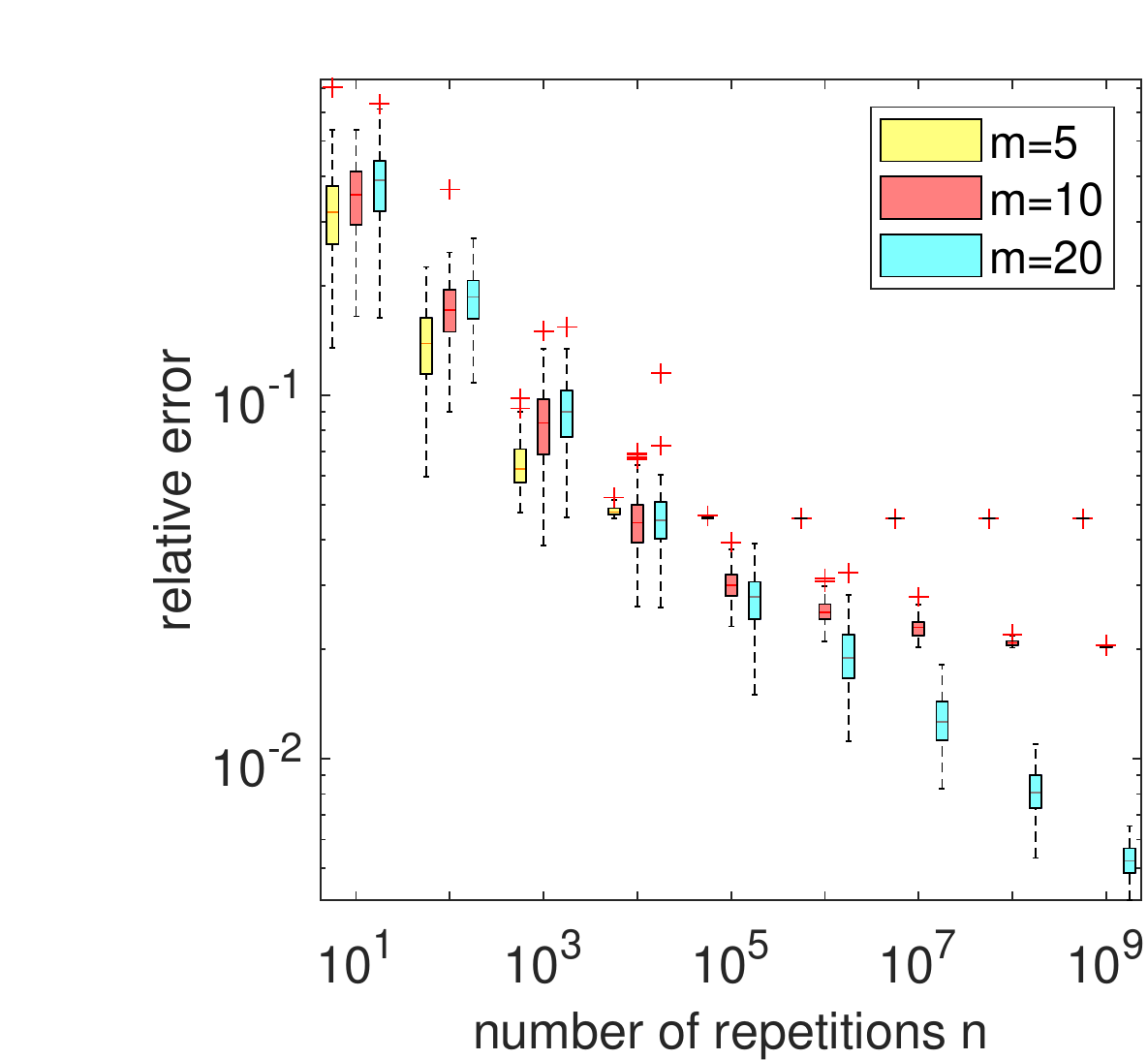} 
 & \includegraphics[width=.4\textwidth]{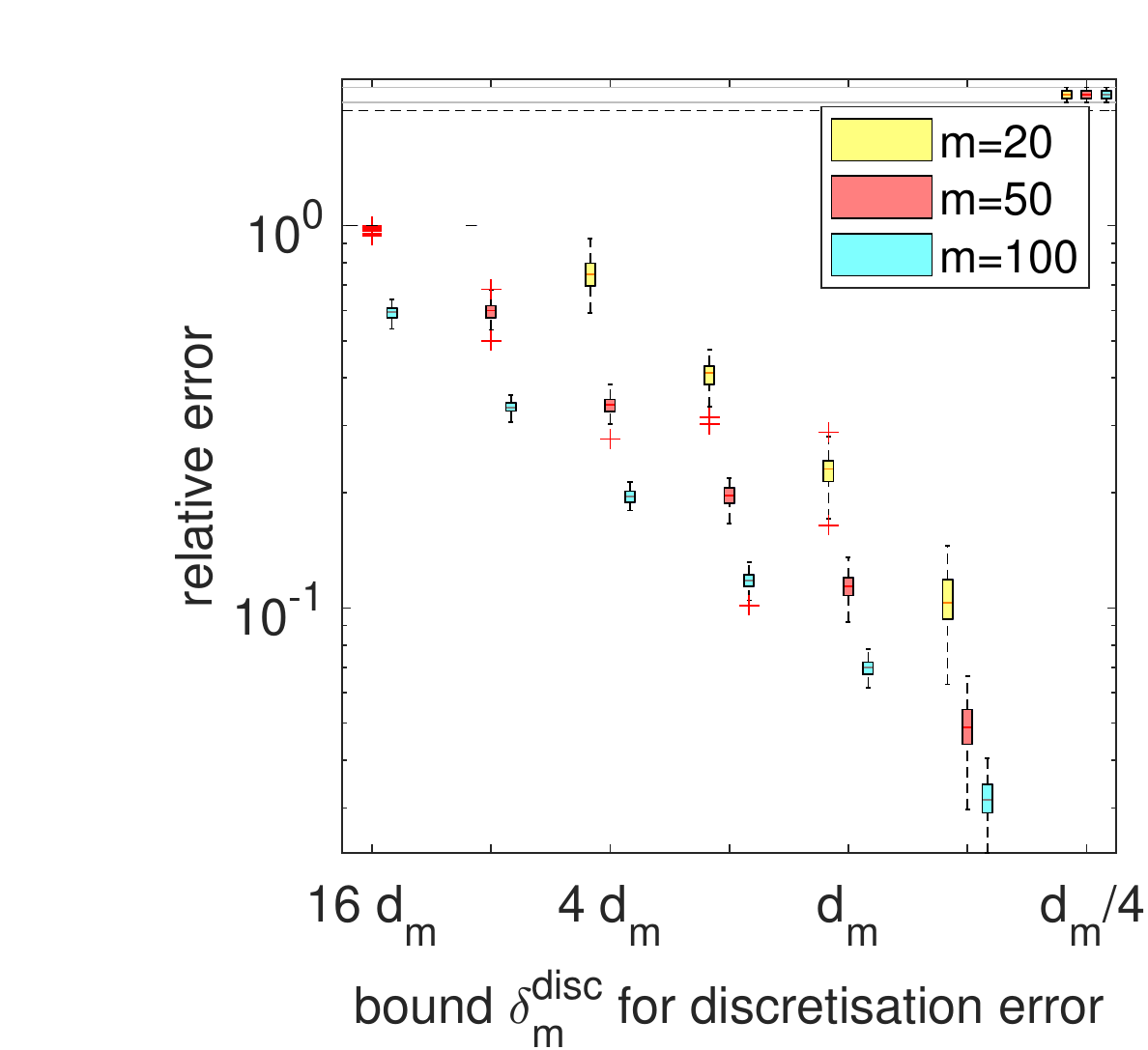}\\
\includegraphics[width=.4\textwidth]{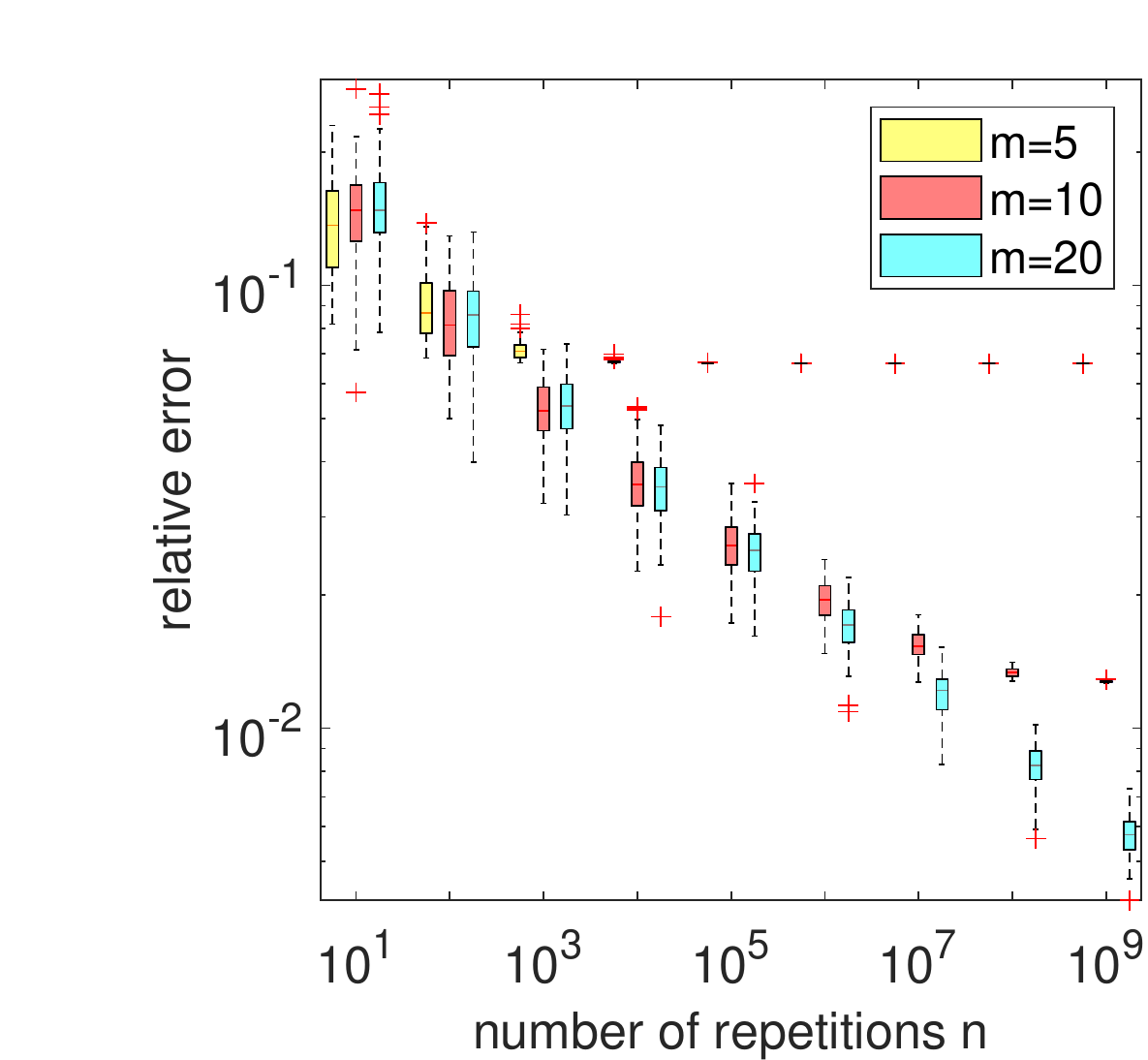} 
& \includegraphics[width=.4\textwidth]{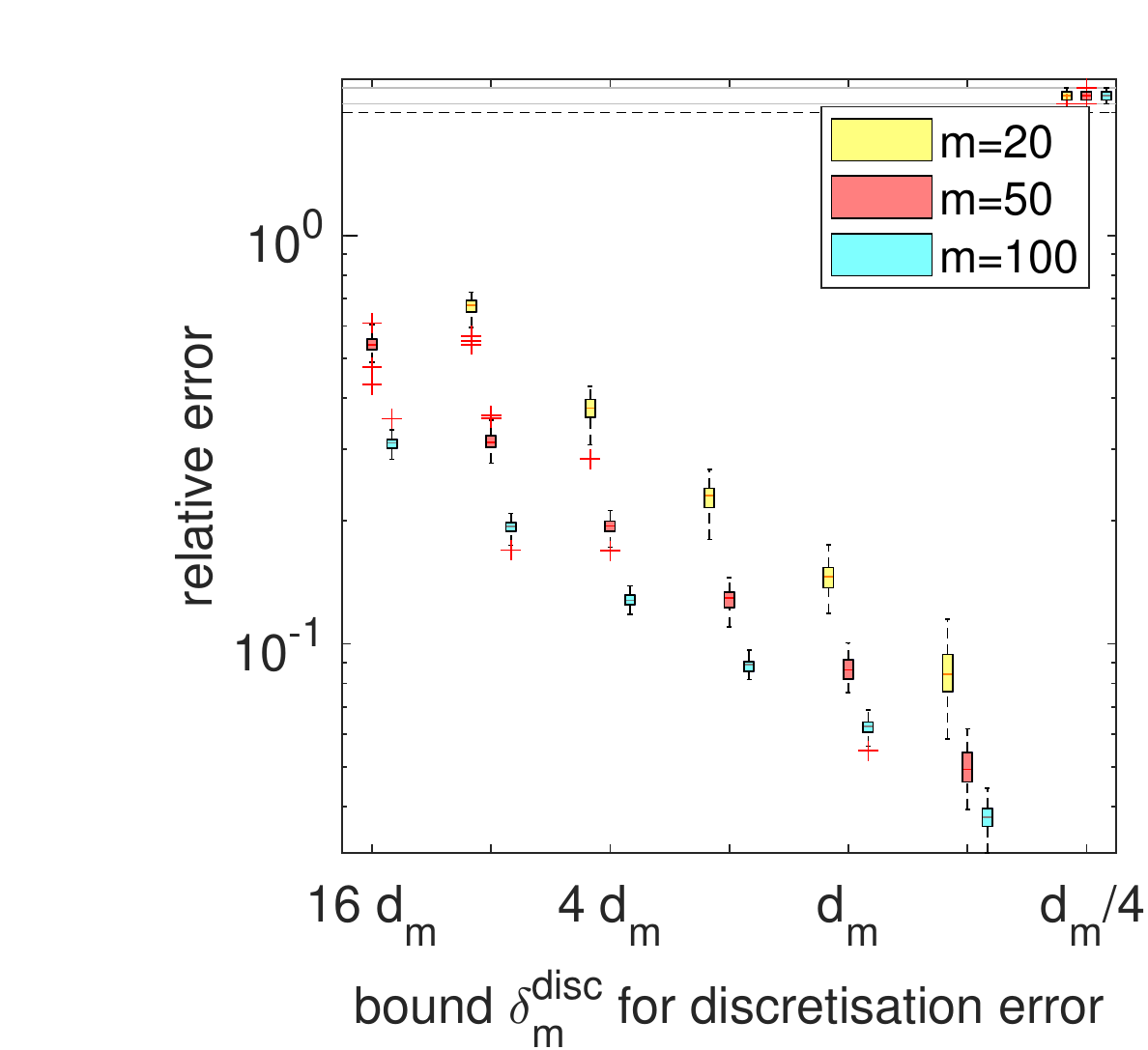} \\
\includegraphics[width=.4\textwidth,trim={-1.5cm -0 -0cm -1cm}]{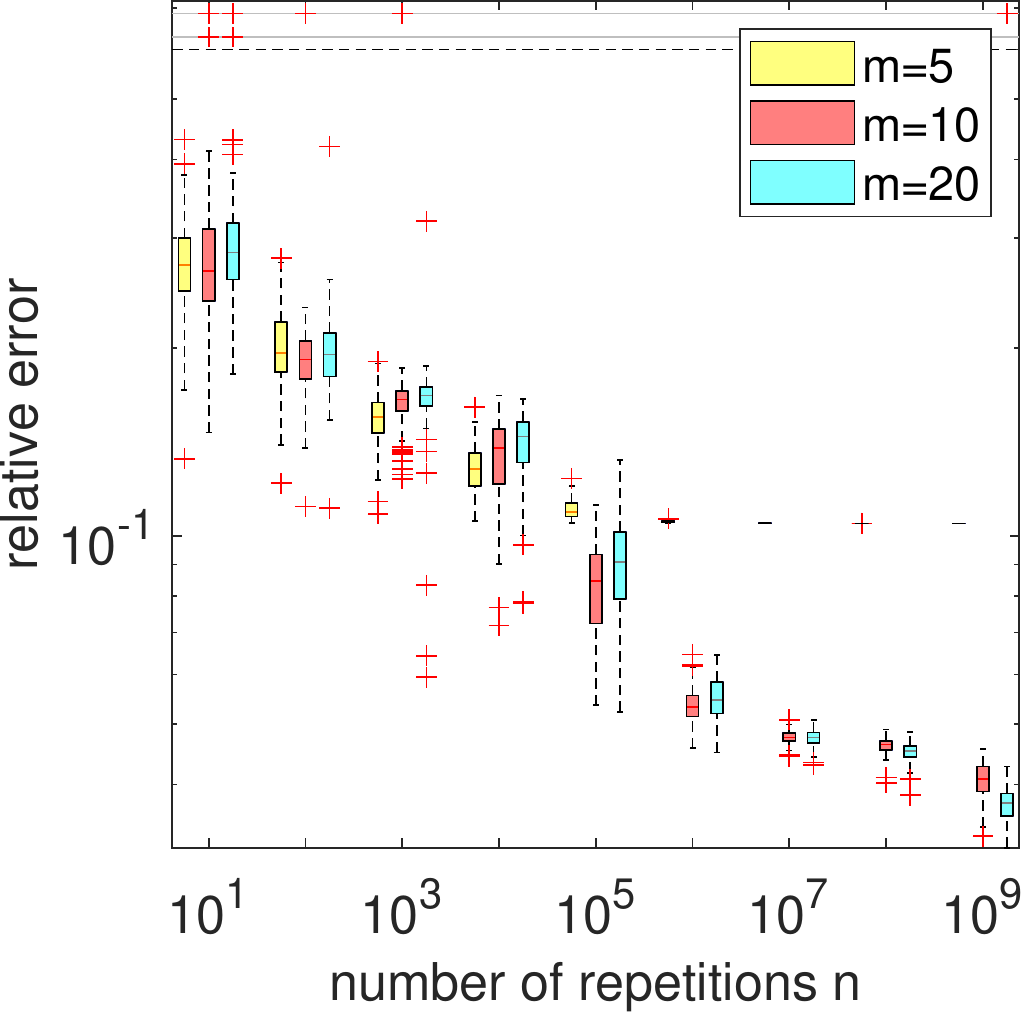}
 & \includegraphics[width=.4\textwidth]{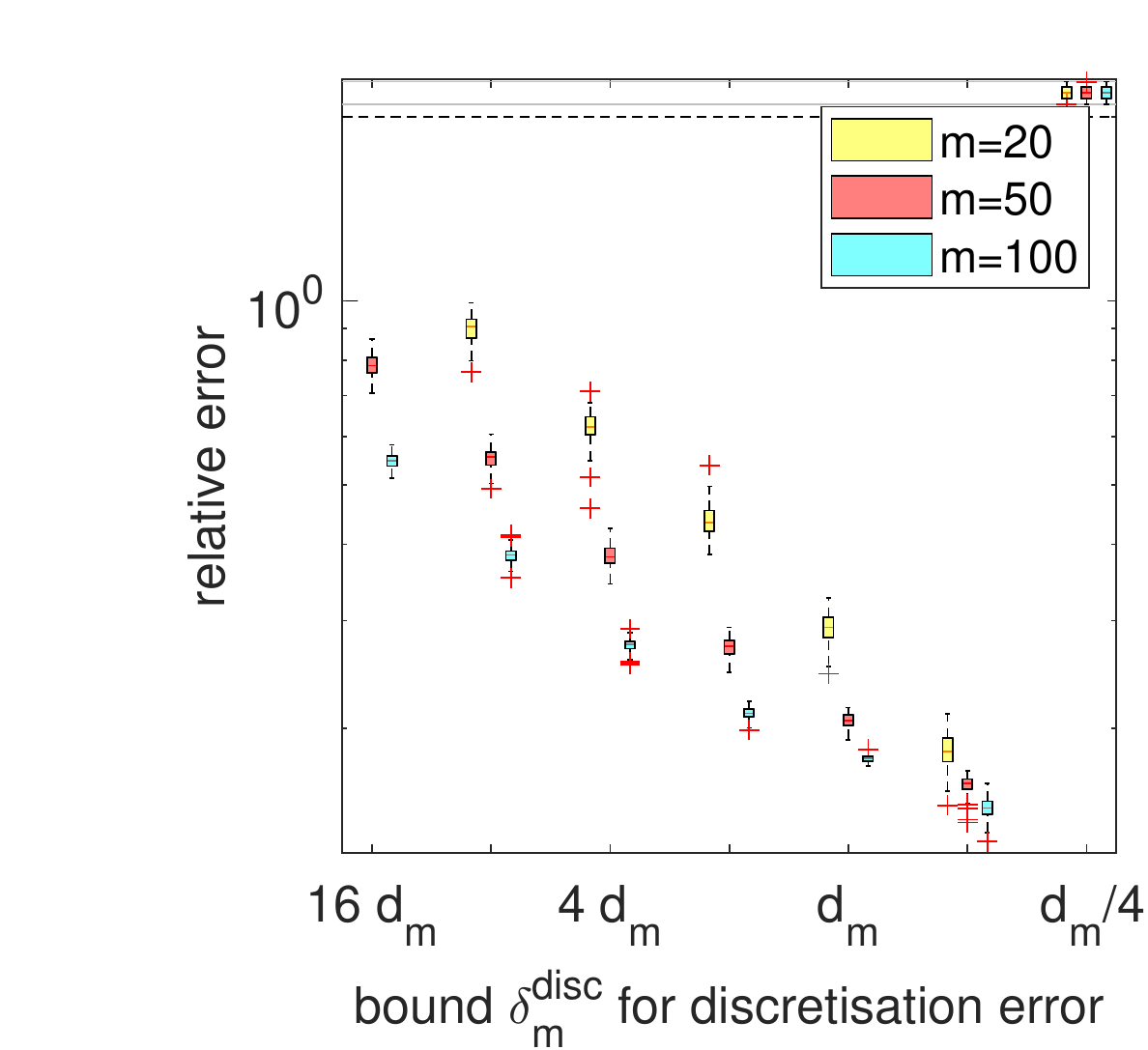}
\end{tabular}
\caption{Results of approach \eqref{tikhapprox} and \eqref{tikhapprox2} with the discrepancy principle as implemented in Algorithm 1 (left column) or 2 (right column) respectively, for 'phillips' (first row), 'gravity' (second row) and 'shaw' (third row) visualised as boxplots for $100$ independent runs. \textit{Left column}: Relative errors $\left\|R_{\alpha_{m,n}}\m \bar{Y}_n\m - \hat{x}\right\|/\|\hat{x}\|$ against number of repetitions $n$ for several numbers of measurement channels $m$. \textit{Right column}: Relative errors $\left\| R_{\alpha_m} P_m^+ \bar{Y}_{n(m,\delta_m^{disc})}\m - \hat{x}\right\|/\|\hat{x}\|$ against bound for the discretisation error $\delta_m^{disc}$ for several numbers of measurement channels $m$. $\delta_m^{disc}$ is chosen in relation to the exact discretisation error $d_m:=\left\|\hat{y}-P_m^+P_m\hat{y}\right\|$.}
\end{figure}

\subsection{Comparison of the both approaches}\label{sec:num:comp}
We now compare the both approaches directly. We consider discretisation by box functions with $m_\infty=4000$ and $m=50,100,200$ and discretisation by hat functions (compare to Proposition \ref{fdr:hat}). The latter is precisely implemented as follows. With  $k=\frac{m_\infty-1}{m-1}$ we set

\begin{align*}
 P_m &: \mathbb{R}^{m_\infty}\to\mathbb{R}^m\\
     &   \begin{pmatrix} y_{(i-1)k+1}\\ ... \\ y_{(i+1)k+1}\end{pmatrix} \mapsto \frac{1}{\sqrt{\sum_{j=1}^{2k+1}a_j^2}}\left(a_1y_{(i-1)k+1}+...+a_{2k+1}y_{(i+1)k+1}\right)e_{i}
\end{align*}

where $i=2,...,m-1$ and 
$$a_i:=\begin{cases} (i-1)/k & i\le k+1,\\
                     1-(i-k-1)/k & i \ge k+1. \end{cases}$$

For the boundaries we set,

\begin{align*}
\begin{pmatrix} y_{1}\\ ... \\ y_{k+1}\end{pmatrix} \mapsto \frac{1}{\sqrt{\sum_{i=k+1}^{k=2k+1}a_i^2}}\left(a_{k+1}y_{1}+...+a_{2k+1}y_{k+1}\right)e_{1}
\end{align*}

and

\begin{align*}
\begin{pmatrix} y_{m_\infty-(k+1)}\\ ... \\ y_{m_\infty}\end{pmatrix} \mapsto \frac{1}{\sqrt{\sum_{i=1}^{k=k+1}a_i^2}}\left(a_{1}y_{m_\infty-(k+1)}+...+a_{k+1}y_{m_\infty}\right)e_{m}.
\end{align*}

Here we use $m_\infty = 4132$ and $m=18,28,52$. We first applied Algorithm 2 with exact upper bound $\delta_m^{disc}=\left\|\hat{y}-P_m^+P_m\hat{y}\right\|$. The (random) stopping index $n(m,\delta_m^{disc})$ from Algorithm 2 is then used as the number of repetitions $n$ in Algorithm 1. We  plot in Figure \ref{comp} the relative errors of the both approaches for growing number of measurement channels $m$. We observe the stated convergence as $m$ grows. Moreover, the errors of the approach with finite-dimensional residuum are even slightly better than the ones of the approach with infinite-dimensional approach in all the examples. This indicates that here no smoothness got lost through discretisation in contrast to Proposition \ref{discussion:prop2}.
\begin{figure}
\label{comp}
\centering
\setlength{\tabcolsep}{0pt}
\begin{tabular}{c c}
\includegraphics[width=.4\textwidth,trim={-1cm -0 -0.7cm -1cm}]{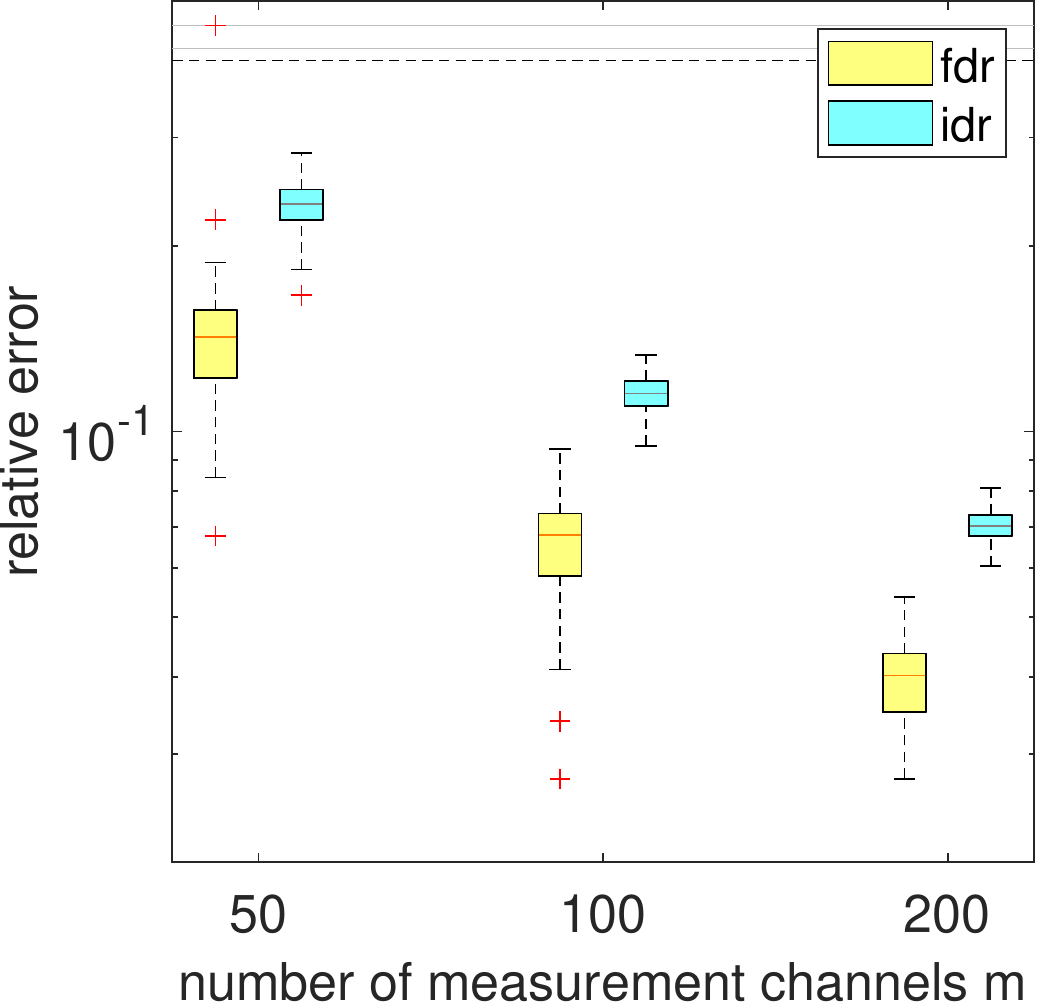} & \includegraphics[width=.4\textwidth,trim={-1cm 0 -0.7cm -1cm}]{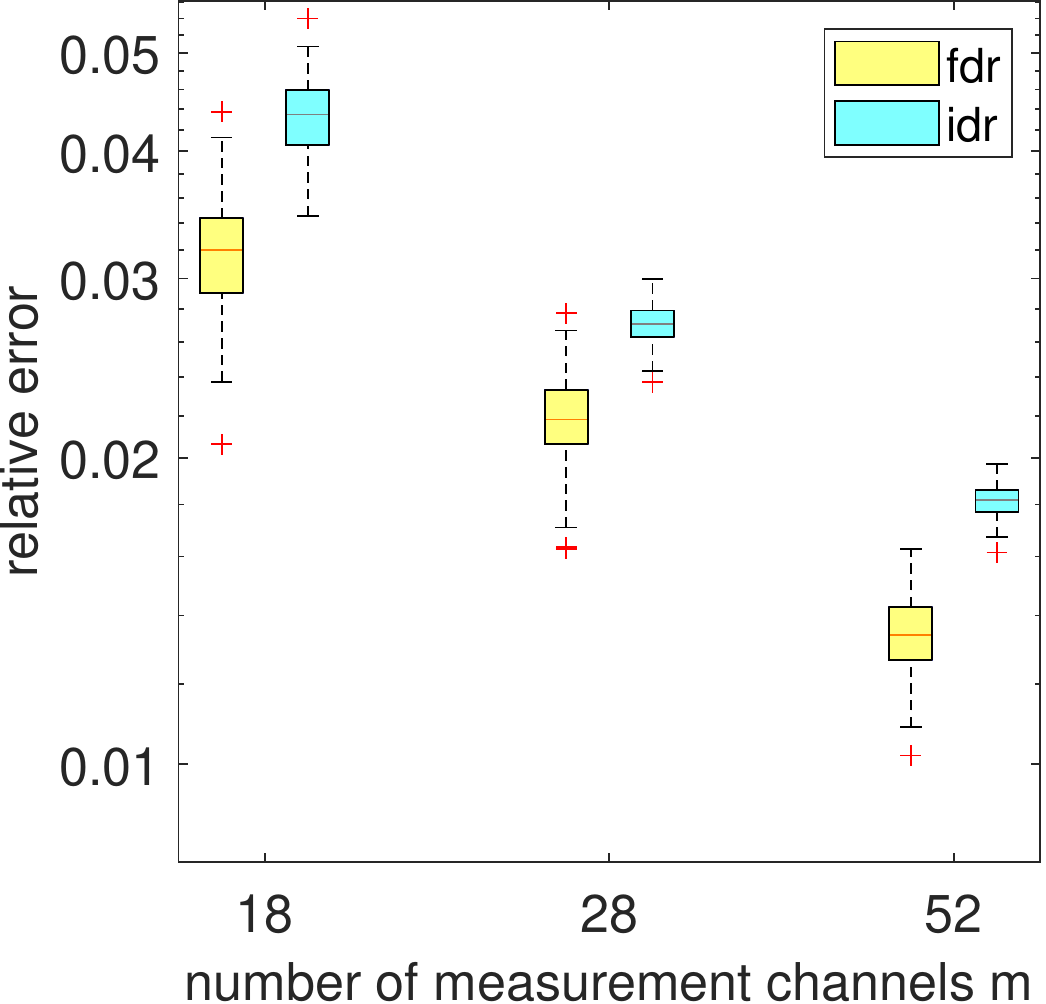}\\ 
\includegraphics[width=.4\textwidth,trim={-1cm 0 -0.7cm -1cm}]{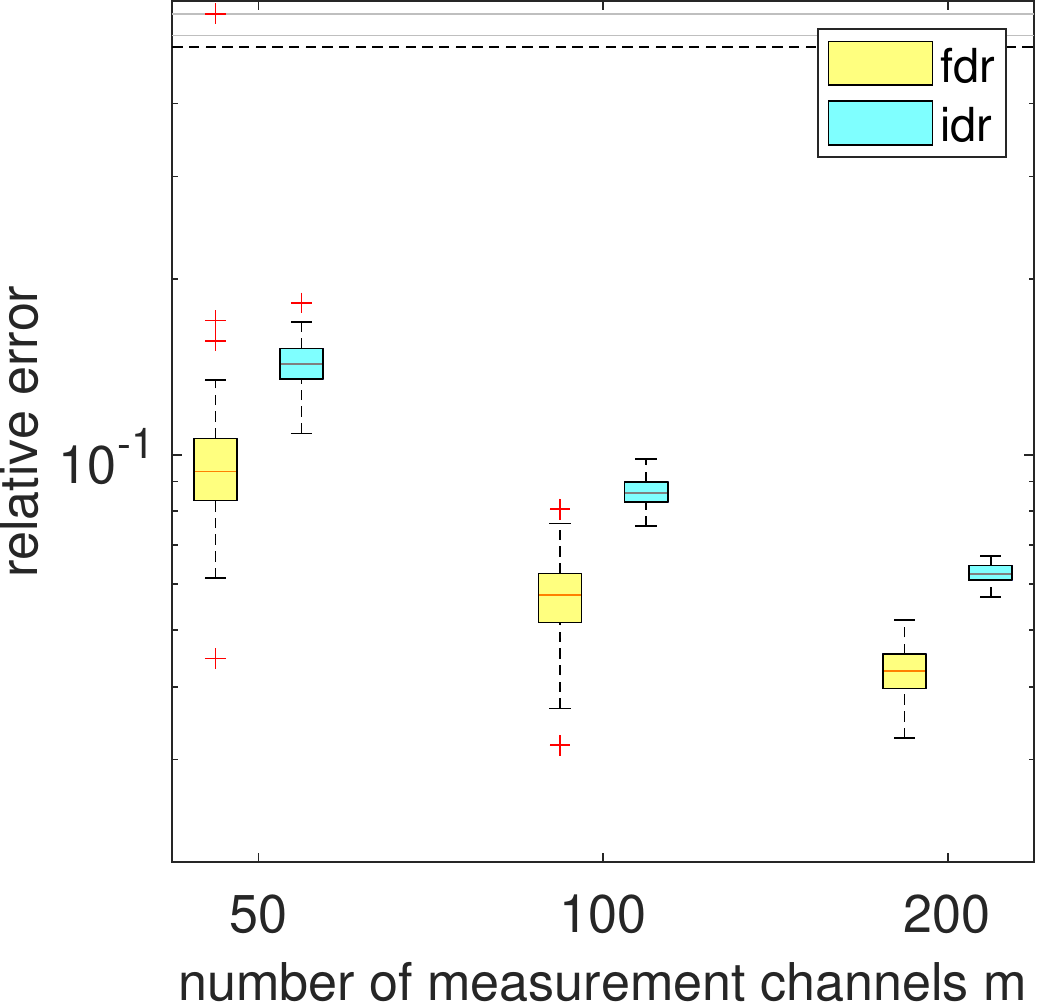} & \includegraphics[width=.4\textwidth,trim={-1cm 0 -0.7cm -1cm}]{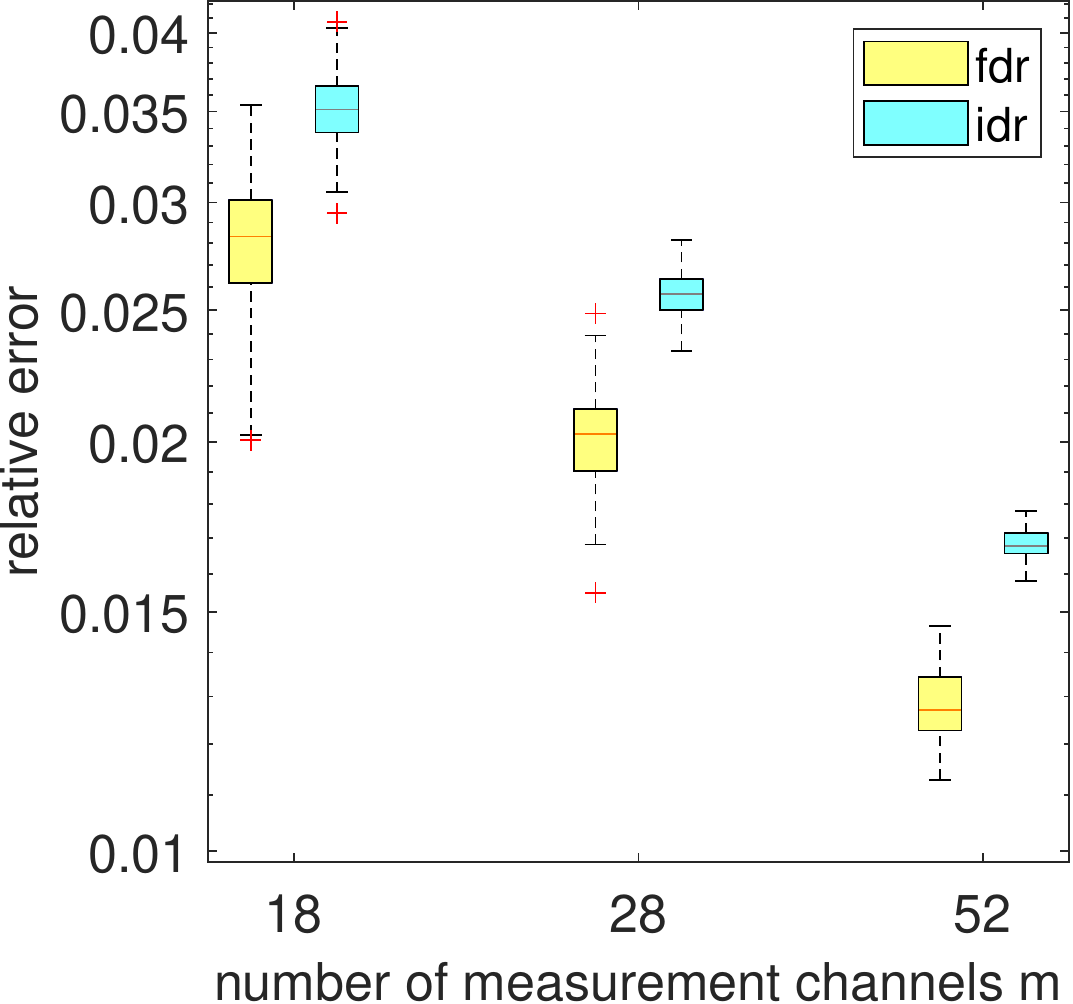}\\
\includegraphics[width=.4\textwidth,trim={-1cm 0 -0.7cm -1cm}]{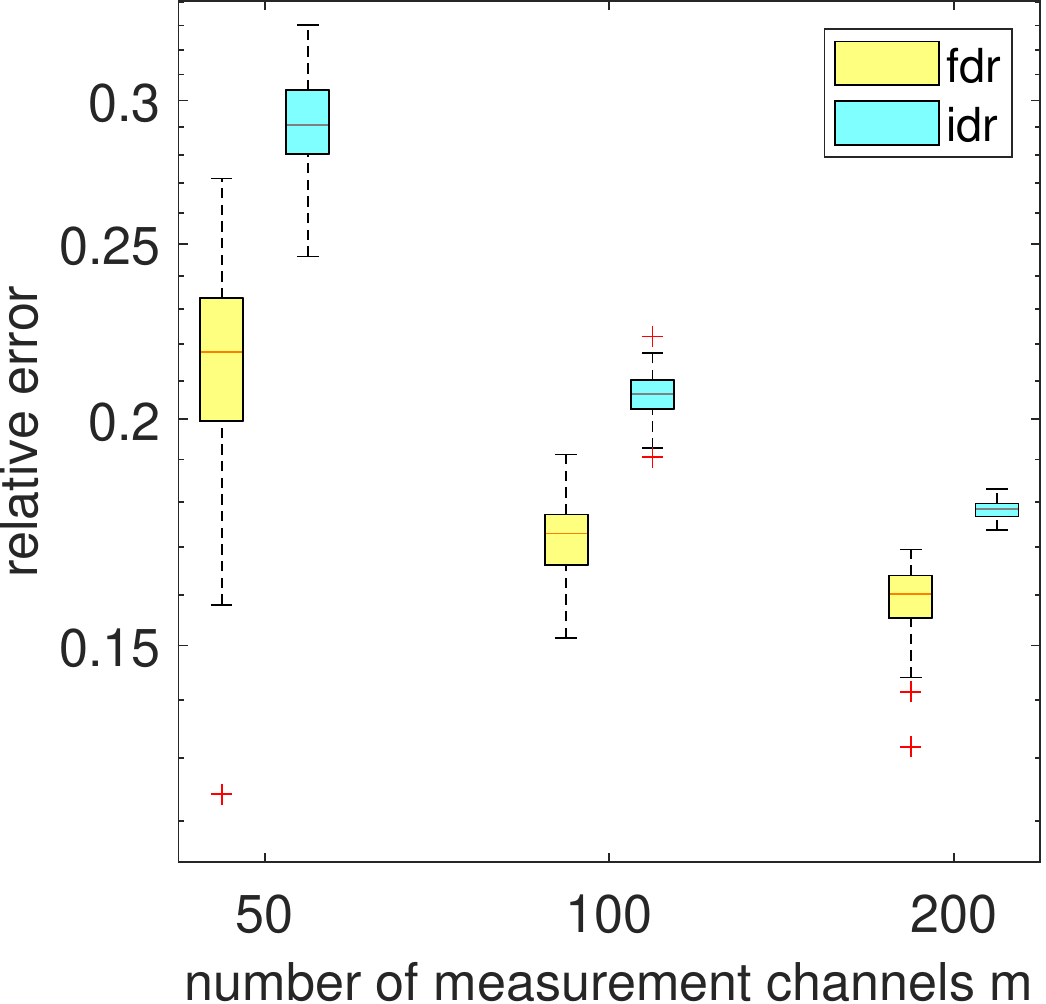} &
 \includegraphics[width=.4\textwidth,trim={-1cm 0 -0.7cm -1cm}]{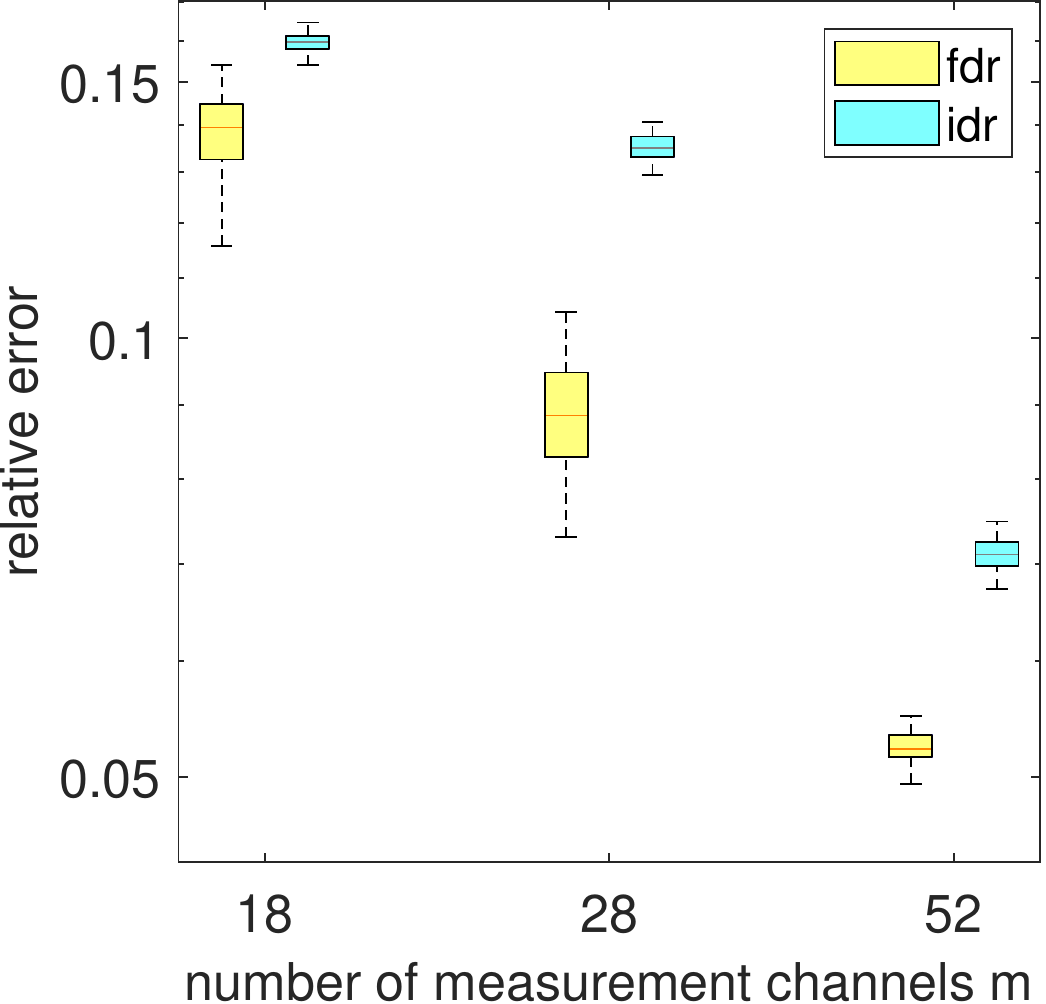}
		\end{tabular}
	\caption{Direct comparison of both approaches \eqref{tikhapprox} (fdr) and \eqref{tikhapprox2} (idr)  with discrepancy principle as implemented in Algorithm 1 and 2 for 'phillips' (first line), 'gravity' (second line) and 'shaw' (third line). For the discretisation of the measurements either box functions (first column) or hat functions (second column) are used. The relative errors $\left\|R_{\alpha_{m,n(m,\delta_m^{disc})}} \bar{Y}_{n(m,\delta_m^{disc})}\m-\hat{x}\right\|/\|\hat{x}\|$ (fdr) and $\left\|R_{\alpha_m}P_m^+\bar{Y}_{n(m,\delta_m^{disc})}\m-\hat{x}\right\|/\|\hat{x}\|$ (idr) are plotted against the number of measurement channels $m$, where $\delta_m^{disc}$ is chosen to be the exact discretisation error $\left\|\hat{y}-P_m^+P_m\hat{y}\right\|$ and $n(m,\delta_m^{disc})$ is calculated with Algorithm 2.}
\end{figure}

\section{Conclusion}\label{sec:con}
In this work we have analysed linear inverse problems under unknown white noise. We presented two approaches for the solution. In both cases we used multiple discretised measurements to prove convergence in probability against the true solution as the number of repetitions and the number of measurement channels tend to infinity. The first approach neither required knowledge of the arbitrary error distribution nor quantitative knowledge of the quality of the discretisation to obtain convergence. For the second approach we also proved a convergence rate under additional knowledge of the discretisation error.
We want to pronounce some important outstanding questions. 
First, one could drop the simplification that one has an equal number of measurements on each measurement channel and try to distribute a fixed total number of measurements on the measurement channels in an optimal way (see also \cite{mathe2017complexity}). Further, the discretisation considered in this article entered the problem through discretised measurements. In particular, this is determined by the practical problem and the way the data is measured or acquired. In order to solve the problem numerically one also has to discretise the true unknown $\hat{x}$. In contrast to the measurements here there is more freedom to choose the numerical discretisation since one is basically only limited by computational power. It therefore is of high interest to find an optimal choice for that.  Also,  it would be desirable to better understand  the interplay between the discretised and the infinite-dimensional problem, e.g. regarding the smoothness of the true solution relative to the former and the latter respectively. Hereby an important open question is to derive natural and verifiable conditions that rigorously guarantee convergence rates also for the approach with finite-dimensional residuum. Finally, it is worth investigating whether it is possible to modify the discrepancy principle to attain optimal convergence rates (in the statistical setting) in our general framework (see \cite{jahn2021optimal}).

\bibliographystyle{IMANUM-BIB}
\bibliography{references}

\end{document}